\documentclass[12pt, letterpaper]{amsart}

\usepackage{epic, eepic, amsfonts, latexsym, amssymb, graphicx,
multicol, mathrsfs, color, amscd, amsxtra, verbatim, paralist,
xspace, url, euscript, stmaryrd, newcent, amsmath, enumitem,
bbold}
\usepackage{multirow}

\usepackage[all]{xy}

\usepackage[pdfauthor={A S W},pdftitle={Ak},bookmarks=false]{hyperref}

\newlength{\baseunit} 
\newcount{\numlines} 
\newtheorem{thm}{Theorem}[section]

\newtheorem{lem}[thm]{Lemma}
\newtheorem{prop}[thm]{Proposition}
\newtheorem{cor}[thm]{Corollary}

\newtheorem{theorem}[thm]{Theorem}

\newtheorem{lemma}[thm]{Lemma}
\newtheorem{corollary}[thm]{Corollary}

\theoremstyle{definition}
\newtheorem{defn}[thm]{Definition}

\newtheorem{example}[thm]{Example}
\newtheorem{definition}[thm]{Definition}
\newtheorem{remark}[thm]{Remark}
\theoremstyle{remark}

\numberwithin{equation}{section}

  \newcommand\cC{\mathcal{C}}\newcommand\cD{\mathcal{D}}  \newcommand\cG{\mathcal{G}}\newcommand\cH{\mathcal{H}}\newcommand\cL{\mathcal{L}}\newcommand\cM{\mathcal{M}}\newcommand\cO{\mathcal{O}}\newcommand\cP{\mathcal{P}}\newcommand\cS{\mathcal{S}}\newcommand\cW{\mathcal{W}}\newcommand\cX{\mathcal{X}}\newcommand\cY{\mathcal{Y}}\newcommand\cZ{\mathcal{Z}}

\renewcommand\AA{\mathbb{A}}\newcommand\CC{\mathbb{C}}\newcommand\GG{\mathbb{G}}\newcommand\PP{\mathbb{P}}\newcommand\QQ{\mathbb{Q}}
\newcommand\ZZ{\mathbb{Z}}

  \newcommand\fD{\mathfrak{D}}  \newcommand\fZ{\mathfrak{Z}}

  \newcommand\fm{\mathfrak{m}}

\newcommand{\bk}{{\bf k}}

\newcommand{\bs}{{\bf s}}
\newcommand{\bc}{{\bf c}}

\newcommand\hookarr{\hookrightarrow}

\newcommand\im{\operatorname{im}}

\renewcommand{\setminus}{\smallsetminus}

\renewcommand{\ss}{\operatorname{ss}}

\newcommand{\Proj}{\operatorname{Proj}}

\newcommand{\Sym}{\operatorname{Sym}}

\newcommand{\oh}{\cO}

\newcommand{\Spec}{\operatorname{Spec}}

\newcommand{\tensor} {\otimes}

\renewcommand{\tilde}{\widetilde}

\newcommand{\ch}{\operatorname{char}}

\newcommand{\Spf}{\operatorname{Spf}}

\newcommand{\iso}{\stackrel{\sim}{\to}}

\renewcommand{\bar}{\overline}

\newcommand{\GL}{\operatorname{GL}}
\newcommand{\PGL}{\operatorname{PGL}}

\newcommand{\dual}{\vee}
\renewcommand{\hat}{\widehat}

\renewcommand{\AA}{{\mathbb A}}
\renewcommand{\emptyset}{\varnothing}
\renewcommand{\leq}{\leqslant}
\renewcommand{\geq}{\geqslant}
\renewcommand{\tilde}{\widetilde}
\renewcommand{\hat}{\widehat}
\renewcommand{\bar}{\overline}

\def\co{\colon\thinspace} 

\newcommand{\sar}[1][]{{\ar@<-0.5ex>@{^{(}->}[#1]}}
\newcommand{\sarl}[1][]{{\ar@<0.5ex>@{_{(}->}[#1]}}

\newcommand{\minus}{\setminus}

\newcommand{\ctdin}{\subseteq}  

\newcommand{\isom}{\cong}

\newcommand{\arr}{\rightarrow}

\newcommand{\inn}{\cap}
\newcommand{\unn}{\cup}
\newcommand{\dunn}{\sqcup}

\newcommand{\cross}{\times}

\newcommand{\eps}{{\varepsilon}}

\DeclareMathOperator{\Cr}{LT}
\DeclareMathOperator{\Tri}{Tri}
\DeclareMathOperator{\TT}{T}
\DeclareMathOperator{\Aut}{Aut}
\DeclareMathOperator{\Def}{Def}

\DeclareMathOperator{\LT}{LT}

\newcommand{\q}[1]{\{ q_i \}_{i=1}^{#1}}

\def\C{\mathcal{C}}

\def\E{\mathcal{E}}

\def\Frac{\text{Frac\,}}

\def\H{\mathcal{H}}

\def\im{\text{im\,}}

\def\pn{\{p_i\}_{i=1}^{n}}
\def\pnprime{\{p'_i\}_{i=1}^{n}}

\def\qr{\{q_i\}_{i=1}^{r}}
\def\qm{\{q_i\}_{i=1}^{m}}

\def\K{\mathcal{K}}

\def\M{\overline{M}}

\def\SM{\overline{\mathcal{M}}}

\def\O{\mathscr{O}}

\def\T{\mathcal{T}}
\def\B{\mathcal{B}}

\def\P{\mathbb{P}}
\def\Q{\mathbb{Q}}
\def\X{\mathcal{X}}

\def\U{\mathcal{U}}

\def\Y{\mathcal{Y}}

\def\sigman{\{\sigma_{i}\}_{i=1}^{n}}
\def\sigmaneta{\{\sigma_{i}(\eta)\}_{i=1}^{n}}

\def\sigmanprime{\{\sigma'_{i}\}_{i=1}^{n}}

\def\sigmanstar{\{\sigma^*_{i}\}_{i=1}^{n}}\def\taun{\{\tau_{i}\}_{i=1}^{n}}
\def\taun{\{\tau_{i}\}_{i=1}^{n}}
\def\taunprime{\{\tau'_{i}\}_{i=1}^{n}}

\def\S{\mathcal{S}}

\def\Spec{\text{\rm Spec\,}}
\def\Proj{\text{\rm Proj\,}}

\newcommand{\epf}{\qed \vspace{+10pt}}

\begin{document}

\title{Weakly proper moduli stacks of curves}

\author[Alper]{Jarod Alper}
\author[Smyth]{David Smyth}
\author[van der Wyck]{Frederick van der Wyck}

\address[Alper]{Department of Mathematics\\
Columbia University\\
2990 Broadway\\
New York, NY 10027}
\email{jarod@math.columbia.edu}

\address[Smyth]{Department of Mathematics\\
Harvard University\\
One Oxford Street\\
Cambridge, MA 01238}
\email{smyth@math.harvard.edu}

\address[van der Wyck]{Department of Mathematics\\
Harvard University\\
One Oxford Street\\
Cambridge, MA 01238}
\email{wyck@math.harvard.edu}

\begin{abstract}
This is the first in a projected series of three papers in which we construct the second flip in the log minimal model program for $\M_g$. We introduce the notion of a weakly proper algebraic stack, which may be considered as an abstract characterization of those mildly non-separated moduli
problems encountered in the context of Geometric Invariant Theory (GIT), and develop techniques for proving that a stack is weakly proper without the usual semistability analysis of GIT. We define a sequence of moduli stacks of curves involving nodes, cusps, tacnodes, and ramphoid cusps, and use the aforementioned techniques to show that these stacks are weakly proper. This will be the key ingredient in forthcoming work, in which we will prove that these moduli stacks have projective good moduli spaces which are log canonical models for $\bar{M}_g$.
\end{abstract}

\maketitle

\setcounter{tocdepth}{1}
\tableofcontents
\section{Introduction}

\noindent
In \cite{hassett_genus2}, Hassett proposed the problem of studying log canonical models of $\M_{g}$. For any $\alpha \in \Q \cap [0,1]$ such that $K_{\SM_{g}}+\alpha\Delta$ is big, Hassett and Keel define
\[
\M_{g}(\alpha):=\Proj \oplus_{m \geq0} H^0(\SM_{g}, \lfloor m(K_{\SM_{g}}+\alpha\Delta) \rfloor),
\]
and ask whether the spaces $\M_{g}(\alpha)$ admit a modular interpretation. In \cite{HH1, HH2}, Hassett and Hyeon carried out the first two steps of this program by showing that:
$$
\M_{g}(\alpha)=\begin{cases}
\M_{g} & \text{  if }\alpha \in (9/11, 1]\\
\M_{g}^{ps}&  \text{  if } \alpha \in (7/10,9/11]\\
\M_{g}^{c} & \text{  if } \alpha =7/10\\
\M_{g}^{h}&  \text{  if } \alpha \in (7/10-\epsilon, 7/10)
\end{cases}
$$
where $\bar{M}_g^{ps}$ is the moduli space of pseudostable curves (parameterizing certain curves with nodes and cusps), and $\M_{g}^{c}$ and $\M_{g}^{h}$ are the moduli spaces of bicanonical c-semistable and h-semistable curves respectively (parameterizing certain curves with nodes, cusps, and tacnodes). In \cite{HH1, HH2}, these alternate birational models of $\bar{M}_g$ are constructed using Geometric Invariant Theory (GIT). Indeed, one of the most appealing features of the Hassett-Keel program is the way it ties together the different compactifications of $M_{g}$ obtained by varying the parameters implicit in Gieseker and Mumford's classical GIT construction of $\bar{M}_g$ \cite{git, gieseker}. 

In this paper, however, we will outline a program to construct modular interpretations for the spaces $\M_{g}(\alpha)$ without GIT. The program has three steps.
\begin{enumerate}
\item Define a \emph{weakly proper} moduli stack $\SM_{g}(\alpha)$ of singular curves.
\item Construct a \emph{good moduli space} $\SM_{g}(\alpha) \rightarrow X$.
\item Show that some multiple of the $\QQ$-line bundle $K_{\SM_{g}(\alpha)}+\alpha\delta$ on $\SM_{g}(\alpha)$ descends to an ample line bundle on $X$. Use a discrepancy calculation to conclude that $X=\M_{g}(\alpha)$.
\end{enumerate}
Let us elaborate on each of these steps.
\begin{enumerate}
\item The notion of a \emph{weakly proper} algebraic stack is
introduced in Section~\ref{section-weak-properness}, and is the
key definition of this paper. Roughly speaking, weak properness is
an abstract characterization of those mildly non-separated moduli
problems encountered in the context of GIT, which nevertheless
possess a proper moduli space.  If $\bar{\cM}$ is any moduli stack of curves, we say that $\bar{\cM}$ is {\it weakly proper} if the following condition holds: given any family of curves $\cC^* \to \Delta^*$ in $\bar{\cM}$ over a punctured disc,
\begin{enumerate}
 \item We may complete  $\cC^* \to \Delta^*$ to a family $\cC \to \Delta$  in $\bar{\cM}$, possibly after a base change.
 \item If $\cC \to \Delta$ and $\cC' \to \Delta$ are two such completions whose central fibers $C$ and $C'$ are \emph{closed} in $\bar{\cM}$, there is an isomorphism $C \cong C'$. 
 \end{enumerate} 
 
\item Good moduli spaces are introduced and studied in \cite{alper_good_arxiv},
and they should be considered as an abstract version of the quotients produced by GIT. One essential difference however is that, whereas GIT quotients are automatically projective, good moduli spaces are \emph{a priori} only algebraic spaces. In forthcoming work, we will prove that weakly proper algebraic stacks satisfying certain additional hypotheses possess a good moduli space. This result may be considered as an analogue of the Keel-Mori theorem guaranteeing the existence of a coarse moduli space for separated Deligne-Mumford stacks.
\item Under mild hypotheses, it is relatively simple to understand what linear combinations of $K_{\bar{\cM}_{g}(\alpha)}$ and $\delta$ descend to the good moduli space. Thus, using Kleiman's criterion, the problem of proving that
$K_{\bar{\cM}_{g}(\alpha)}+\alpha\delta$ defines an ample divisor
class reduces to showing that a certain linear combination of
tautological classes is positive on one parameter families of curves contained in $\SM_{g}(\alpha)$. Finally, a straightforward discrepancy calculation can be used to show that sections of $K_{\bar{\cM}_{g}(\alpha)}+\alpha\delta$ lift to sections of $K_{\bar{\cM}_{g}}+\alpha\delta$. It follows formally that $\M_{g}(\alpha)$ is the good moduli space associated to $\bar{\cM}_{g}(\alpha)$.
\end{enumerate}

This is the first in a projected series of papers in which we will follow this approach to construct the second flip in the Hassett-Keel log minimal model program for $\M_{g}$. In addition, we will recover the results of Hassett and Hyeon, and extend their constructions to the case of $\M_{g,n}$ with $n>0$. In the present paper, we accomplish the first of the three steps outlined above. For $k \in \{2,3,4\}$, we define moduli stacks $\SM_{g,n}(A_k^-)$, $\SM_{g,n}(A_k)$, $\SM_{g,n}(A_k^+)$ parameterizing certain curves with $A_k$-singularities, which we call $A_k^-$-stable, $A_k$-stable, and $A_k^+$-stable curves respectively (Definition~\ref{definition-main-stacks}). Note that our notation has built-in redundancy: $A_k^+$-stability is the same as $A_{k+1}^-$-stability. Our main result is the following (proved in Corollary \ref{C:Openness} and Theorem~\ref{main-theorem}). 
\begin{theorem} \label{theorem-main}
\begin{enumerate}
\item[]
\item $\SM_{g,n}(A_k^-)$, $\SM_{g,n}(A_k)$, $\SM_{g,n}(A_k^+)$ are weakly proper algebraic stacks.
\item These stacks fit into the following diagram, where the horizontal arrows are open immersions.
\[
\def\objectstyle{\scriptstyle}
\def\labelstyle{\scriptstyle}
\begin{matrix} \SM_{g} \\  \shortparallel \\ \SM_{g,n}(A_2^-)  \end{matrix} 
\hookrightarrow
\SM_{g,n}(A_2) 
\hookleftarrow
\begin{matrix} \SM_{g,n}(A_2^+)\\ \shortparallel \\ \SM_{g,n}(A_3^-) \end{matrix} 
\hookrightarrow
\SM_{g,n}(A_3)
\hookleftarrow 
\begin{matrix} \SM_{g,n}(A_3^+) \\ \shortparallel \\ \SM_{g,n}(A_4^-)  \end{matrix}
\hookrightarrow
\SM_{g,n}(A_4)
\hookleftarrow
\SM_{g,n}(A_4^+) 
\] 
\end{enumerate}
\end{theorem}
In forthcoming work, we will complete steps two and three of the program outlined above to prove that these stacks have projective good moduli spaces. We should emphasize that there is no currently known GIT construction of the good moduli spaces of $\bar{\cM}_{g,n}(A_4)$ and $\bar{\cM}_{g,n}(A_4^+)$.  Moreover, our methods avoid a GIT-stability analysis.  Indeed, once completed, our program will provide the first examples where projective moduli spaces associated to non-separated moduli functors  are constructed without GIT (except in very special cases, e.g. where one has an explicit understanding of the global sections of a polarizing line bundle as in \cite{faltings}). Furthermore, we will have the following modular interpretation of the second flip:

{\tiny
$$\xymatrix{
\bar{\cM}_{g}(A_4^-) \ar@{^(->}[rd] \ar[dd]	&	& \bar{\cM}_{g}(A_4^+) \ar@{_(->}[ld] \ar[dd]\\
									& \bar{\cM}_{g}(A_4) \ar[dd]	\\
\bar{M}_{g}(2/3+\epsilon)\ar[rd]				& 			& \bar{M}_{g}(2/3-\epsilon)\ar[ld] \\
									& \bar{M}_{g}(2/3)	
}$$}
in which the locus of curves containing a genus 2 tail attached at a Weierstrass point is flipped to the locus of curves containing a ramphoid cusp.

\noindent
\textbf{Remark.} Our hope is that the techniques developed in this paper will be sufficient to construct $\M_{g}(\alpha)$ for all $\alpha>3/8$. Our main reason for doing just one new step of the program in the present series of papers is that whereas the first three steps are handled by a similar combinatorial formalism (i.e., we can make the definition of $A_k$-stability in a uniform way for $k \in \{2,3,4\}$), the next anticipated step (which occurs at $\alpha=12/19$) does not continue in this vein. Rather than replacing genus two bridges attached at conjugate points as one might expect from naively extending the definition of $A_k$-stability, one instead replaces arbitrary genus two tails by dangling $A_5$-singularities, i.e. curves of the form $C \cup Z$, where $Z$ is an arbitrary genus two curve meeting $C$ in a single node are replaced by curves of the form $C \cup R$, where $R$ is a smooth rational curve meeting $C$ in a single $A_5$-singularity. We refer the reader to \cite{afs} for an explanation of the heuristics behind these predictions.

Before laying out a roadmap for the rest of this paper, it may be useful to give an informal introduction to some of the key concepts in this paper, namely isotrivial specialization, weak properness, and local variation of GIT.

\subsection*{Isotrivial Specialization}
Let $\bar{\cM}$ be a moduli stack of pointed curves.  Let $(C, \{p_i\})$ and $(C', \{p'_i\})$ be two $\CC$-valued points of $\bar{\cM}$.  We say that $(C, \{p_i\})$ \emph{isotrivially specializes} to $(C', \{p'_i\})$ and write $(C, \{p_i\}) \rightsquigarrow (C', \{p'_i\})$ if any of the following equivalent conditions are satisfied:
\begin{enumerate}
\item $(C',\{p'_i\}) \in \overline{\{(C,\{p_i\}) \}}$
\item There exists a map $f \co \Delta \rightarrow \bar{\cM}$ with $f(\eta)=(C,p_i)$ and $f(0)=(C',p'_i)$, where $\Delta$ is the spectrum of a discrete valuation ring with generic point $\eta$ and closed point $0$.
\item There exists a family $(\C \rightarrow \Delta, \sigma_i)$ with $(\cC^*,\{\sigma^*_i\})  = \Delta^* \times (C,\{p_i\})$ and with special fiber $(C',\{p_i'\})$. 
\end{enumerate}
We call a family as in (3) an \emph{isotrivial specialization}. In particular, a $\CC$-valued point $(C, \{p_i\})$ is closed in $\bar{\cM}$, i.e. contains no other $\mathbb{C}$-points in its closure, iff it admits no nontrivial isotrivial specializations.

\subsection*{Weak properness}
In order to illuminate the content of the main theorem, it may be useful to verify the weak properness of a very simple moduli stack of curves. Let
$$\begin{aligned}
\SM_{1,1}(A_2):=\{(E,p) \,\mid \, &\text{$E$ is a genus one curve with nodes and cusps, } \\
	& \text{$p \in E$ smooth, $\omega_{E}(p)$ ample}\}
\end{aligned}
$$
To prove that $\SM_{1,1}(A_2)$ is weakly proper, the first step is to describe the closed points of $\SM_{1,1}(A_2)$. We claim that there is a \emph{unique} closed point of $\SM_{1,1}(A_2)$, namely the unique isomorphism class of a rational cuspidal curve with a smooth marked point. To prove this claim, it suffices to exhibit an isotrivial specialization from any one-pointed stable curve of genus one to the rational cuspidal curve. To do this, start with any stable elliptic curve $(C,p)$ and consider the trivial family $(C,p) \times \Delta$. We may modify the central fiber by blowing up the point $(p,0)$ and blowing down the strict transform of the special fiber. It is an easy exercise to check that blowing down this elliptic curve produces a cusp in the special fiber, i.e. the new special fiber is the one-pointed rational cuspidal curve as desired.

Now it is easy to prove that $\SM_{1,1}(A_2)$ is weakly proper. If $(\C^* \rightarrow \Delta^*, \sigma^*)$ is any family of smooth 1-pointed elliptic curves over the unit disc, we must show that there is a unique closed limit to this family. Evidently, we may complete the family to a stable family $(\C \rightarrow \Delta, \sigma)$ by the usual stable reduction theorem, but this central fiber is not closed in $\SM_{1,1}(A_2)$. But we may modify the central fiber of this family precisely as above, by blowing-up the point $\sigma(0)$ and then contracting the strict transform of the special fiber of $\C$ to obtain a rational cuspidal special fiber. Since this is the \emph{only} closed point of $\SM_{1,1}(A_2)$, it is evidently the unique closed limit of this family. Needless to say, in cases where the stack in question has more closed points, the verification is weak properness is much more delicate.

\subsection*{Local variation of GIT}
One of the key ingredients for proving the weak properness of the moduli stacks $\bar{\cM}_{g,n}(A_k^-), \bar{\cM}_{g,n}(A_k), \bar{\cM}_{g,n}(A_k^+)$ is Proposition \ref{theorem-etale-variation}, which asserts that \'etale locally around any closed point $[C] \in
\bar{\cM}_{g,n}(A_k)$, the open inclusions
$$\bar{\cM}_{g,n}(A_k^-) \subseteq \bar{\cM}_{g,n}(A_k) \supseteq
\bar{\cM}_{g,n}(A_k^+)$$
correspond to the open chambers
$$\hat{\Def}(C) ^- \subseteq \hat{\Def}(C) \supseteq \hat{\Def}(C) ^+$$ given by applying variation of GIT to the action of $\Aut(C)$ on the miniversal deformation space
$\hat{\Def}(C)$.  This gives a powerful tool for analyzing the local geometry of the stacks $\bar{\cM}_{g,n}(A_k^-) \subseteq \bar{\cM}_{g,n}(A_k) \supseteq
\bar{\cM}_{g,n}(A_k^+)$, and is
essential in the proof of Theorem \ref{theorem-main}.
\begin{figure}
\scalebox{1.15}{\includegraphics{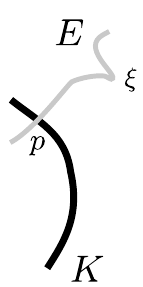}}
\caption{The curve $C=K \cup E$.}\label{F:RamphoidTail}
\end{figure}
To see how this works, let us compute the $\hat{\Def}(C) ^-/\hat{\Def}(C) ^+$-chambers for the action of $\Aut(C)$ on $\hat{\Def}(C)$, where $C$ is the union of a smooth genus $g-2$
curve $K$ with a nodally attached rational ramphoid cuspidal curve $E$
(see Figure \ref{F:RamphoidTail}).  Let $p \in C$ be the attaching node, $\xi \in E$ be the
ramphoid cusp ($y^2 = x^5$), and $\nu: \PP^1 \to E$ the
normalization of $E$ with $\nu(0) = \xi, \nu(\infty)=p$.

To understand the action of $\Aut(C)$ on $\TT^1(C)$, the space of first order deformations of $C$, recall that $\TT^1(C)$ maps surjectively onto the product of
the spaces of first order deformations of the singularities $p$ and $\xi$ with kernel given by the
space of first order locally trivial deformations $\LT^1(C)$.  Moreover, a
locally trivial deformation of $C$ induces a deformation of $(K,p)$ and the map $\LT^1(C) \to \TT^1(K, p)$ is surjective with
kernel $\LT^1(E,p)$, parameterizing how $0 \in \PP^1$ is
``crimped" to a ramphoid cusp.  Thus, we obtain a diagram
$$\xymatrix{
	& 0 \ar[d] \\
	& \LT^1(E,p) \ar[d]\\
0 \ar[r]	& \LT^1(C)\ar[d] \ar[r]	& \TT^1(C) \ar[r]	&
\TT^1(\hat{\oh}_{C,\xi}) \times \TT^1(\hat{\oh}_{C,p}) \to 0 \\
	& \TT^1(K,p) \ar[d] \\
	& 0
}$$	
where the column and row are exact sequences of
$\Aut(C)$-representations.  Let us now explicitly describe each
representation. We have
$$\begin{aligned}
\TT^1(\hat{\oh}_{C, \xi}) &= \{ y^2 = x^5 + s_3x^3 + s_2x^2 + s_1x + s_0\,:\, s_i \in \CC \} \\
\TT^1(\hat{\oh}_{C, p}) &= \{y^2 = x^2 + n\,:\, n \in \CC\}.
\end{aligned}$$
Furthermore, the first order crimping space $\LT^1(E,p)$ parameterizes subrings of the power
series ring $\CC[[z]]$ abstractly isomorphic to $\CC[[z^2, z^5]]$.  One easily sees that this space is parameterized by a parameter $c \in \CC$ so that
$$\LT^1(E,p) = \{ \CC[[(z+cz^2)^2, (z+cz^2)^5]] \subseteq \CC[[z]]\,:\, c \in \CC \}.$$
Therefore, we can write
$$\hat{\Def}(C) = \hat{\Def}(K,p) \times \Spf \CC[[\{s_i\}_{i=0}^{3}, n, c]].$$
If we fix an isomorphism $\Aut(C)^{\circ} \cong \GG_m = \Spec
\CC[t]_t$ which acts on a local coordinate $z$ around $0 \in \PP^1$ by
$z \mapsto tz$, then the action of $\Aut(C)^{\circ}$ on $\hat{\Def}(C)$ is given
by
$$ s_i \mapsto t^{-(10-2i)}s_i, \qquad n \mapsto tn, \qquad c \mapsto tc$$
and is trivial on $\hat{\Def}(K,p)$. Now, in this simple case, the chambers $\hat{\Def}(C) ^-$ and $\hat{\Def}(C) ^+$ are defined as the non-vanishing locus of functions of negative and positive weight respectively (for the general definition of local variation of GIT chambers, see Section \ref{section-local-vgit}). Thus, $\hat{\Def}(C) ^- \subseteq \hat{\Def}(C) \supseteq
\hat{\Def}(C) ^+$ are defined by the closed loci:
$$
\hat{\Def}(C)  \setminus \hat{\Def}(C) ^- = V(\{s_i\}_{i=0}^{3}) \qquad \text{ and } \qquad
\hat{\Def}(C)  \setminus \hat{\Def}(C) ^+  = V(n,c).
$$
The locus $V(\{s_i\}_{i=0}^{3})$ precisely corresponds to deformations in which the ramphoid cusp
singularity $\xi \in C$ is preserved, and the locus $V(n,c)$ corresponds to deformations preserving the node $p \in C$ \emph{and} with the node being a Weierstrass point of the genus 2 curve $E$. Since $\bar{\cM}_g(A_k) \setminus \bar{\cM}_g(A_k^-)$ is the locus of curves with a ramphoid
cusp, and $\bar{\cM}_g(A_k) \setminus \bar{\cM}_g(A_k^+)$
corresponds to the locus
of curves containing a genus 2 tail attached at a Weierstrass
point, we see that the GIT chambers do cut out the inclusions $\bar{\cM}_{g}(A_k^-) \subseteq \bar{\cM}_{g}(A_k) \supseteq
\bar{\cM}_{g}(A_k^+)$ in a neighborhood of $C$. The bulk of this
paper (Sections 5-7) is devoted to generalizing this description of the local structure of the inclusions $\bar{\cM}_{g}(A_k^-) \subseteq \bar{\cM}_{g}(A_k) \supseteq \bar{\cM}_{g}(A_k^+)$ to arbitrary closed points of $\bar{\cM}_g(A_k)$, where combinatorial considerations require a fairly extensive case-by-case analysis.

\subsection*{Roadmap}  In Section \ref{section-weak-properness}, we introduce the key definition of this paper, namely weakly proper morphisms, and systematically develop their properties.  In particular, we show that if an algebraic stack $\cX$ admits a good moduli space $\cX \to Y$, then under mild hypotheses, $\cX$ is weakly proper if and only if $Y$ is proper (Proposition \ref{prop-good-weakly-separated}).

In Section \ref{section-sh}, we introduce the algebraic stacks $\bar{\cH}_{m,1}$ (resp., $\bar{\cH}_{m,2}$) parameterizing genus $m$ hyperelliptic curves with a marked Weierstrass point (resp., genus $m$ hyperelliptic curves with two marked Weierstrass conjugate points).  The algebraic stack $\bar{\cH}_{m,1}$ (resp., $\bar{\cH}_{m,2}$) is a birational model of the variety of ``stable tails'' associated to an $A_{2m}$-singularity (resp., $A_{2m+1}$-singularity).  We also introduce the algebraic stacks $\bar{\cS}_{m,1}$ (resp., $\bar{\cS}_{m,2}$) parameterizing genus $m$ curves obtained by imposing an $A_{2m}$-singularity on a smooth rational curve (resp., an $A_{2m+1}$-singularity on two smooth rational curves.  The algebraic stack $\bar{\cS}_{m,1}$ (resp., $\bar{\cS}_{m,2}$) is a birational model of the variety of ``crimping deformations'' associated to an $A_{2m}$-singularity (resp., $A_{2m+1}$-singularity). We give explicit quotient presentations for the stacks  $\bar{\cH}_{m,1}$, $\bar{\cH}_{m,2}$, $\bar{\cS}_{m,1}$, $\bar{\cS}_{m,2}$ (Proposition \ref{h-quotient-stack} and Proposition \ref{s-quotient-stack}) from which we obtain the key fact that will be used repeatedly throughout the sequel, namely that any $H_{m,1}$ or $S_{m,1}$-curve (resp. $H_{m,2}$ or $S_{m,2}$-curve) admits an isotrivial degeneration to the \emph{monomial $H_{m,1}$-curve} (resp., \emph{monomial $H_{m,2}$-curve}), which is the curve obtained by imposing a monomial $A_{2m}$-singularity at $0 \in \PP^1$ and marking $\infty$ (resp.,  imposing a monomial $A_{2m+1}$-singularity at the two origins in  $\PP^1$ with marked points at the $\infty$'s); see Definition \ref{definition-monomial}.


In Section \ref{section-stability}, we define the moduli stacks $\bar{\cM}_{g,n}(A_k^-)/\bar{\cM}_{g,n}(A_k^+)/\bar{\cM}_{g,n}(A_k^+)$ parameterizing
 $A_k^-/A_k/A_k^+$-stable curves respectively for $k \in \{2,3,4\}$; see Definition \ref{definition-stability}. 
 The definitions are inductive so that $A_k^+$-stability is equivalent to $A_{k+1}^-$-stability.  The main result of this section is Proposition \ref{P:Openness} which states that these 
 stability conditions are deformation open, from which it follows immediately that  $\bar{\cM}_{g,n}(A_k^-)$, $\bar{\cM}_{g,n}(A_k)$, $\bar{\cM}_{g,n}(A_k^+)$ are algebraic stacks.  The proof is surprisingly subtle, due to the intricate combinatorics of the loci being added and removed. Note that in this section and all future sections, we always assume $k \in \{2,3,4\}$.


In Section \ref{section-closed-points}, we give a geometric characterization of the closed points of $\bar{\cM}_{g,n}(A_k)$. We show that for $k$ even (resp., $k$ odd), any curve $C$ in $\bar{\cM}_{g,n}(A_k)$ has a 
canonical decomposition $C=K \cup E_1 \cup \cdots \cup E_r$ where the ``core" $K$ is 
an $A_k$-stable curve containing no $H_{m,1}$-tails and 
$E_1, \ldots, E_r$ are the nodally attached $H_{m,1}$-tails of $C$
(resp., $C=K \cup E_1 \cup \cdots \cup E_r$ where the ``core" $K$
is an $A_k$-stable curve containing no $H_{m,2}$-chains and $E_1, \ldots, E_r$ are the nodally attached $H_{m,2}$-links of $C$).  See Definitions
\ref{canonical-decomposition-even} and
\ref{canonical-decomposition-odd} for precise details.
 The main result of this section is Proposition \ref{P:ClosedPoints} which states that an $A_k$-stable curve $(C,\{p_i\})$ is a closed point of $\bar{\cM}_{g,n}(A_k)$ if and only if the core $K$ is a closed point in $\bar{\cM}_{h,m}(A_k^-)$ and every nodally attached $H_{m,1}$-tail/$H_{m,2}$-bridge is monomial.

In Section \ref{section-deformation-theory}, we give an explicit 
description, around any closed point  $[(C, \{p_i\})] \in \bar{\cM}_{g,n}(A_k)$, of the action of the automorphism group $\Aut(C, \{p_i\})$ on the deformation
space of a marked curve $(C, \{p_i\})$. This makes essential use of the geometric characterization of closed points in Section \ref{section-closed-points}, as well as the description of the deformation theory of $H_{m,1}/H_{m,2}$-curves and $S_{m,1}/S_{m,2}$-curves contained in Section \ref{section-sh}.
First, we describe the action of $\Aut(C, \{p_i\})$ on the
first order deformation space $\TT^1(C, \{p_i\})$
(see Propositions \ref{first-order-action-even} and \ref{first-order-action-odd}).  
We then describe how one can choose geometric coordinates for the 
miniversal deformation space $\hat{\Def}(C, \{p_i\})$ which diagonalize the 
action of $\Aut(C, \{p_i\})$ and have the property that the vanishing of certain subsets of these coordinates cuts out the closed loci $\bar{\cS}_{g,n}(A_k) := 
\bar{\cM}_{g,n}(A_k) \setminus \bar{\cM}_{g,n}(A_k^-)$ 
and $\bar{\cH}_{g,n}(A_k) := \bar{\cM}_{g,n}(A_k) \setminus \bar{\cM}_{g,n}(A_k^+)$ 
(see Propositions \ref{formal-action-even} and \ref{formal-action-odd}).  

In Section \ref{section-local-vgit}, we calculate the plus/minus-chambers obtained by variation of GIT for the action of the automorphism group $\Aut(C, \{p_i\})$ on the deformation space $\Def(C, $ $\{p_i\})$ of a closed point $[C, \{p_i\}] \in \bar{\cM}_{g,n}(A_k)$.  The main result is Proposition \ref{theorem-etale-variation}, which asserts that the inclusions  $\bar{\cM}_{g,n}(A_k^{-}) \subseteq \bar{\cM}_{g,n}(A_k) \supseteq  \bar{\cM}_{g,n}(A_k^{+})$ correspond \'etale locally to the plus/minus open loci obtaining by the variation of GIT.

Section \ref{section-proof} proves the main result of this paper,
namely that the stacks $\SM_{g,n}(A_{k}^-)$, $\SM_{g,n}(A_{k})$
and $\SM_{g,n}(A_{k}^+)$ are weakly proper.  Since
$\bar{\cM}_{g,n} = \bar{\cM}_{g,n}(A_2^-)$ is proper, it suffices
to show: (1) $\SM_{g,n}(A_k^-)$ weakly proper $\implies$
$\SM_{g,n}(A_k)$ weakly proper, and (2) $\SM_{g,n}(A_k)$ weakly
proper $\implies$ $\SM_{g,n}(A_k^+)$ weakly proper.  For (1), given a
family $\C^* \rightarrow \Delta^*$ of smooth curves, one obtains
the unique $A_k$-stable limit as follows: first, take the unique
$A_k^-$-stable limit, then degenerate all $H_{m,1}/H_{m,2}$-curves
to monomial $H_{m,1}/H_{m,2}$-curves. Since this procedure is
canonical, the uniqueness of $A_k$-stable limits essentially follows from the uniqueness of $A_k^-$-stable limits.  For (2), we do not give an explicit construction of the limiting process but instead deduce it from the \'etale local description in Proposition \ref{theorem-etale-variation} of the inclusion $\bar{\cM}_{g,n}(A_k^+) \subset \bar{\cM}_{g,n}(A_k) $ as the variation of GIT locus $\Def(C, \{p_i\})^+ \subseteq \Def(C, \{p_i\})$ using a purely formal diagram chase.


\subsection*{Conventions}
The symbol $\CC$ will denote a fixed algebraically closed field of
characteristic zero.  All schemes, algebraic spaces and algebraic
stacks are assumed to be quasi-separated.  In Section \ref{section-weak-properness} we work over an arbitrary base scheme but in all later sections we work over $\Spec \CC$.

 A \emph{curve}
is a connected reduced proper $\CC$-scheme of dimension one. An
\emph{$n$-pointed curve} $(C, \{p_i\}_{i=1}^{n})$ is a curve $C$
with $n$ distinct smooth marked points $p_i \in C$.
We use the
notation $\Delta = \Spec R$ and $\Delta^* = \Spec K$, where $R$ is
a discrete valuation ring with fraction field $K$; we set $0$, $\eta$
and $\bar{\eta}$ to be the closed point, the generic point and the
geometric generic point respectively of $\Delta$.  If $\cX$ is an algebraic stack, $|\cX|$ will denote the topological space of equivalence classes of field valued points. If $x, y \in \X$ are two $\CC$-points of a stack $\X$, we sometimes say that $x$ admits an {\it isotrivial specialization} to y and write $x \rightsquigarrow y$ to indicate that $y \in \overline{\{ x\}}$. 

Whenever $\X_{g,n}$ is any moduli space of $n$-pointed curves of genus $g$, we always implicitly assume $g \geq 3$ or else $g \geq 1$ and $n \geq 1$.

\subsection*{Acknowledgments} We thank Brendan Hassett for his
enthusiastic support of this project as well as for many useful
conversations and suggestions.  We also thank Maksym Fedorchuk,
Joe Harris, David Hyeon and Johan de Jong for interesting conversations.


\section{Weak properness}\label{section-weak-properness}
Let $G$ be a reductive group acting on a projective scheme $X$ with an ample $G$-lin\-earization $\cL$.  Consider the quotient stack $\cX = [X^{\ss}_{\cL} / G]$ where $X^{\ss}_{\cL}$ is the semistable locus.  If there exists strictly semi-stable points in $X$, then $\cX$ is necessarily non-separated (and therefore not proper) as a strictly semi-stable point $x \in X$ has a non-finite affine stabilizer $G_x$.  However, there is a good moduli space $[X^{\ss} / G] \to X^{\ss}/G:= \Proj \bigoplus_{d \ge 0} \Gamma(X, \cL^{d})^G$ where the GIT quotient $X^{\ss}//G$ is projective.  For the definition of a good moduli space and a proof of this fact, see \cite{alper_good_arxiv}.

The GIT quotient stack $[X^{\ss} / G]$ satisfies the well-known semistable replacement property:  if $0 \in \Delta$ is a pointed disc and $f \co \Delta^* \to [X^{\ss}/G]$ is a morphism from the punctured disc, then there exists a covering $\Delta' \to \Delta$ branched over $0$ and a morphism $\Delta' \to [X^{\ss} /G]$ extending $f$.  In other words, $[X^{\ss}/G]$ satisfies the valuative criterion for universally closedness.  Moreover, one can always choose an extension $\tilde f \co \Delta' \to [X^{\ss} / G]$ such that $\tilde f (0')$ is a closed point.  Furthermore, the closed point $\tilde f(0')$ is unique; as $X^{\ss} //G$ is proper there exists a unique extension $h \co \Delta' \to X^{\ss}//G$ and $\tilde f(0')$ corresponds to the unique closed point in the stack $[X^{\ss}/G]$ (i.e., the unique closed orbit in $X^{\ss}$) over $h(0')$.

In this section, we introduce a notion of weakly separatedness (resp., weak properness) 
modeled on the semistable replacement property of GIT quotient stacks
such that:
\begin{enumerate}
\item the property can be checked in practice on moduli problems,
\item given the existence of a good moduli space, the property is equivalent to the separation (resp., properness) of the good moduli space, and 
\item the property is useful for establishing the existence of a good moduli space.  
\end{enumerate}

In Definition \ref{definition-weakly-proper} and Lemma \ref{lemma-specialization}, we give two different formulations of the condition of weak properness.  Propositions \ref{lemma-technical-uc} through \ref{prop_weakly_separated_k} are proved in order to show that one may check the valuative criterion for weak properness on DVRs of the form $k[[x]]$ or complex analytic discs (when working over the complex numbers).   Finally, in Proposition \ref{prop_weakly_separated_k}, we prove that an algebraic stack with a locally separated good moduli space is weakly proper if and only if the good moduli space is proper.


\begin{defn}  
\label{definition-weakly-proper}
Let $f \co \cX \to \cY$ be a morphism of algebraic stacks.
\begin{enumerate}
\item We say that $f$ is \emph{weakly separated} if for every valuation ring $R$ with fraction field $K$, and 2-commutative diagrams
\begin{equation}
\label{diagram_weakly_separated}
\xymatrix{
 \Delta^* = \Spec K \ar[r] \ar[d]	& \cX \ar[d]^f \\
\Delta = \Spec R \ar[r] \ar@<.5ex>[ur]^{h_1} \ar@<-.5ex>[ur]_{h_2}		& \cY
}
\end{equation}
such that $h_1(0)$ and $h_2(0)$ are closed in $|\cX \times_{\cY} \Delta|$, then $h_1(0) = h_2(0) \in |\cX \times_{\cY} \Delta|$.
\item We say that $f$ is \emph{weakly proper} if $f$ is weakly separated, finite type and universally closed.
\end{enumerate}
\end{defn}


\begin{remark} Note that in (1) we are only requiring that the images of the points $h_1(0)$ and $h_2(0)$ agree as points in the topological space $|\cX \times_{\cY} \Delta|$.
If one was to require in (1) above the existence of an isomorphism $h_1 \iso h_2$ (not necessarily extending the given isomorphism $h_1|_{\Spec K} \iso h_2|_{\Spec K}$), the resulting definition would be too strong. Indeed, even GIT quotient stacks $[X^{\ss} / G]$ with projective good moduli spaces will not necessarily satisfy this stronger property.  For instance, consider the quotient stack $[ (\PP^1)^{4, \ss} / \PGL_2]$ with the symmetric linearization.  Consider the two families $(0, [t^2, 1], [1, t^2], \infty)$ and $(0, [t^3, 1], [1,t], \infty)$ over $\Spec \CC[[t]]$.  Over the generic fiber $\Spec \CC((t))$, the families are isomorphic by $x \mapsto tx$.  The central fibers both correspond to the closed point $(0,0, \infty, \infty)$ but there is no isomorphism of the families over $\Spec \CC[[t]]$.
\end{remark}

\medskip \noindent
We now give a different interpretation of this definition.  Let $\Delta$ be the spectrum of a valuation ring $R$ with fraction field $K$.  Consider a 2-commutative diagram $\fD$:
\begin{equation}
\label{diagram}
\xymatrix{
\Delta^* \ar[r] \ar[d]	& \cX \ar[d]^f \\
 \Delta \ar[r]		& \cY
}
\end{equation}
\begin{definition} We define \emph{the set of extensions of $\cD$} to be the set $\Sigma_{\fD} \subseteq | \cX \times_{\cY} \Delta |$ consisting of points $x$ such that there exists an extension $K'$ of $K$ and a valuation ring $R'$ for $K'$ dominating $R$ with an extension of $\fD$ to a 2-commutative diagram $\fD'$
\begin{equation}
\label{diagram-extension}
\xymatrix{
\Delta'^* = \Spec K'	\ar[r] \ar[d]			& \Delta^* = \Spec K \ar[r] \ar[d]	& \cX \ar[d]^f \\
 \Delta' = \Spec R' \ar[r] \ar[rru]^{h'}		& \Delta = \Spec R \ar[r]		& \cY
}
\end{equation} 
such that $x = h'(0')$.  
\end{definition}

\begin{lem}  
\label{lemma-specialization}
Let $f \co \cX \to \cY$ be a morphism of algebraic stacks.
Let $\fD$ be a commutative diagram as in Diagram (\ref{diagram}).  Let $\xi \in |\cX \times_{\cY} \Delta|$ be the image of $\Delta^* \to \cX \times_{\cY} \Delta$.  Let $x \in |\cX \times_{\cY} \Delta|$.  Then the following are equivalent:
\begin{enumerate} 
\item $x\in \Sigma_{\fD}$.
\item There is a specialization $\xi \rightsquigarrow x$ in $|\cX \times_{\cY} \Delta|$ over $\eta \rightsquigarrow 0$.
\item There exists an extension to a diagram $\fD'$ as in Diagram (\ref{diagram-extension}) where $K \hookarr K'$ is a finite type, separable extension and $x = h'(0')$.
\end{enumerate}
In particular, the topological space $\Sigma_{\fD}$ is closed under specialization.  
\end{lem}

\begin{proof} It is clear that $(3) \implies  (1) \implies (2)$.  The direction
$(2) \implies (3)$ follows from \cite[Prop. 7.2.2]{lmb}.
\end{proof}

\begin{remark}  Suppose in addition that $f$ is locally of finite presentation.  If $\xi \rightsquigarrow x$ is a specialization in $|\cX \times_{\cY} \Delta|$ over $\eta \rightsquigarrow 0$ with $x \in |\cX \times_{\cY} \Delta|$ closed, then there exists an extension to a diagram $\fD'$ as in Diagram (\ref{diagram-extension}) where $K \hookarr K'$ is a finite, separable extension and $x = h'(0')$.  Indeed, if $U \to \cX$ is any smooth presentation and $u \rightsquigarrow u_0$ is a specialization over $\xi \rightsquigarrow x$ with $u_0 \in U$ closed, then one may slice $U$ to obtain a morphism $U' \to \cX$ and a specialization $u' \rightsquigarrow u'_0$ over $\xi \rightsquigarrow x$ where the induced maps on residue fields $K \to k(u')$ is finite and separable.  
\end{remark}

\begin{lem} \label{lemma-technical-uc}
 Let $f \co \cX \to \cY$ be a quasi-compact morphism of algebraic stacks.  If $f \co \cX \to \cY$ is not universally closed, there exists a morphism $Y' \to \cY$ which is locally of finite presentation such that $\cX_{Y'} \to Y'$ is not closed.
\end{lem}

\begin{proof}  We may assume that $\cY$ is an affine scheme.  Let $g \co Y \to \cY$ be a morphism such that $\cX_Y \to Y$ is not closed.  There exists a diagram
$$\xymatrix{
x \ar@{|->}[d]	&		& \cX_Y \ar[d] \\
y \ar@{~>}[r]	& y_0	& Y
}$$
where there does not exist a specialization $x \rightsquigarrow x_0$ over $y \rightsquigarrow y_0$.  If $\cZ = \overline{ \{x\} } \subseteq \cX_Y$, then $\cZ \cap \cX_{y_0} = \emptyset$.  Let $p \co U \to \cX$ be a smooth presentation where $U$ is an affine scheme.  Denote $p_Y \co U_Y \to \cX_Y$ and $Z = p_Y^{-1}(\cZ) \subseteq U_Y$.  Note that $Z \cap U_{y_0} = \emptyset$.  Then  \cite[Tag \href{http://math.columbia.edu/algebraic_geometry/stacks-git/locate.php?tag=05BD}{05BD}]{stacks_project}
implies that after replacing $Y$ with an open subscheme containing $y_0$, there exists a factorization $g \co Y \xrightarrow{a} Y' \xrightarrow{g'} \cY$ and a closed subscheme $Z' \subseteq U_{Y'}$ such that (1) $Y' \to \cY$ is locally of finite presentation, (2) $Z' \cap U_{a(x_0)} = \emptyset$ and (3)  $\im (Z \to U_{Y'}) \subseteq Z'$.  Let $x'$ be the image of $x$ under $\cX_Y \to \cX_{Y'}$.  Then $p_{Y'}^{-1}( \overline{ \{x'\} } ) \subseteq Z'$ where $p_{Y'} \co U_{Y'} \to \cX_{Y'}$ 
so that $\overline{ \{x' \} } \cap \cX_{a(x_0)} = \emptyset$.  Therefore, $\cX_{Y'} \to Y'$ is not closed.
\end{proof}

\begin{lem}  Let $f \co \cX \to \cY$ be a morphism of algebraic stacks.  Then
\begin{enumerate}
\item $f$ is weakly separated if and only if for all diagrams $\fD$ as in Diagram (\ref{diagram}), the set of extensions $\Sigma_{\fD}$ has at most one closed point.
\item If $f$ is quasi-compact, then $f$ is universally closed if and only if for all diagrams $\fD$ as in Diagram (\ref{diagram}), the set of extensions $\Sigma_{\fD}$ is non-empty.
\item If $f$ is finite type, then $f$ is weakly proper if and only if for all diagrams $\fD$ as in Diagram (\ref{diagram}), the set of extensions $\Sigma_{\fD}$ has a unique closed point.
\end{enumerate}
Furthermore, in $(1)$, $(2)$ and $(3)$ we may restrict to diagrams $\fD$ as in Diagram (\ref{diagram}) where $R$ is a complete valuation ring with algebraically closed residue field.  If $f \co \cX \to \cY$ is a locally of finite type morphism with $\cY$ locally noetherian, then in $(1)$, $(2)$ and $(3)$ we may restrict to diagrams $\fD$ as in Diagram (\ref{diagram}) where $R$ is a complete discrete valuation ring with algebraically closed residue field. 
\end{lem}

\begin{proof}
For $(1)$, the ``if'' direction is clear.  Conversely, suppose there is a commutative diagram $\fD$ as in Diagram (\ref{diagram}) and two extensions as in Diagram (\ref{diagram-extension})
$$\xymatrix{
\Delta_i^* = \Spec K_i	\ar[r] \ar[d]			& \Delta^* \ar[r] \ar[d]	& \cX \ar[d]^f \\
 \Delta_i  = \Spec R_i \ar[r] \ar@{-->}[rru]^{h_i}		& \Delta \ar[r]		& \cY
}$$
for $i = 1,2$ with both $h_1(0_1)$ and $h_2(0_2)$ distinct closed points of $|\cX \times_{\cY} \Delta|$, where $0_i \in \Spec R_i$ is the closed point. There is a field extension $K \hookarr K'$ containing both $K_1$ and $K_2$ and a valuation ring $R' \subseteq K'$ dominating both $R_1$ and $R_2$ giving a 2-commutative diagram 
$$
\xymatrix{
\Delta'^* =  \Spec K' \ar[r] \ar[d]	& \cX \ar[d]^f \\
\Delta' = \Spec R' \ar[r] \ar@<.5ex>[ur]^{h'_1} \ar@<-.5ex>[ur]_{h'_2}		& \cY
}$$
with two lifts $h'_1$ and $h'_2$ such that $h'_1(0')$ and $h'_2(0')$ are closed and distinct.  This contradicts $f$ being weakly separated.  

Statements $(2)$ and $(3)$ follow from \cite[Theorem 7.3]{lmb}.  The refinements for the valuative criterion for $(2)$ follow from \cite[7.2.1-7.2.3, 7.10]{lmb} and Lemma \ref{lemma-technical-uc}.

We now check that the condition of weakly separatedness for morphisms (resp., locally of finite type morphisms of locally noetherian algebraic stacks) can be tested on complete valuation rings (resp., complete discrete valuation rings) with algebraically closed residue field.  Suppose $f$ is not weakly separated so that there exists a commutative diagram $\cD$ with $\Delta$ a valuation ring
$$\xymatrix{
 \Delta^* \ar[r] \ar[d]	& \cX \ar[d]^f \\
 \Delta \ar[r]  \ar@<.5ex>[ur]^{h_1} \ar@<-.5ex>[ur]_{h_2}		& \cY
}$$
and two extensions $h_1, h_2 \co \Delta \to \cX$ where $h_1(0), h_2(0) \in |\cX \times_{\cY} \Delta|$ are closed and distinct.  Let $\xi$ be the image of $\Delta^* \to \cX \times_Y \Delta$.  The extensions $h_1$ and $h_2$ give specializations $\xi \rightsquigarrow h_1(0) $ and $\xi \rightsquigarrow h_2(0)$ over $\eta \rightsquigarrow 0$. 
By applying \cite[Prop. 7.2.1]{lmb}(resp., \cite[Prop. 7.2.2]{lmb}), there exists a complete valuation ring (resp., complete discrete valuation ring) with algebraically closed residue field $R'$ and morphisms $h'_1, h'_2 \co \Spec R' \to \cX \times_{\cY} \Delta$ such that $h'_i(0') = h_i(0)$ where $0' \in \Spec R'$ denotes the closed point.  This gives the desired 2-commutative diagram
$$\xymatrix{
 \Spec K' \ar[r] \ar[d]	& \cX \ar[d]^f \\
 \Spec R' \ar[r] \ar@<.5ex>[ur]^{h'_1} \ar@<-.5ex>[ur]_{h'_2}		& \cY
}$$
with $h_1'(0')$ and $h_2'(0')$ closed in $|\cX \times_{\cY} \Spec R'|$.
\end{proof}

\begin{lem}  \label{lemma-ws}
Let $\cX$ be an algebraic stack finite type over an algebraically closed field $k$.  Then $\cX$ is weakly separated over $k$ if and only for every field extension $k \to k'$ with $k'$ algebraically closed, the morphism $\cX \times_k k' \to \im (\Delta_{\cX/k}) \times_k k'$ surjects onto closed points, where $\im(\Delta_{\cX/k})$ denotes the scheme-theoretic image of $\Delta_{\cX/k}: \cX \to \cX \times_k \cX$.
\end{lem}

\begin{proof}
\medskip \noindent
First suppose $\cX$ is weakly separated over $k$.  Let $k \to k'$ be a field extension and set $\cX' = \cX \times_k k'$.  Let $(x_1,x_2) \in |\im(\Delta_{\cX'/k'})|$ be a closed point.  Then $x_1, x_2 \in |\cX'|$ are closed points. There is a valuation ring $R$ with fraction field $K$, residue field $k'$ and a diagram
$$\xymatrix{
\Delta^* = \Spec K \ar[r] \ar[d]		& \Delta = \Spec R \ar[d]^{h_1,h_2} \\
\cX'  \ar[r]^{\Delta_{\cX'/k'}}	& \cX' \times_{k'} \cX' 
}$$
such that $h_i(0) = x_i$. 
Then we have a diagram
$$
\xymatrix{
 \Spec K \ar[r] \ar[d]	& \cX' \ar[d]^f \\
 \Spec R \ar[r] \ar@<.5ex>[ur]^{h_1} \ar@<-.5ex>[ur]_{h_2}		& \Spec k'
}$$
Since $\cX$ is weakly separated over $k$, $h_1(0) = h_2(0)$ so that $(x_1,x_2)$ is in the set-theoretic image of $\cX' \to \im \Delta_{\cX'/k'}$.

\medskip \noindent
Conversely, if $\cX$ is not weakly separated over $\Spec k$, then there is diagram
$$
\xymatrix{
 \Delta^* = \Spec K \ar[r] \ar[d]	& \cX \ar[d]^f \\
 \Delta = \Spec R \ar[r] \ar@<.5ex>[ur]^{h_1} \ar@<-.5ex>[ur]_{h_2}		& \Spec k
}$$
where $R$ is a valuation ring with fraction field $K$ and residue field $\kappa$  such that 
$h_1(0)$ and $h_2(0) \in |\cX \times_k \Delta|$ closed and distinct.  But then $(h_1(0), h_2(0))$ is a closed point in $|\im \Delta_{\cX \times_k \kappa /\kappa}|$ which is not in the image of $|\cX \times_k \kappa|$. 
\end{proof}

\begin{prop}  \label{prop-ws}
Let $\cX$ be an algebraic stack finite type over an algebraically closed field $k$.  Then $\cX$ is weakly separated over $k$ if and only if the morphism $\cX \to \im (\Delta_{\cX/k})$ surjects onto closed points.
\end{prop}

\begin{proof}  By Lemma \ref{lemma-ws}, it suffices to show that if $\cX \to \im (\Delta_{\cX/k})$ surjects onto closed points, then for every extension $k \to k'$ with $k'$ algebraically closed, the morphism $\cX \times_k k' \to \im (\Delta_{\cX/k}) \times_k k'$ surjects onto closed points.  Let $\cX' = \cX \times_k k'$ and consider the cartesian diagram
$$\xymatrix{
\cX' \ar[r] \ar[d]		& \im(\Delta_{\cX'/k'}) \ar@{^(->}[r] \ar[d]	& \cX' \times_{k'} \cX' \ar[d] \\
\cX \ar[r]			& \im(\Delta_{\cX/k}) \ar@{^(->}[r]		& \cX \times_k \cX
}$$	
Suppose $(z'_1, z'_2) \in \im(\Delta_{\cX'/k'})$ is a closed point with image $(x'_1, x'_2) \in \im(\Delta_{\cX/k})$.  Since every closed point $(x_1, x_2) \in |\im(\Delta_{\cX/k})|$ which is a specialization of $(x'_1, x'_2)$ is in the image of $\cX \to \im(\Delta_{\cX/k})$, so is $(x'_1, x'_2)$.  It follows that $(x'_1, x'_2)$ is in the image of $\cX' \to \im \Delta_{\cX'/k'}$.  
\end{proof}

\begin{prop} \label{prop_weakly_separated_k}
Let $\cX$ be an algebraic stack finite type over an algebraically closed field $k$.  Then is $\cX \to \Spec k$ weakly separated if and only if for all diagrams
$$
\xymatrix{
 \Delta^* = \Spec k((x)) \ar[r] \ar[d]	& \cX \ar[d]^f \\
 \Delta = \Spec k[[x]] \ar[r] \ar@<.5ex>[ur]^{h_1} \ar@<-.5ex>[ur]_{h_2}		& \Spec k
}$$
such that $h_1(0)$ and $h_2(0)$ are closed in $|\cX|$, then $h_1(0) = h_2(0)$.
\end{prop}

\begin{proof} By Proposition \ref{prop-ws}, it suffices to show that if the valuative criterion holds for DVRs of the form $k[[x]]$, then $\cX \to \im(\Delta_{\cX/k})$ surjects on closed points.  As in the argument of Lemma \ref{lemma-ws}, let $(x_1,x_2) \in |\im(\Delta_{\cX/k})|$ be a closed point and choose a diagram
$$\xymatrix{
\Delta^* = \Spec k((x)) \ar[r] \ar[d]		& \Delta = \Spec k[[x]] \ar[d]^{h_1,h_2} \\
\cX  \ar[r]^{\Delta_{\cX/k}}	& \cX \times_{k} \cX
}$$
with $h_i(0) = x_i$. 
By applying the restricted valuative criterion, we see that $x_1 = x_2$.
\end{proof}

\begin{remark}  \label{remark-disc} It follows from Proposition \ref{prop_weakly_separated_k} and Artin approximation that an algebraic stack $\cX$ finite type over an algebraically closed field $k$ is weakly separated if and only if for every smooth curve $C \to \Spec k$ with a closed point $0 \in C$ and diagram
$$
\xymatrix{
C \setminus \{ 0 \} \ar[r] \ar[d]	& \cX \ar[d]^f \\
C \ar[r] \ar@<.5ex>[ur]^{h_1} \ar@<-.5ex>[ur]_{h_2}		& \Spec k
}$$
such that $h_1(0)$ and $h_2(0)$ are closed in $|\cX|$, then $h_1(0) = h_2(0)$.  
\end{remark}

\begin{prop} 
\label{proposition-weakly-separated-properties} \quad
\begin{enumerate}
\item  If $f \co \cX \to \cY$ is a locally separated and representable morphism of algebraic stacks, then $f$ is separated if and only if $f$ is weakly separated.  In particular, a morphism of schemes is separated if and only if it is weakly separated.
\item Weakly separated (resp., weakly proper) morphisms are stable under base change.
\item Weakly separated (resp., weakly proper) morphisms satisfy descent in the fpqc topology.  
\end{enumerate}
\end{prop} 

\begin{proof}
For (1), it is clear that if $f$ is separated, then $f$ is weakly separated.  Conversely, let $f \co \cX \to \cY$ be a locally separated and weakly separated morphism of algebraic stacks.  Let $Y \to \cY$ be a smooth presentation with $Y$ a scheme.  Then $f' \co \cX \times_{\cY} Y \to Y$ is a locally separated and weakly separated morphism of algebraic spaces.  If $f'$ is separated, then by descent $f$ is separated.  So we can assume that $f \co X \to Y$ is a locally separated and weakly separated morphism of algebraic spaces.  Suppose we have a commutative diagram
$$
\xymatrix{
 \Delta^* \ar[r] \ar[d]	& X \ar[d]^f \\
 \Delta \ar[r] \ar@<.5ex>[ur]^{h_1} \ar@<-.5ex>[ur]_{h_2}		& Y
}$$
with two lifts $h_1$ and $h_2$.  This gives a commutative diagram
$$\xymatrix{
\Delta^* \ar[r]\ar[d]		& X \ar[d] \\
\Delta \ar[r]^{(h_1,h_2)}			& X \times_Y X
}$$
Since $X \to X \times_Y X$ is an immersion, $X \times_{X \times_Y X} \Delta \to \Delta$ is an immersion.  Since $f$ is weakly separated, $h_1(0) = h_2(0)$ in $X \times_Y \Delta$ and it follows that $X \times_{X \times_Y X} \Delta \to \Delta$ is an isomorphism.  Therefore $h_1 = h_2$ and $f$ is separated.  Part $(2)$ is clear.

For part $(3)$, suppose $f \co \cX \to \cY$ and $g \co \cY' \to \cY$ are morphism of algebraic stacks with $g$ faithfully flat and quasi-compact such that the base change $\cX' := \cX \times_{\cY} \cY' \to \cY'$ is weakly separated.  Suppose we have a 2-commutative diagram as in Diagram (\ref{diagram})
$$\xymatrix{
 \Delta^* = \Spec K \ar[r] \ar[d]	& \cX \ar[d]^f \\
\Delta = \Spec R \ar[r] \ar@<.5ex>[ur]^{h_1} \ar@<-.5ex>[ur]_{h_2}		& \cY
}
$$
with two lifts $h_1, h_2 \co \Delta \to  \cX$ such that $h_1(0)$ and $h_2(0)$ are closed in $|\cX \times_{\cY} \Delta|$.   There exists an extension $K'$ of $K$ and a valuation ring $R'$ for $K'$ dominating $R$ and a morphism $\Delta'=\Spec R' \to \cY'$ extending $\Delta \to \cY$.  There exists unique extensions $h_1',h_2' \co \Delta' \to \cX'$ with a 2-isomorphism $h_1'|_{\Delta'^*} \iso h_2'|_{\Delta'^*}$.  Moreover, $h_1'(0')$ and $h_2'(0')$ are closed points in $|\cX' \times_{\cY'} \Delta'|$.  It follows that $h_1'(0') = h_2'(0')$ and $h_1(0) = h_2(0)$.  Therefore, $f$ is weakly separated.
\end{proof}



\begin{remark}  When $\cY$ is locally noetherian, the above proof of part (3) shows more generally that if $f \co \cX \to \cY$ is a morphism of algebraic stacks and $g \co \cY' \to \cY$ is a quasi-compact,  universally submersive morphism
such that $\cX \times_{\cY} \cY' \to \cY'$ is weakly separated (resp., weakly proper), then $f \co \cX \to \cY$ is weakly separated (resp., weakly proper).
\end{remark}


\begin{example}
\label{example-weakly-separated1}
A weakly separated morphism of algebraic spaces need not be separated.  For example, take the bug-eyed cover of $\AA^1$ over $\CC$ obtained by taking the quotient $X/\ZZ_2$ of the non-separated affine line $X = \AA^1 \bigcup_{\AA^1 \setminus \{0\}} \AA^1$ by the action of $\ZZ_2$ which acts by $x \mapsto -x$ and flips the origins.  Let $h_1, h_2 \co \Spec \CC[[x]] \to X/\ZZ_2$ be the two morphisms obtained by mapping to the two origins in $X$.  This gives a diagram
$$\xymatrix{
 \Delta^* \ar[r] \ar[d]^i	& X/\ZZ_2 \ar[d]^f \\
 \Delta \ar[r] \ar@<.5ex>[ur]^{h_1} \ar@<-.5ex>[ur]_{h_2}		& \Spec \CC
}
$$
with $h_1 \circ i = h_2 \circ i$ and $h_1(0) = h_2(0) \in |X/\ZZ_2|$.  However, $h_1 \neq h_2$.
\end{example}

\begin{example}
\label{weakly-separated2}
A Deligne-Mumford stack with non-finite inertia may be weakly separated.  For an example, let $\cX$ be the Deligne-Mumford locus of $[(\Sym^4 \PP^1) / PGL_2]$ over $\CC$ consisting of points with finite stabilizer groups.  Then $\cX \to \Spec \CC$ is weakly separated but is not separated.
\end{example}

\begin{example} 
\label{example-weakly-separatd3}
Consider the $\GG_m$-action on the nodal cubic $X$ in $\PP^2$.  Let $\cX = [X/\GG_m]$.  Then $\cX$ is weakly proper but does not admit a good moduli space; a good moduli space would necessarily be $\Spec \CC$ which would imply that $\cX$ is cohomologically affine and therefore that $X$ is affine, a contradiction.  Moreover, consider the $\GG_m$-action on the normalization $\tilde X \cong \PP^1$ fixing $0$ and $\infty$.  Consider the composition $[\tilde X / \GG_m] \stackrel{f}{\to} [X/\GG_m] \stackrel{g}{\to} \Spec \CC$.  Then $g$ is weakly proper and $f$ is finite, but the composition $g \circ f$ is not weakly separated.
\end{example}


The following proposition is the main result of this section and justifies the introduction of the weakly separated/properness condition.

\begin{prop} \label{prop-good-weakly-separated}
Suppose $\cX$ is an algebraic stack over a scheme $S$.  Let $\phi \co \cX \to Y$ be a good moduli space.  
\begin{enumerate}
\item $\phi$ is weakly separated.  If $\phi$ is also finite type (e.g., if $\cX \to S$ is finite type), then $\phi$ is weakly proper.
\end{enumerate}
Suppose in addition that $Y$ is locally separated.  Then
\begin{enumerate} \setcounter{enumi}{1}
\item $\cX \to S$ is weakly separated if and only if $Y \to S$ is separated.
\item If $\cX$ and $Y$ are finite type over $S$, then $\cX \to S$ is weakly proper if and only if $Y \to S$ is proper.
\end{enumerate}
\end{prop}

\begin{remark}  Recall that if $S$ is locally noetherian and $\cX \to S$ is finite type, then $Y \to S$ is also finite type; see \cite[Theorem 6.3.3]{alper_adequate}.  
\end{remark}

\begin{proof} 
For $(1)$, suppose we have a commutative diagram as in Diagram (\ref{diagram_weakly_separated}) for the morphism $\cX \to Y$ with two lifts $h_1, h_2 \co \Delta \to \cX$   Then the base change $\cX \times_Y \Delta \to \Delta$ is a good moduli space and therefore there is a unique closed point in $\cX \times_Y \Delta$ over the closed point of $0 \in \Delta$ (see \cite[Prop. 4.7, Thm. 4.16]{alper_good_arxiv}).  Therefore $h_1(0) = h_2(0) \in |\cX \times_Y \Delta|$.  The second statement in $(1)$ follows because $\phi$ is universally closed (\cite[Thm. 4.16]{alper_good_arxiv}).

For $(2)$, suppose $Y \to S$ is separated.  As the composition of a weakly separated morphism followed by a separated morphism is weakly separated, it follows from (1) that $\cX \to S$ is weakly separated.
Conversely, suppose that $\cX \to S$ is weakly separated.  Then there exists a commutative diagram as in Diagram (\ref{diagram_weakly_separated}) for the morphism $Y \to S$ with two distinct lifts $h_1, h_2 \co \Delta \to Y$.  Since $\cX \to Y$ is universally closed, there exists an extension $K \hookarr K'$ and valuation ring $R' \subseteq K'$ dominating $R$, a lift $\Delta'^*=\Spec K' \to \cX'$ and two lifts $h'_1,  h'_2 \co \Delta'= \Spec R' \to \cX$ such that $h'_1(0')$ and $h'_2(0')$ are closed in $|\cX \times_S \Delta'|$, where $0' \in \Delta'$ is the closed point.  But since $\cX \to S$ is weakly separated, $h'_1(0') = h'_2(0')$ in $| \cX \times_S \Delta|$ so $h_1(0) = h_2(0)$ in $|Y \times_S \Delta|$.  Therefore, $Y \to S$ is weakly separated and by Proposition \ref{proposition-weakly-separated-properties}$(1)$, $Y \to S$ is separated.

For $(3)$, the morphism $\cX \to S$ is universally closed if and only if $Y \to S$ is universally closed as $\phi$ is surjective. 
\end{proof}

\section{$S_{m,1}/S_{m,2}$-curves and $H_{m,1}/H_{m,2}$-curves} \label{section-sh}

The purpose of this section is to construct weakly proper moduli stacks of crimping data (Section \ref{s-scurves}) and of stable tails associated to an arbitrary $A_k$-singularity (Section \ref{s-hcurves}). In both cases, the weak properness of the moduli stacks follows from their explicit construction as quotient stacks (Corollary \ref{cor-H-weakly-proper}).
\begin{defn}[$A_k$-singularities]

We say that a curve $C$ has an \emph{$A_k$-singularity} at a point
$p$ if $\widehat{\cO}_{C,p} \isom
\CC[[x,y]]/(y^2-x^{k+1})$.
An $A_1$- (resp., $A_2$-, $A_3$-, $A_4$-) singularity is also
called a \emph{node} (resp., cusp, tacnode, ramphoid cusp).
\end{defn}



\subsection{$S_{m,1}/S_{m,2}$-curves}\label{s-scurves}

\begin{defn}[$S_{m,1}/S_{m,2}$-curves]

\begin{enumerate}
\item[]
\item An \emph{$S_{m,1}$-curve} is a $1$-pointed curve of
arithmetic genus $m$ obtained by taking a smooth rational curve
with a labeled point and imposing an $A_{2m}$-singularity on it
at a point distinct from the labeled point.
\item An \emph{$S_{m,2}$-curve} is a $2$-pointed curve of
arithmetic genus $m$ obtained by taking two smooth rational
components, each with a labeled point, and identifying them at
points distinct from the labeled points to form an
$A_{2m+1}$-singularity.
\end{enumerate}
\end{defn}

A family of curves $(\cC \stackrel{\pi}{\arr} B,\sigma)$
(resp., $(\cC \stackrel{\pi}{\arr} B, \sigma_1,\sigma_2)$) is a
\emph{family of $S_{m,1}$- (resp., $S_{m,2}$-) curves} if its
geometric fibers are $S_{m,1}$- (resp., $S_{m,2}$-) curves,
the relative singular locus $\Sigma$ of $\pi$ maps
isomorphically to $B$ and $\pi$ has an $A_{2m}$-
(resp., $A_{2m+1}$-) singularity along $\Sigma$. Note that this is
not just a condition on fibers.

\begin{defn}

Let $\bar{\cS}_{m,1}$ (resp., $\bar{\cS}_{m,2}$) be the moduli stack of families of $S_{m,1}$- (resp., $S_{m,2}$-) curves.
\end{defn}

\begin{prop}\label{s-quotient-stack}

$\bar{\cS}_{m,1}$ is isomorphic to the stack $[\AA^{m-1}/\GG_m]$,
where for $m \geq 2$ $\GG_m$ acts with weights
$1,3,\ldots,2m-3$. $\bar{\cS}_{m,2}$ is isomorphic to the stack
$[\AA^{m-1}/\GG_m]$, where for $m \geq 2$ $\GG_m$ acts with
weights $1,2,\ldots,m-1$.
\end{prop}

\begin{proof}

This is proved carefully in \cite[Examples 1.111 and 1.112]{wyck}. Let us sketch the idea of the isomorphism $\bar{\cS}_{m,1} \simeq [\AA^{m-1}/\GG_m]$. To specify the isomorphism class of an $S_{m,1}$-curve is equivalent to specifying a $\CC$-subalgebra $R \subset \CC[[t]] \simeq \hat{\O}_{\P^1,0}$ which is abstractly isomorphic to $\CC[[x,y]]/(y^2=x^{2m+1})$. It is not difficult to see that any such subalgebra is generated by $t^{2m+1}$ and $t^2+a_1t^3+a_2t^5+ \ldots a_{m-1}t^{2m-1}$, with $a_i \in \CC$, and that two such subalgebras are isomorphic iff
$
(a_1,\ldots, a_{m-1})=(\lambda a_1', \ldots, \lambda^{m-1}a_{m-1}'), \lambda \in \GG_m.
$

\end{proof}

\subsection{$H_{m,1}/H_{m,2}$-curves}\label{s-hcurves}

\begin{defn}[$H_{m,1}/H_{m,2}$-curves] 
\begin{enumerate}
\item[]
\item An \emph{$H_{m,1}$-curve} is a $1$-pointed curve $(E,q)$ of
arithmetic genus $m$ which admits a finite, surjective, degree two
map $\phi: E \rightarrow \P^{1}$ such that
$\phi^{-1}(\{\infty\})=\{q\}$ and $q$ is a smooth point.
\item An \emph{$H_{m,2}$-curve} is a $2$-pointed curve
$(E,q_1,q_2)$ of arithmetic genus $m$ which admits a finite,
surjective, degree two map $\phi: E \rightarrow \P^{1}$ such that
$\pi^{-1}(\{\infty\})=\{q_1+q_2\}$ and $q_1,q_2$ are smooth
points.
\end{enumerate}
\end{defn}
\begin{remark}
It will occasionally be useful to use the notation
\emph{$H_k$-curve}, defined as follows: if $k=2m$ is even, then an
$H_k$-curve is an $H_{m,1}$-curve; if $k=2m+1$ is odd, then an
$H_k$ curve is an $H_{m,2}$-curve.
\end{remark}

\begin{lemma}\label{C:Akcurves}
\begin{enumerate}
\item[]
\item If $(E,q)$ is an $H_{m,1}$-curve, then $E$ is irreducible.
\item If $(E,q_1,q_2)$ is an $H_{m,2}$ curve, then $E$ has at most
two irreducible components. Furthermore, if $E$ has two
components, each component is a smooth rational curve.
\end{enumerate}
\end{lemma}
\begin{proof}
In case (1), $\O_{E}(q)$ is ample, since
$\O_{E}(2q)=\phi^*\O_{\P^{1}}(1)$. In particular, $\O_{E_1}(q)$
has positive degree on every irreducible component of $E$, which
forces $E$ to be irreducible. Similarly, in case (2),
$\O_{E}(q_1+q_2)$ has positive degree on each irreducible
component, which means that $E$ has at most two components, one
containing $q_1$ and one containing $q_2$. Finally, if $E$ has two
components, say $E=E_1 \cup E_2$, then each composition $E_1
\hookrightarrow E \rightarrow \P^{1}$ is finite, surjective of
degree one, so $E_i \simeq \P^{1}$.
\end{proof}

A family of $1$-pointed curves $(\cC \stackrel{\pi}{\arr}
B,\sigma)$ is a \emph{family of $H_{m,1}$-curves} if there is a
map $\phi$ from $\cC$ to a $\PP^1$-bundle $\cP$ over $B$,
such that $\phi$ is a uniform cyclic cover of degree $2$ and
branch degree $m+1$ which is simply ramified along $\sigma$.

A family of $2$-pointed curves $(\cC \stackrel{\pi}{\arr}
B,\sigma_1,\sigma_2)$ is a \emph{family of $H_{m,2}$-curves} if
there is a map $\phi$ from $\cC$ to a $\PP^1$-bundle $\cP$ over
$B$,
such that $\phi$ is a uniform cyclic cover of degree $2$ and
branch degree $m+1$ and $\phi \sigma_1 = \phi \sigma_2$.




\begin{defn}

Let $\bar{\cH}_{m,1}$ (resp., $\bar{\cH}_{m,2}$) be the moduli stack of families of $H_{m,1}$- (resp., $H_{m,2}$-)
curves.
\end{defn}

\begin{prop}\label{h-quotient-stack}

$\bar{\cH}_{m,1}$ is isomorphic to the stack $[\AA^{2m}/\GG_m]$
where $\GG_m$ acts with weights
$-4,-6,\ldots,-(4m+2)$. $\bar{\cH}_{m,2}$ is isomorphic to the
stack $[\AA^{2m+1}/\GG_m]$, where $\GG_m$ acts with weights
$-2,-3,\ldots,-(2m+2)$.
\end{prop}

\begin{proof}

It follows as in the proof of \cite[Theorem~4.1]{arsie-vistoli}
that $\bar{\cH}_{m,1}$ is isomorphic to
$[\AA_{2m+2}/(\Tri_2/\mu_{m+1})]$, where 
\[
\AA_{2m+2} := \Spec \CC[a_{2m+1},a_{2m+1}^{-1},a_{2m},\ldots,a_0]
\]
and $\Tri_2$ is the subgroup scheme of $\GL_2$ of upper-triangular
matrices and $\mu_{m+1}$ is the group scheme of
$(m+1)$\textsuperscript{th} roots of unity, embedded in $\Tri_2$
as multiples of the identity matrix. For any ring $A$, we regard
$\AA_{2m+2}(A)$ as the following set of homogeneous
polynomials in the variables $x,z$:
\[
\{ a_{2m+1}x^{2m+1}z+a_{2m} x^{2m}
z^2+\ldots+a_0z^{2m+2}:a_{2m+1}\in A^\times;a_{2m},\ldots,a_0\in
A\}.
\]
The right action of $\Tri_2/\mu_{m+1}$ on $\AA_{2m+2}$ is given
in functorial notation by
\[
f(\begin{smallmatrix}
x\\z \end{smallmatrix}).[M] = f\left(M(\begin{smallmatrix}
x\\z \end{smallmatrix})\right).
\]
This is clearly well-defined. Note that the absence of an
$x^{2m+2}$ term in elements of $\AA_{2m+2}(A)$ corresponds to the
requirement that $\phi:\cC \arr \cP$ be ramified along $\sigma$
and the condition that the coefficient of $x^{2m+1}z$ is a unit in
$A$ corresponds to the requirement that $\cC$ be smooth along
$\sigma$.

It follows similarly that $\bar{\cH}_{m,2}$ is isomorphic
to $[\AA_{2m+3}/(\Tri_2/\mu_{m+1})]$, where
\[
\AA_{2m+3} := \Spec \CC[b,b^{-1},a_{2m+1},a_{2m},\ldots,a_0].
\]
We regard $\AA_{2m+3}(A)$ as the following set of homogeneous
polynomials in the variables $x,z$.
\[
\{ b x^{2m+2} + a_{2m+1}x^{2m+1}z+a_{2m} x^{2m}
z^2+\ldots+a_0z^{2m+2}:b\in A^\times; a_{2m+1},\ldots,a_0\in A\}.
\]
The right action of $\Tri_2/\mu_{m+1}$ on $\AA_{2m+3}$ is now
defined as follows. The coefficients of
$x^{2m+1}z,\ldots,z^{2m+2}$ in $f(\begin{smallmatrix}
x\\z \end{smallmatrix}).[M]$ are determined via the formula
\[
f(\begin{smallmatrix} x\\z \end{smallmatrix}).[M] =
f\left(M(\begin{smallmatrix} x\\z \end{smallmatrix})\right)
\]
but the coefficient of $x^{2m+2}$ in $f(\begin{smallmatrix}
x\\z \end{smallmatrix}).[M]$ is $(M_{11})^{m+1}b$, where $b$ is
the coefficient of $x^{2m+2}$ in $f(\begin{smallmatrix}
x\\z \end{smallmatrix})$. This action is clearly well-defined.
Note that the condition on elements of $\AA_{2m+3}(A)$ that the
coefficient of $x^{2m+2}$ be a unit corresponds to the requirement
that $\cC$ is not ramified along $\sigma_1,\sigma_2$. The fact
that $[M]$ acts on the coefficient of $x^{2m+2}$ by multiplication
by $(M_{11})^{m+1}$ rather than $(M_{11})^{2m+2}$ corresponds to
the requirement that $\sigma_1, \sigma_2$ have been given an
ordering, and that automorphisms respect the ordering.

Return to the the stack $\cH_{m,1}$ and consider the map $\pi:
\AA^{2m} \arr [\AA_{2m+2}/(\Tri_2/\mu_{m+1})]$ induced by the map
$\imath:\AA^{2m} \arr \AA_{2m+2}$, $(a_{2m-1},\ldots,a_0) \mapsto
x^{2m+1}z+a_{2m-1}x^{2m-1}z^3+\ldots+a_0 z^{2m+2}$. We will check
that $\pi$ is smooth and surjective and that the associated
groupoid scheme is isomorphic to the groupoid scheme
$\xymatrix{\AA^{2m} \times \GG_m \ar@<0.5ex>[r] \ar@<-0.5ex>[r] &
\AA^{2m}}$ associated to the specified action of $\GG_m$ on
$\AA^{2m}$.
Indeed, consider the $2$-commutative diagram
\[
\xymatrix{ R \ar[r] \ar[d] &
           R_2 \ar[r] \ar[d] &
           \AA^{2m} \ar[d]^\imath \\
           R_1 \ar[r] \ar[d] &
           \AA_{2m+2} \times (T_2/\mu_{m+1}) 
             \ar[r]^{t} \ar[d]^{s} &
           \AA_{2m+2} \ar[d] \\
           \AA^{2m} \ar[r]^\imath &
           \AA_{2m+2} \ar[r] &
           [\AA_{2m_2}/(T_2/\mu_{m+1})] }
\]
where $t$ and $s$ denote the action and projection maps
respectively and $R_1$, $R_2$ and $R$ are
defined to make the relevant squares Cartesian. It follows
immediately from the definitions that we can
identify $R$ with (using functorial notation)
\[
\{ f,[M] : f_{2m+1}=1, f_{2m}=0,
                  (f.[M])_{2m+1}=1, (f.[M])_{2m}=0 \}.
\]
But this is equal to
\[
\{ f,[M] : f_{2m+1}=1, f_{2m}=0, 
                  M_{22}=(M_{11})^{-2m-1}, M_{12}=0 \},
\]
using the fact that $2m+1$ is not a zero-divisor
since $\ch(\CC)=0$. Under this identification, the horizontal map
$R \arr \AA^{2m}$ is
\[
(a_{2m-1},\ldots,a_0) , [(\begin{smallmatrix}\alpha & 0 \\ 0 &
\alpha^{-2m-1} \end{smallmatrix})] \mapsto
(\ldots,\alpha^{i+(2m+2-i)(-2m-1)} a_{i},\ldots).
\]
Now the map of group schemes
\[
\{ [M] \in T_2/\mu_{m+1}: M_{22}=(M_{11})^{-2m-1}, M_{12}=0 \}
\arr \GG_m \quad 
[M] \mapsto (M_{11})^{m+1}
\]
is an isomorphism,
so we can identify $R$ with $\AA^{2m} \times \GG_m$. Under this
identification, the horizontal map $R \arr \AA^{2m}$ is
\[
(a_{2m-1},\ldots,a_0) , t \mapsto
(\ldots,t^{2i-4m-2} a_{i},\ldots).
\]
and the vertical map $R \arr \AA^{2m}$ is projection onto the
first factor. This proves the desired isomorphism of groupoid
schemes. It remains to check that the map $R_2 \arr
\AA_{2m+2}$ is smooth and surjective. This is immediate from the
fact that $R_2$ can be identified with $\AA_{2m+2} \times
(\GG_m/\mu_{m+1})$ (using as above that $\ch{\CC}=0$) in such a
way that that the map to $\AA_{2m+2}$ is projection onto the first
factor.

For $\bar{\cH}_{m,2}$, we find, using analogous notation, that $R$
can be identified with
\[
\{ f,[M] : f_{2m+1}=1, f_{2m}=0, 
                  M_{11}^{m+1}=1, M_{12}=0 \}.
\]
Now using the isomorphism of group schemes
\[
\{ [M] \in T_2/\mu_{m+1}: (M_{11})^{m+1}=1, M_{12}=0 \}
 \arr \GG_m \quad [M] \mapsto M_{11}/M_{22},
\]
we can identify $R$ with $\AA^{2m+1} \times \GG_m$. Under this
identification, the horizontal map $R \arr \AA^{2m+1}$ is
\[
(\ldots,a_i,\ldots), t \mapsto
(\ldots,t^{i-2m-2} a_{i},\ldots)
\]
and the vertical map $R \arr \AA^{2m+1}$ is projection onto the
first factor. This proves the desired isomorphism of groupoid
schemes. It is easy to complete the proof by checking as for
$\bar{\cH}_{m,2}$ that the map $R_2 \arr \AA_{2m+1}$ is smooth and
surjective.
\end{proof}

\subsection{Monomial $H_{m,1}/H_{m,2}$-curves} \label{section-monomial}

\begin{defn}[Monomial $H_{m,1}/H_{m,2}$-curves]
\label{definition-monomial}

\begin{enumerate}
\item[]
\item The \emph{monomial $H_{m,1}$-curve} is the following
$1$-pointed curve $(E,q)$: $E$ is obtained by identifying $\Spec
\CC[x,y]/(y^2-x^{2m+1})$ and
$\Spec \CC[s]$ along $D(x)$ and $D(s)$ via
$x=s^{-2},y=s^{-(2m+1)}$; $q$ is the point
$s=0$.
\item The \emph{monomial $H_{m,2}$-curve} is the following
$2$-pointed curve $(E,q_1,q_2)$: $E$ is obtained by identifying
$\Spec \CC[x,y]/(y^2-x^{2m+2})$ and
$\Spec \CC[s_1] \dunn \Spec \CC[s_2]$ along $D(x)$ and
$D(s_1) \dunn D(s_2)$ via $x=s_1^{-1}\oplus
s_2^{-1}$, $y=s_1^{-(m+1)}\oplus
-s_2^{-(m+1)}$; $q_1$ is the point $s_1=0$ of $\Spec \CC[s_1]$ and
$q_2$ is the point $s_2=0$ of $\Spec \CC[s_2]$.
\end{enumerate}
\end{defn}

The monomial $H_{m,1}$- (resp., $H_{m,2}$-) curve is also an
$H_{m,1}$- (resp., $H_{m,2}$-) curve.  We will denote by $p$
its singular point.

The automorphism group scheme of the monomial $H_{m,1}$-curve
$(E,q)$ is isomorphic to $\GG_m$. We fix once and for all the
following isomorphism $\GG_m \iso \Aut(E,q)$:
\[
A^\cross \arr \Aut(E,q)(A) \qquad
a \mapsto (x \mapsto a^2 x, y \mapsto a^{2m+1} y,
           s\mapsto a^{-1}s).
\]

The automorphism group scheme of the monomial $H_{m,2}$-curve
$(E,q_1,q_2)$ is isomorphic to $\GG_m$. We fix once
and for all the following isomorphism $\GG_m \iso
\Aut(E,q_1,q_2)$:
\begin{gather*}
A^\cross \arr \Aut(E,q_1,q_2)(A) \\
a \mapsto (x \mapsto a x, y \mapsto a^{m+1} y,
            s_1 \mapsto a^{-1} s_1
            s_2 \mapsto a^{-1} s_2)
\end{gather*}

\begin{prop}\label{prop-H-S-minimal}

Every $S_{m,1}$- (resp., $S_{m,2}$-) curve and every $H_{m,1}$
(resp., $H_{m,2}$-) curve admits an isotrivial specialization to
the monomial $H_{m,1}$- (resp., $H_{m,2}$-) curve, which is the
unique closed point of $\bar{\cS}_{m,1}$ (resp., $\bar{\cS}_{m,2}$)
and $\bar{\cH}_{m,1}$ (resp., $\bar{\cH}_{m,2}$).
\end{prop}

\begin{proof}

It is immediate from the descriptions of $\bar{\cS}_{m,1}$
(resp., $\bar{\cS}_{m,2}$) and $\bar{\cH}_{m,1}$
(resp., $\bar{\cH}_{m,2}$), Proposition~\ref{s-quotient-stack} and
Proposition~\ref{h-quotient-stack}, that the point corresponding
to the monomial $H_{m,1}$- (resp., $H_{m,2}$-) curve is the unique
closed point of both stacks and that every point contains this
point in its closure.
\end{proof}
\begin{cor}\label{cor-H-weakly-proper}

$\bar{\cH}_{m,1}$, $\bar{\cH}_{m,2}$, $\bar{\cS}_{m,1}$, $\bar{\cS}_{m,2}$ are weakly proper. \epf
\end{cor}

\section{$A_{k}^-/A_{k}/A_{k}^+$-stability} \label{section-stability}

In this section, we define $A_{k}^-/A_{k}/A_{k}^+$-stability for
$k \in\{2,3,4\}$, and show that these are deformation open conditions.  In particular, we show that corresponding moduli stacks $\SM_{g,n}(A_{k}^-)$, $\SM_{g,n}(A_{k})$, and $\SM_{g,n}(A_{k}^+)$ are algebraic stacks of finite type over $\mathbb{C}$.

\subsection{Definition of $A_{k}^-/A_{k}/A_{k}^+$-stability}

\begin{definition}[Gluing morphism]
If $(E,\qm)$ is an $m$-pointed curve and $C$ is any curve, a \emph{gluing morphism} $i:(E, \qm) \hookrightarrow C$ is a morphism $E \rightarrow C$, which is an open immersion when restricted to $E-\{q_1, \ldots, q_m\}$.
\end{definition}
\begin{remark}
\begin{enumerate}
\item[]
\item We do not require the points $\{i(q_i)\}_{i=1}^{m}$ to be distinct. 
\item Locally around $i(q_j)$, $i$ is the normalization of one branch of $i(q_j) \in C$.
\end{enumerate}
\end{remark}

\begin{definition} We say that an $n$-pointed curve $(C, \pn)$
\emph{contains an $H_{m,1}$-curve} (resp., \emph{contains an
$H_{m,2}$-curve}) if there is a gluing morphism
\[
i:(E,q) \hookrightarrow C
\text{  (resp.,  }
i:(E,q_1,q_2) \hookrightarrow C
\text{  ),}
\]
where $(E,q)$ is an $H_{m,1}$-curve (resp., $(E,q_1,q_2)$ is an
$H_{m,2}$-curve). In this case we say also that $(E,q)$ \emph{is
an $H_{m,1}$-tail} (resp., $(E,q_1,q_2)$ \emph{is an
$H_{m,2}$-bridge}) of $(C,\pn)$.
\end{definition}

\begin{definition}
We say that an $n$-pointed curve $(C, \pn)$ \emph{contains an $H_{m,2}$-chain of length r} if there exists a morphism
$$i: \coprod_{i=1}^{r}(E_i,q_{2i-1},q_{2i}) \hookrightarrow C,$$
where each $i|_{(E_i, q_{2i-1}, q_{2i})}$ is a gluing morphism satisfying:
\begin{enumerate}
\item $(E_i,q_{2i-1},q_{2i})$ is an $H_{m,2}$-curve for $i=1, \ldots, r$;
\item $i(q_{2i})=i(q_{2i+1})$ is an $A_{2m+1}$-singularity of $C$ for $i=1, \ldots, r-1$.
\end{enumerate}
In this case, we say also that $(\bigcup_{i=1}^{r}
E_i,q_{1},q_{2r}) $ \emph{is an $H_{m,2}$-chain} of $(C,\pn)$
\end{definition}
\begin{remark}
An $H_{m,2}$-bridge is the same thing as an $H_{m,2}$-chain of
length one.
\end{remark}

\begin{definition}[Nodally-attached and destabilizing tails/chains]
\begin{enumerate}
\item[]
\item We say that an $H_{m,1}$-tail $i:(E,q) \hookrightarrow C$ is \emph{nodally-attached} if $i(q)$ is a node or marked point. We say that an $H_{m,2}$-chain  $i:  (E,q_{1},q_{2}) \hookrightarrow C$ is \emph{nodally-attached} if $i(q_1), i(q_{2})$ are nodes or marked points.
\item We say that an $H_{m,1}$-tail $i:(E,q) \hookrightarrow C$ is \emph{destabilizing} if $i(q)$ is a node, a marked point, or an $A_{l}$-singularity with $l \geq 2m+1$. We say that an $H_{m,2}$-chain  $i:  (E,q_{1},q_{2}) \hookrightarrow C$ is \emph{destabilizing} if $i(q_1), i(q_{2})$ are nodes, marked points, or $A_{l}$-singularities with $l \geq 2m+2$.
\end{enumerate}
\end{definition}

\begin{remark}
\begin{enumerate}
\item[]
\item We will sometimes say that $C$ has an $H_k$-curve to mean that $C$ has an $H_{m,1}$-curve or $H_{m,2}$-curve (in the case that $k$ is even or odd respectively) where $m$ is the unique integer such that $k=2m$ or $k=2m+1$. Similarly, if $k$ is odd, we will say that $C$ contains an $H_k$-chain to mean that $C$ contains an $H_{m,2}$-chain, where $m$ is the unique integer such that $k=2m+1$.
\item If $C$ is a curve with only $A_1, \ldots, A_k$-singularities, then a destabilizing $H_k$-curve or $H_k$-chain of $C$ is the same as a nodally-attached $H_k$-curve or $H_k$-chain.
\end{enumerate}
\end{remark}

Now we may define our stability conditions:
\begin{definition}[$A_{k}^-/A_{k}/A_{k}^+$-stability] \label{definition-stability}
An $n$-pointed curve $(C, \pn)$ is \emph{$A_{k}^-/A_{k}/A_{k}^+$-stable} if  $\omega_{C}(\Sigma_ip_i)$ is ample, and
\begin{enumerate}
\item ($A_{k}^-$-stability)
\begin{enumerate}
\item $C$ has only $A_{l}$-singularities, $l < k$,
\item $C$ contains no destabilizing $H_l$-curves/chains, $l<k$. Equivalently, $C$ contains no $H_{m,1}$-tails $( m< \frac{k}{2})$ or destabilizing $H_{m,2}$-chains $(m  < \frac{k-1}{2})$.
\end{enumerate}
\item ($A_{k}$-stability)
\begin{enumerate}
\item $C$ has only $A_{l}$-singularities, $l  \leq k$,
\item $C$ contains no destabilizing $H_l$-curves/chains, $l<k$. Equivalently, $C$ contains no $H_{m,1}$-tails $( m< \frac{k}{2})$ or destabilizing $H_{m,2}$-chains $(m  < \frac{k-1}{2})$.
\end{enumerate}
\item ($A_{k}^+$-stability)
\begin{enumerate}
\item $C$ has only $A_{l}$-singularities, $l  \leq k$,
\item $C$ contains no destabilizing $H_l$-curves/chains, $l \leq k$. Equivalently, $C$ contains no $H_{m,1}$-tails $( m \leq \frac{k}{2})$ or destabilizing $H_{m,2}$-chains $(m  \leq \frac{k-1}{2})$.
\end{enumerate}
\end{enumerate}
\end{definition}
\begin{remark}
\begin{enumerate}
\item[]
\item $A_{k}^{+}$-stability is the same as $A_{k+1}^{-}$-stability.
\item The definition of $A_2^+$-stability is equivalent to Schubert's definition of pseudostability \cite{schubert}. The definition of $A_3$-stability and $A_3^+$-stability is equivalent to Hassett and Hyeon's definition of $c$-semistability and $h$-semistability respectively \cite{HH2}. The definitions of $A_4$ and $A_4^+$-stability are original to this paper.
\item The most subtle point in this definition is the fact that when $A_k$-singularities are introduced, one removes not arbitrary $H_k$-curves/chains, but only \emph{destabilizing} ones. The motivation for the definition of a destabilizing $H_k$-curve/chain stems from the order in which singularities are introduced. For example, a tacnodally-attached elliptic tail is a destabilizing $H_{1,1}$-curve and hence may not appear in an $A_3$-stable curve. Intuitively, the reasons for this is that elliptic tails should be replaced by cusps before elliptic bridges are replaced by tacnodes, so the correct $A_3$-stable replacement of a curve with an elliptic tail dangling off an elliptic bridge is to create a cuspidally-attached elliptic bridge (which \emph{is} $A_3$-stable), rather than a tacnodally-attached is elliptic tail.
\end{enumerate}
\end{remark}

A family of $A_{k}^-/A_{k}/A_{k}^+$-stable curves is defined in
the usual way to be a flat, proper, finitely presented morphism,
together with $n$ sections, whose geometric fibers are
$A_{k}^-/A_{k}/A_{k}^{+}$-stable curves of arithmetic genus $g$. Evidently, families of $A_{k}^-/A_{k}/A_{k}^+$-stable curves form a stack, so we may make the following definition.

\begin{defn}\label{definition-main-stacks}
Let $\SM_{g,n}(A_{k}^-), \SM_{g,n}(A_{k}) ,\SM_{g,n}(A_{k}^+)$ denote the moduli stacks of families of $A_{k}^-/A_{k}/A_{k}^+$-stable curves.
\end{defn}


\subsection{Deformation-Openness of Moduli Functors}

The purpose of this section is to prove Proposition \ref{P:Openness} stating that the stability conditions introduced in the previous section are deformation-open conditions. We will do this by showing that the stacks $\SM_{g,n}(A_k)$ are obtained by removing certain closed loci from the stack of all curves with $A_k$-singularities.

\begin{definition} \label{defn-ugn}
Let $\U_{g,n}(A_k)$ be the moduli stack of families of curves $(\pi:\C \rightarrow T, \sigman)$ satisfying:
\begin{enumerate}
\item $\omega_{\C/T}$ is relatively ample.
\item The sections $\sigman$ are distinct and lie in the smooth locus of $\pi$
\item The geometric fibers of $\pi$ are connected, reduced curves of arithmetic genus $g$.
\item The only singularities of the geometric fibers of $\pi$ are of type $A_1$-$A_k$.
\end{enumerate}
Since $\U_{g,n}(A_k)$ parameterizes canonically polarized curves, $\U_{g,n}(A_k)$ is obviously an algebraic stack of finite type over $\mathbb{C}$.
\end{definition}

\begin{definition} \label{definition-attaching}
If $(C, \pn)$ is a curve containing an $H_{m,1}$-tail (resp., $H_{m,2}$-chain) $i: (E, \qm) \hookrightarrow C$, we say that this $H_{m,1}$-tail (resp., $H_{m,2}$-chain) has \emph{$A_{k}$-attaching} (resp. \emph{$A_{k_1, k_2}$-attaching}) if $i(q_1)$ is an $A_{k}$-singularity (resp. $i(q_1)$, $i(q_2)$ are $A_{k_1}$, $A_{k_2}$-singularities). Note that we allow $k, k_1, k_2$ to be zero, with the understanding that in this case $i(q)$ (resp., $i(q_1)$, $i(q_2)$) is one of the marked points of $C$. We may then define the following constructible subsets of $\U_{g,n}(A_{l})$:
\begin{align*}
\T_{m}^{k}:=& \text{ Locus of curves admitting an $H_{m,1}$-tail with $A_{k}$-attaching},\\
\B_{m}^{k_1,k_2}:=& \text{ Locus of curves admitting an $H_{m,2}$-chain with $A_{k_1}/A_{k_2}$-attaching}.
\end{align*}

\end{definition}

Each of our stability conditions is defined by removing loci of the form $\T_{m}^{k}$ and $\B_{m}^{k_1,k_2}$ from $\U_{g,n}(A_k)$. More precisely, we have (as sets):

\begin{align*}
& \SM_{g,n}(A_1)=\U_{g,n}(A_1)&& \SM_{g,n}(A_1^+)=\SM_{g,n}(A_1)\\
& \SM_{g,n}(A_2)=\U_{g,n}(A_2) && \SM_{g,n}(A_2^+)=\SM_{g,n}(A_2)-\bigcup_{j \in \{0,1\}}T^{j}_1\\
& \SM_{g,n}(A_3)=\U_{g,n}(A_3)-\bigcup_{j \in \{0,1,3\}}T_1^j&& \SM_{g,n}(A_3^+)=\SM_{g,n}(A_3)-\bigcup_{i,j \in \{0,1\}}B_1^{i,j}\\
& \SM_{g,n}(A_4)=\U_{g,n}(A_4)-\bigcup_{j \in \{0,1,3\}}T_1^j-\bigcup_{i,j \in \{0,1,4\}}B_1^{i,j}&& \SM_{g,n}(A_4^+)=\SM_{g,n}(A_4)-\bigcup_{j \in \{0,1\}}T_2^{j},
\end{align*}

To show that our stability conditions are open, we must show that at each stage the collection of loci $\T_{m}^{k}$ and $\B_{m}^{k_1,k_2}$ that we excise are closed. For this, we must analyze degenerations of curves with $H_{m,1}$-tails and $H_{m,2}$-chains. We break this analysis into two stages: In Lemma \ref{L:HmLimits}, we analyze degenerations of a single $H_{m,1}$-tail or $H_{m,2}$-bridge, and in Lemma \ref{L:LimitNode}, we analyze how the attaching singularities of an $H_{m,1}$-tail or $H_{m,2}$-chain may degenerate. Combining these results will allow us to prove the desired statement (Proposition \ref{P:Openness}).

\begin{figure}
\scalebox{.55}{\includegraphics{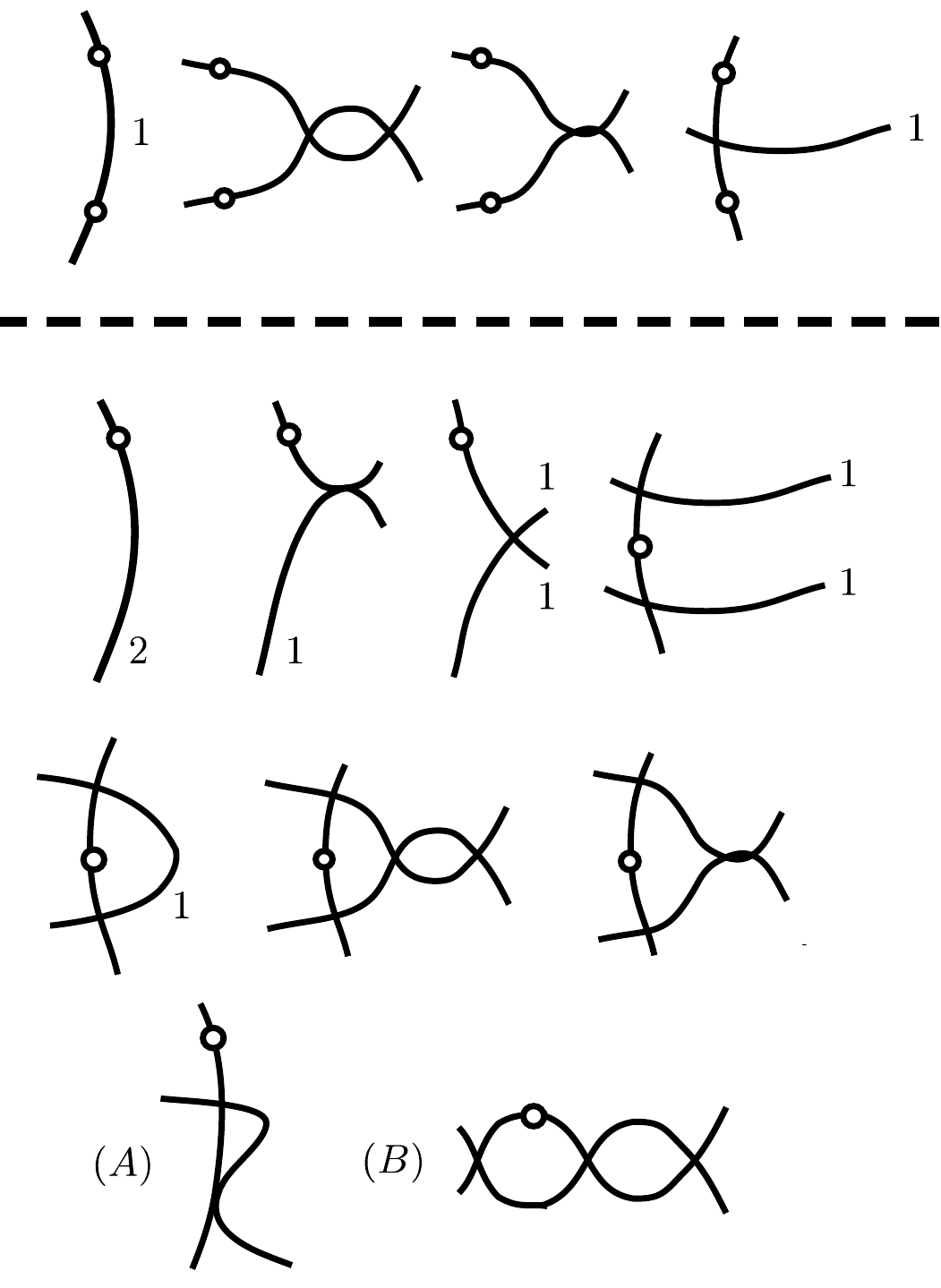}}
\caption{Topological types of curves in $\U_{1,2}(A_4)$ and $\U_{2,1}(A_4)$. For convenience, we have suppressed the data of singularities internal to each component, and we record only: the arithmetic genus of each component, and the singularities where two components meet (which are either nodes or tacnodes, as indicated by the picture). Components without a label have arithmetic genus zero.}\label{F:TopologicalTypes}
\end{figure}

\begin{lemma}[Limits of $H_{m,1}/H_{m,2}$-curves]\label{L:HmLimits}
\begin{itemize}
\item[]
\item[(1)]Let $(\H \rightarrow \Delta, \tau_1)$ be a family in $\U_{1,1}(A_{4})$ whose generic fiber is a smooth $H_{1,1}$-curve. Then the special fiber $C$ is an $H_{1,1}$ curve.
\item[(2)]Let $(\H \rightarrow \Delta, \tau_1, \tau_2)$ be a family in $\U_{1,2}(A_{4})$ whose generic fiber is a smooth $H_{1,2}$-curve. Then the special fiber $C$ satisfies one of the following conditions:
\begin{itemize}
\item[(a)] $C$ is an $H_{1,2}$-curve.
\item[(b)] $C$ contains an $H_{1,1}$-tail.
\end{itemize}
\item[(3)]Let $(\H \rightarrow \Delta, \tau_1)$ be a family in $\U_{2,1}(A_{4})$ whose generic fiber is a smooth $H_{2,1}$-curve. Then the special fiber $C$ satisfies one of the following conditions:
\begin{itemize}
\item[(a)] $C$ is an $H_{2,1}$-curve.
\item[(b)] $C$ contains an $H_{1,1}$-tail or an $H_{1,2}$-bridge.
\end{itemize}
\end{itemize}
\end{lemma}
\begin{proof}
For (1), the special fiber $(H,p)$ is necessarily a curve of arithmetic genus one with $\omega_{H}(p)$ ample. Since $\omega_{H}(p)$ has degree one, $H$ must be irreducible. It follows immediately (by Riemann-Roch) that $|2p|$ gives a degree two map to $\P^1$, so $(H,p)$ is an $H_{1,1}$-curve.

For (2), the special fiber $(H,p_1, p_2)$ is a curve of arithmetic genus one with  $\omega_{H}(p_1+p_2)$ ample. Since  $\omega_{H}(p_1+p_2)$ has degree two, $H$ has at most two components. The possible topological types of $H$ are listed in the top row of Figure \ref{F:TopologicalTypes}. We see immediately that any curve with one of the first three topological types is an $H_{2,1}$-curve, while any curve with the last topological type has an $H_{1,1}$-tail.

Finally, for (3), the special fiber $(H,p)$ is a curve of arithmetic genus two with  $\omega_{H}(p)$ ample. Since  $\omega_{H}(p)$ has degree three, $H$ has at most three components, and the possible topological types of $H$ are listed in the bottom three rows of Figure \ref{F:TopologicalTypes}. One sees immediately that if $H$ does not have an $H_{1,1}$-tail or an $H_{1,2}$-bridge, there are only three possibilities for the topological type of $H$: either $H$ is irreducible or $H$ has topological type $(A)$ or topological type $(B)$. Thus, it suffices to show that if $H$ is irreducible, then $(H,p)$ is an $H_{2,1}$-curve, and that curve of type $(A)$ and type $(B)$ cannot arise as limits of a family of smooth $H_{2,1}$-curves. The first claim is easy; we only need to know that $\omega_{H} \sim 2p$, but this follows from the corresponding linear equivalence on the general fiber.

It remains to show that topological types $(A)$ and $(B)$ cannot occur as the special fiber of a family of $H_{1,2}$-tails. For this, it suffices to prove that if $\H \rightarrow \Delta$ is \emph{any} family of genus two curves with smooth general fiber and special fiber of topological type $(A)$ or $(B)$, then the limits of the 6 Weierstrass points of the general fiber lie in the singular locus of the special fiber. More precisely, we claim that in a curve of type $(A)$, two Weierstrass points are absorbed into the node and four are absorbed into the tacnode, while in the curve of type $(B)$, two Weierstrass points are absorbed into each node. To see this, one simply observes that the unique isomorphism class of curve of type $(A)$ and $(B)$ can each be expressed as a double cover of $\P^1$ branched over $2(0)+4(\infty)$ and $2(0)+2(1)+2(\infty)$ respectively, and that all deformations of these curves are obtained by deforming the branch divisor (along with the corresponding cover).
\end{proof}

Next, we must consider how the singularities being used to attach a hyperelliptic bridge or tail may degenerate.
\begin{lemma}[Limits of Attaching Singularities]\label{L:LimitNode}
Suppose $\C \rightarrow \Delta$ is a flat, proper family of curves whose geometric fibers have only $A_{l}$-singularities. Suppose this family is endowed with $k$ sections $\tau_1, \ldots, \tau_k$ satisfying:
\begin{enumerate}
\item[(a)] $\tau_{i}(\bar{\eta}) \in \C_{\bar{\eta}}$ is an $A_{2m_i-1}$-singularity.
\item[(b)] The normalization of $\C_{\bar{\eta}}$ along $\bigcup_{i=1}^{k}\tau_{i}(\bar{\eta})$ consists of two connected components, and that $\pi^{-1}(\tau_i(\bar{\eta}))$ consists of two points $\alpha_i(\bar{\eta})$ and $\beta_i(\bar{\eta})$, with $\{\alpha_i(\bar{\eta})\}_{i=1}^{k}$ lying in the first component and $\{\beta_i(\bar{\eta})\}_{i=1}^{k}$ in the second.
\end{enumerate}
Then the normalization $\pi: \tilde{\C} \rightarrow \C$ of $\C$ along $\bigcup_{i=1}^{k} \tau_i$ consists of two connected components and (after a finite base-change) we may assume that $\pi^{-1}(\tau_i)$ splits into two sections $\alpha_i$ and $\beta_i$, with $\{\alpha_i\}_{i=1}^{k}$ lying in the first component and $\{\beta_i\}_{i=1}^{k}$ in the second. We claim that
\begin{enumerate}
\item If the limit points $\{\alpha_i(0)\}_{i=1}^{k}$ are distinct, then each limit point $\tau_i(0)$ remains on $A_{2m_i-1}$-singularity.
\item If any subset of limit points $\{\alpha_i(0) \}_{i \in S}$ coincide, then the subset $\{\beta_{i}(0)\}_{i \in S}$ also coincides, and the limit point $\tau_i(0)$ (for any $i \in S$) is an $A_{\sum_{j \in S}2m_{j}-1}$-singularity.
 \end{enumerate}
\end{lemma}
\begin{proof}
Note that since $\tilde{\C}$ is $S_{2}$, the special fiber $\tilde{\C}|_0$ is reduced and $\pi|_0: \tilde{\C}|_{0} \rightarrow \C|_0$ is a partial normalization of $C_0:=\C|_0$. We will show that it is actually the full normalization of $\C_0$.

First, suppose that the limit points $\{\alpha_i(0)\}_{i=1}^{k}$ remain distinct. We claim that the limits $\{\beta_i(0)\}_{i=1}^{k}$ remain distinct as well. If $\beta_i(0)$ and $\beta_j(0)$ coincide, then $\pi|_0$ maps $\alpha_i(0)$, $\alpha_j(0)$, and $\beta_{i}(0)=\beta_{j}(0)$ to the same point $p \in C_0.$ Since $\pi|_0$ is a partial normalization map, this implies that $p \in C$ has at least three branches. This is impossible since $A_{k}$-singularities have at most two branches. Next, we claim that each of the limit points $\{\alpha_i(0)\}_{i=1}^{k}$ and $\{\beta_i(0)\}_{i=1}^{k}$ is smooth. If not, then one of the limit points $\alpha_i(0)$ or $\beta_i(0)$ has either two branches or a singular branch. In the first case, the limit point $\tau_i(0)$ would have at least three branches (since $\alpha_i(0)$ and $\beta_i(0)$ map to the same point of $C_0$). In the second case, one branch of $\tau_i(0)$ would necessarily be singular. Both are impossible since we are assuming the special fiber has only $A_{k}$-singularities. Obviously, $\tilde{\C}|_0 \rightarrow \C|_0$ is finite and surjective, and we have just shown that there are two smooth points lying above each of the points $\tau_i(0)$. It follows that $\tilde{\C}|_0$ is in fact the normalization of $\C|_0$ at $\{\tau_i(0)\}_{i=1}^{k}$. Thus,
$$
\sum_{i=1}^{k}\delta(\tau_i(0))-1=p_a(\tilde{\C}|_0)-p_a(\C|_0)=p_a(\tilde{\C}|_{\eta})-p_a(\C|_{\eta})=\sum_{i=1}^{k}m_i-1.
$$
Obviously, $\delta(\tau_i(0)) \geq m_i$ since the $\delta$-invariant of a singularity can only increase under specialization, so the above inequality forces $\delta(\tau_i(0)) = m_i$ for each $i$. Since $\tau_i(0)$ is an $A_{k}$-singularity with two branches, we conclude $\tau_i(0)$ is an $A_{2m_i-1}$-singularity as desired.

The second case is argued similarly. Suppose that $S_1, \ldots, S_l$ is a partition of $[k]$, and that two limit points $\alpha_i(0)$ and $\alpha_j(0)$ coincide iff $i,j$ lie in the same subset of the partition. Arguing exactly as above, we conclude that the limit points $\beta_i(0)$ coincide according to the same partition, and that each of the limit points $\{\alpha_i(0)\}_{i=1}^{k}$ and $\{\beta_i(0)\}_{i=1}^{k}$ is smooth. Thus, $\tilde{\C}|_0 \rightarrow \C|_0$ is a full normalization of $\C|_0$ at $\{\tau_{i}(0)\}_{i=1}^{k}$, and genus considerations force $\tau_i(0)$ to be an $A_{2\sum_{j \in S}m_j-1}$-singularity (where $S$ is the unique subset containing $i$).
\end{proof}

\begin{lem}\label{L:Openness}
\begin{enumerate}
\item[]
\item $T_1^k \subset \U_{g,n}(A_4)$ is closed for any $k \in \{0,1,3\}$.
\item $B_1^{k,l} \subset \U_{g,n}(A_4)-\bigcup_{i \in \{0,1,3\}}T_1^i$ is closed for any $k,l \in \{0,1,4\}$.
\item $T_2^k \subset \U_{g,n}(A_4)-\bigcup_{i \in \{0,1,3\}}T_1^i-\bigcup_{i,j \in \{0,1,4\}} B_1^{i,j}$ is closed.
\end{enumerate}
\end{lem}

\begin{proof}
The loci $T_{m}^k$ and $\B_{m}^{k_1, k_2}$ are obviously constructible, so it suffices in each case to show that they are closed under specialization.

For (1), let $(\pi: \C \rightarrow \Delta, \sigman)$ be the family in $\U_{g,n}(A_4)$ such that the generic fiber lies in $T_{1}^{k}$ with $k \in \{0,1,3\}$. Note that the case $k=0$ is vacuous unless $g=n=1$ in which case $T_1^1=\U_{1,1}(A_4)$ and the statement is obvious, so we may assume $k \in \{1,3\}$. We must show that special fiber lies in $T_1^k$. Possibly after  a finite base change, $\pi$ admits a section $\tau_1$ picking out the attaching $A_k$-singularity of an $H_{1,1}$-tail in the generic fiber. By Lemma \ref{L:LimitNode}, the limit point $\tau_1(0)$ is still an $A_{k}$-singularity and the normalization of $\tilde{\C} \rightarrow \C$ induces a simultaneous normalization of the family. Let $\H \subset \tilde{\C}$ be the component whose generic fiber is a smooth $H_{m,1}$-curve, and let $\alpha_1$ be the section on $\H$ lying over $\tau_1$. We may consider $(\H \rightarrow \Delta, \alpha_1)$ as a family in $\U_{1,1}(A_{4})$ whose generic fiber is a smooth $H_{1,1}$-curve. By Lemma \ref{L:HmLimits} (1), $(H_0, \alpha_1(0))$ is an $H_{1,1}$-curve, so $C_0$ contains an $H_{1,1}$-tail with $A_{k}$-attaching, as desired. The proof for (3) is identical, using Lemma \ref{L:HmLimits} (3) instead of Lemma \ref{L:HmLimits} (1).

For (2), let $(\pi: \C \rightarrow \Delta, \sigman)$ be a family in $\U_{g,n}(A_4)$ such that the general fiber lies in $B_{1}^{1,1}$, i.e. contains an elliptic chain. (The cases where the general fiber lies in $B_1^{i,j}$ with other combinations $i,j \in \{0,1,4\}$ are essentially identical, and we leave the details to the reader.) Possibly after a finite base change, there exist sections $\tau_1$, $\tau_2$ picking out the attaching nodes of the elliptic chain in the general fiber. We claim that $\tau_1(0)$ and $\tau_2(0)$ are distinct, so that the normalization $\tilde{\C} \rightarrow \C$ along $\tau_1$ and $\tau_2$ gives a simultaneous normalization of the fibers. By Lemma \ref{L:LimitNode}, it suffices to check that if $\alpha_1$ and $\alpha_2$ are the sections of $\tilde{\C}$ lying on the connected component of $\tilde{\C}$ representing the elliptic chain, then $\alpha_1(0) \neq \alpha_2(0)$. If $\alpha_1$ and $\alpha_2$ lie on different irreducible components of the general fiber, then this is clear; otherwise, the entire elliptic chain is a single irreducible curve of arithmetic genus one and $\alpha_1(0)=\alpha_2(0)$ would force the special fiber to have an elliptic curve meeting the rest of the fiber in a single point - a contradiction, since we are assuming the special fiber contains no elliptic tails. Thus, $\alpha_1(0) \neq \alpha_2(0)$ as desired.

Now, let $(\H \rightarrow \Delta, \alpha_1, \alpha_2)$ be the connected component of $\tilde{\C}$ whose generic fiber is the given elliptic chain; we must show that the special fiber is an elliptic chain. If the chain has length $r$, then there exist sections $\gamma_1, \ldots, \gamma_{r-1}$ picking out the tacnodes in the general fiber at which the sequence of $H_{1,2}$-curves are attached to each other. Applying Lemma \ref{L:LimitNode} to each of the sections $\gamma_i$ individually, we conclude that the limits $\gamma_1(0), \ldots, \gamma_{r-1}(0)$ remain tacnodes, so the normalization of $\H$ along $\gamma_1, \ldots, \gamma_{r-1}$ induces a simultaneous normalization of the fibers, and we obtain $r$ distinct flat families whose generic fiber is a $H_{1,2}$-curve. It suffices to see that all these remain $H_{1,2}$-curves in the special fiber. This follows immediately from Lemma \ref{L:HmLimits} (3), since we are assuming the special fiber has no elliptic tails.
\end{proof}

\begin{prop}\label{P:Openness}
For $k \in\{2,3,4\}$, $A_k$-stability and $A_k^+$-stability are deformation open conditions.  \end{prop}
\begin{proof}[Proof of Proposition \ref{P:Openness}]
Using the description of $\SM_{g,n}(A_k)$ and  $\SM_{g,n}(A_k^+)$ in the discussion following Definition \ref{definition-attaching}, we see that  Lemma \ref{L:Openness}(1) implies $\SM_{g,n}(A_2^+)$ and $\SM_{g,n}(A_3)$ are obtained by excising closed subsets of $\U_{g,n}(A_2)$ and $\U_{g,n}(A_3)$ respectively. Similarly, Lemma \ref{L:Openness}(2) implies that $\SM_{g,n}(A_3^+)$ and $\SM_{g,n}(A_4)$ are obtained by excising closed subsets of $\U_{g,n}(A_3)$ and $\U_{g,n}(A_4)$. Finally, Lemma \ref{L:Openness}(3) implies that $\SM_{g,n}(A_4^+)$ is obtained by excising a closed subset from $\U_{g,n}(A_4)$.
\end{proof}

\begin{corollary}\label{C:Openness}
  $\SM_{g,n}(A_{k}^-),\SM_{g,n}(A_{k}),\SM_{g,n}(A_{k}^+)$ are algebraic stacks of finite-type over $\Spec \CC$. Moreover, that natural inclusions, 
$$
\SM_{g,n}(A_{k}^-) \hookrightarrow \SM_{g,n}(A_{k}) \hookleftarrow \SM_{g,n}(A_{k}^+).
$$
are open immersions.

 \end{corollary}
 \begin{proof}
 Families of  $A_{k}^{-}/A_{k}/A_{k}^+$-stable curves satisfy \'etale descent since they are canonically polarized. Once we know these are deformation-open conditions, we can use an open subset of a suitable Hilbert scheme to give an atlas. The obvious set-theoretic inclusions give rise to the desired open immersions.
 \end{proof}

\section{Closed points of $\SM_{g,n}(A_k)$} \label{section-closed-points}
Throughout this section, we assume $k \in \{2,3,4\}$.

\subsection{The canonical decomposition of an $A_{k}$-stable curve}

\begin{lemma}\label{L:Disjoint}
Suppose that $(C,\pn)$ is an $A_k$-stable curve.
\begin{enumerate}
 \item If $k=2m$ is even and $i_1(E_{1}), i_2(E_{2}) \subset C$ are the images of two distinct destabilizing $H_{m,1}$-tails, then $i_1(E_1)$ and $i_2(E_2)$ have no component in common.
  \item If $k=2m+1$ is odd and $i_1(E_{1}), i_2(E_{2}) \subset C$ are the images of two distinct destabilizing $H_{m,2}$-chains, then $i_1(E_1)$ and $i_2(E_2)$ have no component in common.
 \end{enumerate}
\end{lemma}
\begin{proof}
Case (1) is obvious, since any $H_{m,1}$-tail is irreducible. For case (2), consider two distinct destabilizing $H_{m,2}$-chains:
\begin{align*}
i_1:&(E_1, q_1, q_2) \hookrightarrow C,\\
i_2:&(E_2,r_1,r_2) \hookrightarrow C.
\end{align*}
Note that since $C$ contains no $A_{l}$-singularities with $l>2m+1$, the attaching points $i_1(q_1), i_1(q_2), i_2(r_1), i_2(r_2)$ are either nodes or marked points of $C$. We claim that if $i_1(E_1) \subset C$ and $i_2(E_2) \subset C$ share a common component, we may assume, without loss of generality, that $i_1(q_1)$ is a node internal to $E_2$. To see this, consider three cases:
\begin{enumerate}
\item If $i_2(r_1)=i_2(r_2) \in C$ is a node, then $C \simeq E_2$ and $C$ has no marked points. Now if $E_1$ and $E_2$ are not identical, then one of the attaching points $i_1(q_1), i_2(q_2)$ must be a node internal to $E_2$. Without loss of generality, we may assume this attaching point is $i_1(q_1)$.
\item If $i_2(r_1)$ and $i_2(r_2)$ are marked points of $C$, then $C$ has exactly two marked points and $(C,p_1,p_2) \simeq (E_2,r_1,r_2)$. Obviously, if $E_1$ and $E_2$ are distinct, then one of the attaching points of $E_1$ must be a node, and this node is internal to $E_2$.
\item If $i_2(r_1)$ and $i_2(r_2)$ are distinct nodes of $C$ or a node and a marked point, then either $E_1$ contains a nodal attaching point internal to $i_2(E_2)$, or else $i_1(E_1)$ contains $i_2(E_2)$ entirely, in which case we just switch the role of $E_1$ and $E_2$.
\end{enumerate}
Now let $p:=i_1(q_1)$. Evidently, $p \in C$ must be a node adjacent to two irreducible components, say $Z_1, Z_2 \in C$, both of which are contained in $E_2$. Since the only nodes internal to an $H_{m,2}$-chain are contained within a single $H_{m,2}$-bridge, we must have $Z_1 \cup Z_2$ constituting an $H_{m,2}$-bridge inside $E_2$. By Corollary \ref{C:Akcurves}, each of $Z_1$ and $Z_2$ is smooth rational. Furthermore, the fact that $p \in C$ is in attaching point for the $H_{m,2}$-chain $E_1$ implies that $Z_1$ meets $Z_2$ \emph{only} at $p$. We conclude that $Z_1 \cup Z_2$ has arithmetic genus zero, a contradiction since the genus of an $H_{m,2}$-bridge is $m>0$. We conclude that $i_1(E_1)$ and $i_2(E_2)$ may not have components in common.
\end{proof}

Using Lemma \ref{L:Disjoint}, we may define a decomposition of an $A_k$-stable curve into its $H_{m,1}$-tails/$H_{m,2}$-chains and a complementary subcurve which we shall call the core. The $k$ even case is considerably easier than the $k$ odd case, so we state them separately.

\begin{figure}
\scalebox{1.15}{\includegraphics{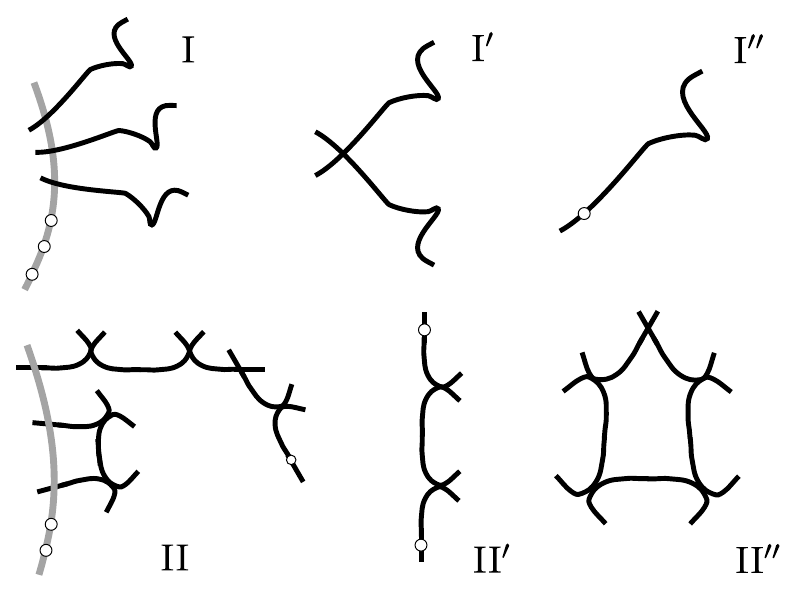}}
\caption{Canonical decompositions for an $A_k$-stable curve. The core of $C$ is shown in grey (cases I, II). Each black component represents an $H_{m,1}$-curve (cases I, I$'$, I$''$) or $H_{m,2}$-curve (cases II, II$'$, II$''$).}\label{F:CanonicalDecompositions}
\end{figure}

\begin{definition}[Canonical Decomposition - $k$ even] \label{canonical-decomposition-even}
Suppose that $(C,\pn)$ is $A_k$-stable ($k=2m$). Then one of the following holds:
\begin{itemize}
\item[] \textbf{Case I:} The $H_{m,1}$-tails of $C$ are disjoint, and we have a decomposition
$$C=K \cup E_1 \cup \ldots \cup E_r,$$
where $E_1, \ldots, E_r$ are the nodally-attached $H_{m,1}$-tails of $C$, $K:=\overline{C \backslash (E_1 \cup \ldots \cup E_r)}$, and each $E_i$ meets $K$ at a single node $q_i \in C$. Furthermore, if we consider $q_1, \ldots, q_r$ as marked points on $K$, then $(K, \pn, \q{r})$ is an $A_{k}$-stable curve with no destabilizing $H_{m,1}$-tails.  We call $(K, \pn, \q{r})$ the \emph{core of $(C, \pn)$}. In cases I$'$ and I$''$ below, we say that the core is empty.\\

\item[] \textbf{Case I$'$:} $C = E_1 \cup E_2,$
 where $(E_1,q_1)$ and $(E_2,q_2)$ are each $H_{m,1}$-curves, attached nodally at $q_1=q_2 \in C$.\\
\item[] \textbf{Case I$''$:} $(C, p_1)=(E,q_1)$ where $(E_1,q_1)$ is an $H_{m,1}$-curve.
\end{itemize}

\end{definition} 
To see that every $A_k$-stable curve $(C,\pn)$ satisfies I,I$'$, or I$''$, simply let $E_1, \ldots, E_r$ be the destabilizing $H_{m,1}$-tails of $C$. If the union of these tails comprise all of $C$, we must be in case I$'$ or I$''$. Otherwise, each $E_i$ meets $K:=:=\overline{C \backslash (E_1 \cup \ldots \cup E_r)}$ in a single node, and the fact that $(K, \pn, \q{r})$ is $A_k$-stable is immediate from the definition of $A_k$-stability.

To define the canonical decomposition in the $k$ odd case, we need a preliminary definition:
\begin{definition}[$H_{m,2}$-link]
We say that $C$ contains an \emph{$H_{m,2}$-link of length $r$} ($k=2m+1$) if there exists a gluing morphism
$$i: \bigcup_{i=1}^{r}(E_i,q_{2i-1},q_{2i}) \hookrightarrow C,$$
satisfying:
\begin{enumerate}
\item $(E_i,q_{2i-1},q_{2i})$ is an $H_{m,2}$-chain for $i=1, \ldots, r$.
\item $i(q_{2i})=i(q_{2i+1})$ is a node for $i=1, \ldots, r-1$.
\item $i(q_1)$, $i(q_{2r})$ are nodes, marked points, or $A_l$-singularities with $l > 2m +1$.
\end{enumerate}
\end{definition}
\begin{remark}
An $H_{m,2}$-chain is the same things as an $H_{m,2}$-link of length one.
\end{remark}

\begin{definition}[Canonical Decomposition - $k$ odd]\label{canonical-decomposition-odd}
Suppose that $(C,\pn)$ is $A_{k}$-stable with $k=2m+1$. Then one of the following holds:
\begin{itemize}

\item[] \textbf{Case II: } We have a decomposition
$$C=K \cup E_1 \cup \ldots \cup E_r \cup E_{r+1} \cup \ldots \cup E_{r+s}$$
where $E_1, \ldots, E_r$ are $H_{m,2}$-links meeting $K$ at two distinct nodes, $E_{r+1}, \ldots, E_s$ are $H_{m,2}$-links meeting $K$ in a single node (and whose other endpoint is a marked point of $C$), and $K:=\overline{C \backslash (E_1 \cup \ldots \cup E_s)}$. If we denote the points where $E_i$ meets $K$ by $\{q_{2i-1},q_{2i}\}$ $(i=1, \ldots, r)$, the point where $E_{r+i}$ meets $K$ by $q_{2r+i}$ $(i=1, \ldots, s)$, and the single marked point of $E_{r+i}$ by $p_{n-s+i}$, then the connected components of $(K, \{p_i\}_{i=1}^{n-s}, \q{2r+s})$ are $A_{k}$-stable curves with no destabilizing $H_{m,2}$-chains. We call $(K, \{p_i\}_{i=1}^{n-s}, \q{2r+s})$ the  \emph{core of $(C, \pn)$}.  In cases II$'$ and II$''$ below we say that the core of $(C, \pn)$ is empty.\\

\item[] \textbf{Case II$'$: } $(C,p_1,p_2)$ is an $H_{m,2}$-link, whose endpoints are marked points, i.e.
$$(C, p_1, p_2)=(E_1 \cup \ldots \cup E_r, q_1, q_{2r}),$$
where each $(E_i,q_{2i-1},q_{2i})$ is an $H_{m,2}$-chain, attached to each other nodally via $q_{2i} \sim q_{2i+1}$ for $i=1, \ldots, r-1$.\\

\item[] \textbf{Case II$''$: } $C$ is an $H_{m,2}$-link, whose endpoints are identified in a single node, i.e.
$C = E_1 \cup \ldots \cup E_r,$
where each $(E_i,q_{2i-1},q_{2i})$ is an $H_{m,2}$-chain, attached to each other nodally via $q_{2i} \sim q_{2i+1}$ for $i=1, \ldots, r-1$ and $q_1 \sim q_{2r}$.
\end{itemize}

\end{definition} 
To see that every $A_k$-stable curve $(C,\pn)$ satisfies I,I$'$, or I$''$, let $Z \subset C$ be the union of all $H_{m,2}$-chains of $(C,\pn)$, and note that each connected component of $Z$ is an $H_{m,2}$-link. If $Z$ comprises all of $C$, then we are in case I$'$ or I$''$. Otherwise, each connected component of $Z$ meets the complement $K:=\overline{C \backslash Z}$ in one or two nodes, and we are in case II.

The following lemma, which will be used repeatedly in the sequel, allows us to decompose one-parameter families of $A_k$-stable curves, according to the canonical decomposition of their generic fiber.
\begin{lemma}\label{L:LimitCanDecomp}
Let $(\C \rightarrow \Delta, \sigman)$ be a family of the $A_k$-stable curves, and let 
$$C_{\bar{\eta}}=K \cup E_1 \cup \ldots \cup E_r \text{ (resp. $K \cup E_1 \cup \ldots \cup E_{r+s}$)}$$
be the canonical decomposition of the generic fiber of $\C$ in the case $k$ even (resp. $k$ odd). After a finite base change, we may assume there exist sections $\{\tau_{i}\}_{i=1}^{r}$ (resp.  $\{\tau_{i}\}_{i=1}^{2r+s})$ which pick out the attaching nodes $E_i \cap K$ of the canonical decomposition of the generic fiber. We claim that the limits $\tau_i(0)$ are all nodes, and consequently we may consider the pointed normalization $\tilde{\C} \rightarrow \C$ of $\C$ along the union of the $\tau_i$. Then $\tilde{\C}$ decomposes as
$$\tilde{\C}=\K \cup \E_1 \cup \ldots \cup \E_r \text{(resp. $\K \cup \E_1 \cup \ldots \cup \E_s$)}$$
where $\K, \E_1, \ldots, \E_r$ (resp. $\K \cup E_1 \cup \ldots \cup \E_s$) are families of $A_k$-stable curves with generic fibers $K, E_1, \ldots, E_r$ (resp. $K \cup E_1 \cup \ldots \cup E_s$).
\end{lemma}
\begin{proof}
The only fact which needs to be proved is that the limits $\tau_i(0)$ are indeed nodes. In the case $k$ even, the fact that each limits $\tau_i(0)$ is a node is an immediate application of Lemma \ref{L:LimitNode}.
In the case $k$ odd, if $E_{r+i}$ is an $H_{m,2}$-link attached to $K$ at a single node at $\tau_{2r+i}$ picks out this attaching node in the general fiber, then the fact that $\tau_{2r+i}(0)$ is a node is again an immediate application of Lemma \ref{L:LimitNode}.

It only remains to show that $\tau_{2i-1}(0)$ and $\tau_{2i}(0)$ are nodes, in the case where $\tau_{2i-1}$ and $\tau_{2i}$ are the attaching sections of an $H_{m,2}$-link $E_i$ attached to $K$ at two nodes. Let $\tilde{\C} \rightarrow \C$ be the normalization of $\C$ along $\tau_{2i-1}$ and $\tau_{2i}$, let $\E$ be the connected component of $\tilde{\C}$ with $E_i$ in the general fiber and let $\alpha_{2i-1}$ and $\alpha_{2i}$ be the preimages of $\tau_{2i-1}$ and $\tau_{2i}$ on $\E$. According to Lemma \ref{L:LimitNode}, we only need to show that $\alpha_{2i-1}$ and $\alpha_{2i}$ do not collide in the special fiber. If $E_i$ consists of more than two $H_{m,2}$-bridges, then $\alpha_{2i-1}$ and $\alpha_{2i}$ lie on non-adjacent irreducible components of $\E$ and hence cannot possibly collide. If $E_i$ consists of two $H_{m,2}$-bridges meeting at an $A_{2m+1}$ singularity, then $\alpha_{2i-1}$ and $\alpha_{2i}$ can collide only if they collide with the limit of this $A_{2m+1}$ singularity; but this is impossible since the limit of the $A_{2m+1}$ singularity must be an $A_{2m+1}$ singularity. Finally, if $E_i$ consists of a single $H_{m,2}$-bridge and $\alpha_{2i-1}$ and $\alpha_{2i}$ collide, then the special fiber of $\E$ would be an arithmetic genus $m$ curve attached to the rest of the special fiber of $\C$ at a single point. Here, we invoke the fact that the only odd $k=2m+1$ under consideration is $k=3$ (equivalently, $m=1$) and an $A_3$-stable curve may not contain elliptic tails with any form of attaching. This completes the proof of the lemma.
\end{proof}

\subsection{Characterization of closed points of $\SM_{g}(A_{k})$}

\begin{defn} We say that an $A_k^-/A_k/A_k^+$-stable curve $(C, \pn)$ is \emph{closed} if $(C, \pn)$ is a closed point of $\bar{\cM}_{g,n}(A_k^-)/\bar{\cM}_{g,n}(A_k)/\bar{\cM}_{g,n}(A_k^+)$.
\end{defn}

\begin{definition}
We say that an $A_{k}$-stable curve $(C,\pn)$ with $k = 2m$ (resp. $k =2m+1$) is \emph{maximally degenerate} if the following conditions hold:
\begin{enumerate}
\item Every $A_{k}$-singularity of $C$ lies on a nodally-attached $H_{m,1}$-tail (resp. $H_{m,2}$-bridge).
\item Every nodally-attached $H_{m,1}$-tail (resp. $H_{m,2}$-bridge) of $C$ is monomial.
\item The core of $(C,\pn)$ is a closed $A_k^-$-stable curve.
\end{enumerate}
\end{definition}
\begin{remark}
\begin{enumerate}
\item[]
\item If $(C, \pn)$ satisfies the first two conditions, then the core of $C$ contains no $A_{k}$-singularities. Thus, the core of a maximally degenerate $A_{k}$-stable curve is $A_{k}^-$-stable, and it makes sense to require the third condition.
\item In a maximally degenerate $A_{k}$-stable curve $C$ ($k=2m+1$), condition (2) implies that two $H_{m,2}$-bridges can only meet in nodes. In particular, $C$ only admits $H_{m,2}$-chains of length one.
\end{enumerate}
\end{remark}

The goal of this section is to prove that an $A_k$-stable curve is closed iff it is maximally degenerate. One direction, namely that a closed $A_k$-stable curve must be maximally degenerate, is contained in the following lemma.

\begin{lemma}\label{L:AMD}
Every $A_{k}$-stable curve admits an isotrivial specialization to a maximally degenerate $A_{k}$-stable curve.
\end{lemma}
\begin{proof}
Given any $A_k$-stable curve $(C, \pn)$, we must construct an isotrivial specialization to an $A_k$-stable curve satisfying (1), (2), and (3) in three steps. First, let us construct an isotrivial specialization $C \rightsquigarrow C'$ where $C'$ is an $A_k$-stable curve satisfying (1). Let $q_1, \ldots, q_r \in C$ be the $A_{k}$-singularities of $C$. Consider the trivial family $C \times \Delta$ and let $\C \rightarrow C \times \Delta$ denote the normalization of $C \times \Delta$ along $\cup_{i=1}^{r}(q_i \times \Delta)$ and let $\{\sigma_{i}\}_{i=1}^{r}$ (resp., $\{\sigma_{i}, \sigma_i'\}_{i=1}^{r}$) denote the sections of $\C \rightarrow \Delta$ lying above the $A_{k}$-singular locus in the case $k=2m$ is even (resp., $k=2m+1$ odd). Now let $\tilde{\C} \rightarrow \C$ denote the blow-up of $\C$ at the smooth points $\{\sigma_i(0)\}_{i=1}^{r}$ ($\{\sigma_{i}(0), \sigma_i'(0)\}_{i=1}^{r}$), and let $\tilde{\sigma}_{i}$  (resp., $\{\tilde{\sigma}_{i}, \tilde{\sigma}_i'\}_{i=1}^{r}$) denote the strict transforms of the sections. Note that the special fiber $\tilde{C}$ now decomposes as $C \cup Z_{1} \cup \ldots \cup Z_{r}$ (resp., $C \cup Z_{1}\cup Z_1' \cup \ldots \cup Z_{r} \cup Z_r'$), where each $Z_i$ is a smooth $\P^{1}$ meeting $C$ in a single node, with $\tilde{\sigma}_{i}(0) \in Z_{i}$. Now let $\tilde{\C} \rightarrow \C^{'}$ denote the map obtained by crimping the sections $\{\tilde{\sigma_i}\}_{i=1}^{n}$ (resp., $\{\tilde{\sigma}_{i}, \tilde{\sigma}_i'\}_{i=1}^{r}$) back to $A_{k}$-singularities. (For the fact that the limit singularity is again an $A_k$-singularity, cf. the proof of Proposition \ref{prop-S-weakly-proper}.) Now the family $\C^{'}$ is an isotrivial specialization $C \rightsquigarrow C^{'}$ in which the curve $C^{'}$ has sprouted $H_{m,1}$-curves (resp. $H_{m,2}$-bridges) at $q_1, \ldots, q_r$. The only problem is that this procedure may have introduced semistable $\P^{1}$'s into the special fiber, i.e. smooth $\P^{1}$'s meeting the rest of the fiber in two nodes. Since, however, one or both of these two attaching nodes must smooth in the general fiber, we can blow-down all these semistable $\P^{1}$'s to obtain a new special fiber which is $A_k$-stable and satisfies (1).

Next, we will construct an isotrivial specialization $C' \rightsquigarrow C''$ where $C''$ is an $A_k$-stable curve and satisfies (1) and (2). For this, it suffices to isotrivially specialize all the $H_{m,1}$-tails (resp., $H_{m,2}$-bridges) of $C^{'}$ to monomial $H_{m,1}$-tails (resp., $H_{m,2}$-bridges). More precisely, if
$$
C^{'}=K^{'} \cup E_1 \cup \ldots \cup E_{r}
$$ 
is the core decomposition of $C^{'}$, then by Proposition \ref{prop-H-S-minimal}, there exist isotrivial families $\E_i \rightarrow \Delta$ with generic fiber isomorphic to $E_i$ and special fiber isomorphic to the monomial $H_{m,1}$-curve (resp. $H_{m,2}$-curve). Gluing these to the trivial family $K \times \Delta$ gives the desired isotrivial specialization $C' \rightsquigarrow C''$.

Finally, we  will construct an isotrivial specialization $C'' \rightsquigarrow C'''$, where $C'''$ satisfies (1), (2), and (3). Indeed, if the core $K''$ of $C''$ is not closed, then there exists an isotrivial specialization $\K \rightarrow \Delta$ with general fiber $K'$ and special fiber a closed $A_k^-$-stable curve. Gluing this specialization to trivial families of monomial $H_{m,1}$-curves (resp., $H_{m,2}$-curves) gives the desired isotrivial specialization $C'' \rightsquigarrow C'''$.
\end{proof}

Before we can show the converse direction, namely that any maximally degenerate $A_k$-stable curve is closed, we will need a few lemmas.

\begin{lemma}\label{L:AkMinusClosed}
Let $(C,\pn)$ be an $A_k^-$-stable curve, and let $(K,\pn,\qm)$ be the core of $C$, which is defined by considering $(C, \pn)$ as an $A_k$-stable curve. Then $(C,\pn)$ is a closed $A_k^-$-stable curve iff $(K, \pn, \qm)$ is a closed $A_k^-$-stable curve.
\end{lemma}
\begin{proof}
We shall prove the lemma in the case $k=2m$ is even (the proof in the $k$ odd case is similar). Let $(\C \rightarrow \Delta, \sigman)$ be any isotrivial specialization in $\SM_{g}(A_k^-)$ and apply Lemma \ref{L:LimitCanDecomp} to decompose the normalization of $\C$ along the attaching nodes of the canonical decomposition of the general fiber:
$$
\tilde{\C}=\K \cup \E_1 \cup \ldots \E_r,
$$
where $\K$ is the connected component of $\tilde{\C}$ containing the core of the generic fiber and $\E_1, \ldots, \E_r$ comprise the $H_{m,1}$-tails of the general fiber. To prove the lemma, we simply need to show that each of the isotrivial specializations $\E_i$ is trivial. Equivalently, that any $H_{m,1}$-curve is a closed point of $\SM_{m,1}(A_k^-)$ but this is clear as any $H_{m,1}$-curve is closed in $\bar{\cH}_{m,1}$ minus the unique point corresponding to a monomial $H_{m,1}$-curve. 
\end{proof}

\begin{lemma}\label{L:AlmostClosed}
Suppose $(C, \pn)$ is closed point in $\SM_{g,n}(A_k^-)$ and that $(C, \pn)$ has no nodally-attached $H_{m,1}$-tail/$H_{m,2}$-bridge ($k=2m/2m+1$). Then $(C, \pn)$ is closed point in $\SM_{g,n}(A_k)$.
\end{lemma}
\begin{proof}
Let $(\C \rightarrow \Delta, \sigman)$ be any isotrivial specialization in $\SM_{g,n}(A_k)$ with geometric generic fiber isomorphic to $(C,\pn)$ and special fiber closed in $\SM_{g,n}(A_k)$. We must show that this specialization is trivial. Since we are assuming that $(C, \pn)$ is closed in $\SM_{g,n}(A_k^-)$, it is sufficient to prove that the special fiber lies in $\SM_{g,n}(A_k^-)$, i.e. has no $A_k$-singularities. Suppose, on the contrary, that the special fiber contains an $A_k$-singularity; we will obtain a contradiction by showing that the generic fiber would then necessarily contain a nodally-attached $H_{m,1}$-tail/$H_{m,2}$-bridge ($k=2m/2m+1$).

By induction and Theorem \ref{main-theorem}, we may assume that $\SM_{g,n}(A_k^-)$ is weakly proper, so there exists a family $(\C' \rightarrow \Delta, \sigman)$ such that $(\C')^* \simeq \C^*$ and with central fiber $C'$ a closed point of $\SM_{g,n}(A_k^-)$. By assumption, the generic fiber of $\C'$ is a closed point of $\SM_{g,n}(A_k^-)$, so the isotrivial specialization $\C'$ must be trivial, i.e. the generic fiber of $\C$ is isomorphic to the special fiber $C'$. We will compute the $A_k^-$-stable limit $C'$ explicitly, and show that if $C$ contains $A_k$-singularities, then $C'$ necessarily contains $H_{m,1}$-tails/$H_{m,2}$-bridges ($k=2m/2m+1$). This will yield the desired contradiction.

By Lemma \ref{L:AMD}, the closed points of $\SM_{g,n}(A_k)$ are maximally degenerate. Thus, all $A_k$-singularities of $C$, say $q_1, \ldots, q_r$, lie on monomial $H_{m,1}$-curves (resp. $H_{m,2}$-curves), there are no $H_{m,2}$-chains, and the core of $C$ is maximally degenerate as an $A_k^-$-stable curve. We claim that the $A_k^-$-stable limit is then obtained in two steps as follows. By Lemma \ref{L:BlowUp} below, there exists (after a suitable base-change) a single weighted blow-up $\phi: \tilde{\C} \rightarrow \C$ centered over $q_1,\ldots, q_r$ such that $\phi^{-1}(q_i)$ is an $H_{m,1}$-curve (resp. $H_{m,2}$-curve). Note that $\phi^{-1}(q_i)$ meets the special fiber $\tilde{C}$ in one (resp. two) semistable $\P^1$'s in the case $k$ even (resp. odd). Let $\tilde{\C} \rightarrow \C'$ be the blow-down of these semistable $\P^1$'s. By Lemma \ref{L:AkMinusClosed}, $C'$ is a maximally degenerate $A_k^-$-stable curve. Since $C'$ was obtained by replacing $A_k$-singularities of $C$ by nodally attached $H_{m,1}/H_{m,2}$-curves, we are done. 
\end{proof}

\begin{lemma}(\cite[Proposition 7.2]{fedorchuk})\label{L:BlowUp}
Let $\X \rightarrow T$ be a miniversal deformation of an $A_{k}$-singularity. Then there exists an alteration $f: T'\rightarrow T$ and a weighted blow-up $\Y \rightarrow \X\times_{T}T'$ of the $A_{k}$-locus (i.e. the locus of $A_k$-singularities of the fibers of $\X \times_T T' \rightarrow T'$) such that
\begin{enumerate}
\item[(1)] $\Y \rightarrow T'$ is a flat and proper family of curves with at-worst $A_{k-1}$-singularities. \item[(2)] $\Y \vert_{f^{-1}(0)}\rightarrow f^{-1}(0)$ is a family of $H_{m,1}$-tails if $k=2m$ is even (resp. $H_{m,2}$-bridges if $k=2m+1$ is odd). 
\end{enumerate}
\end{lemma}   

\begin{prop}[Closed Points of $\SM_{g,n}(A_k)$]\label{P:ClosedPoints}
An $A_k$-stable curve $(C,\pn)$ is closed if and only if $(C, \pn)$ is maximally degenerate.
\end{prop}
\begin{proof}[Proof of Proposition \ref{P:ClosedPoints}]
The fact that a closed $A_k$-stable curve is maximally degenerate follows immediately from Lemma \ref{L:AMD}. It remains to show that a maximally degenerate curve $(C, \pn)$ is closed. Given a maximally degenerate curve $(C, \pn)$, let $(\C \rightarrow \Delta, \sigman)$ be any isotrivial specialization which is isomorphic to $(C,\pn) \times \Delta^*$ over $\Delta^*$. We must show that the central fiber is isomorphic to  $(C, \pn)$ as well.

As in Lemma \ref{L:LimitCanDecomp}, we may decompose $\C$ along the locus of attaching nodes of the canonical decomposition of the general fiber to obtain:
$$
\tilde{\C}=\K \cup \E_1 \cup \ldots \E_r,
$$
where $\K$ is the connected component of $\tilde{\C}$ containing the core of the generic fiber and $\E_1, \ldots, \E_r$ comprise the $H_{m,1}$-tails ($H_{m,2}$-bridges) of the general fiber. Since the monomial $H_{m,1}$-tail (resp., $H_{m,2}$-bridge) is the unique closed point of $\bar{\H}_{m,1}/\bar{\H}_{m,2}$ (Proposition~\ref{prop-H-S-minimal}) and $\bar{\H}_{m,1} \subset \SM_{m,1}(A_k)$ (resp., $\bar{\H}_{m,2} \subset \SM_{m,2}(A_k)$) is closed by Lemma \ref{L:HmLimits}, the isotrivial specializations $\E_i$ must be trivial. Thus, to show that the entire specialization $\C \rightarrow \Delta$ is trivial, it only remains to show that the isotrivial specialization $\K$ is trivial. By hypothesis, the generic fiber is a closed $A_k^-$-stable curve with no nodally attached $H_{m,1}$-tail (resp. $H_{m,2}$-bridge). By Lemma \ref{L:AlmostClosed}, it is closed as an $A_k$-stable curve as well, so the specialization $\K$ is trivial, as desired.
\end{proof}

A corollary of this characterization, which will be used in the proof of weak properness of $\SM_{g,n}(A_k)$, is that the isomorphism class of a closed $A_{k}$-stable curve is uniquely determined by its core.

\begin{corollary}[Isomorphism class determined by core]\label{C:CoreIso}
When $k$ is even, the isomorphism class of a maximally degenerate $A_{k}$-stable curve $(C, \pn)$ is determined by the isomorphism class of its core $(K, \q{r})$. When $k$ is odd, the isomorphism class of a maximally degenerate $A_{k}$-stable curve $(C, \pn)$ is determined by its core $(K, \{p_i\}_{i=s+1}^{n}, \q{2r+s})$,  along with the sequence of integer lengths $l_1, \ldots, l_r$ of the links of monomial $H_{m,2}$ curves connecting $q_{2i-1}$ to $q_{2i}$, and the lengths $l_1, \ldots, l_s$ of the link of monomial $H_{m,2}$ curves attached at $q_{2r+i}$.
\end{corollary}


\section{Deformation theory} \label{section-deformation-theory}

Let $(C,\pn)$ be an $A_k$-stable curve. We denote by
$\TT^1(C,\pn)$ the vector space of first-order deformations of
$(C,\pn)$. The \emph{deformation space} of $(C,\pn)$ is
\[
\Def(C,\pn) 
:= \Spec \Sym \TT^1(C,\pn)^\vee;
\]
the \emph{formal deformation space} of $(C,\pn)$ is
\[
\hat{\Def}(C,\pn) 
:= \Spf \hat{\cO}_{\Def(C,\pn),0}.
\]
Let $\Aut(C,\pn)$ denote the automorphism group scheme of $(C,\pn)$ 
and  $\Aut(C,$ $\pn)^{\star}$  the connected component of the identity of the subgroup  
consisting of automorphisms which restrict to the identity on the
core of $C$ (Definition \ref{canonical-decomposition-even} and 
Definition \ref{canonical-decomposition-odd}). 

In this section, we describe the natural action of the group
$\Aut(C,\pn)^{\star}$ on $\TT^1(C,$ $\pn)$ and more generally the action of
the group scheme $\Aut(C,\pn)^{\star}$ on $\hat{\Def}(C, $ $\pn)$. This
prepares the way for the proof of Proposition~\ref{theorem-etale-variation} in the next section, which says that, \'etale locally around closed points of $\SM_{g,n}(A_k)$, the
reduced closed substacks
\[
\bar{\cS}_{g,n}(A_k) :=
\bar{\cM}_{g,n}(A_k) \minus \bar{\cM}_{g,n}(A_k^-)
\]
and
\[
\bar{\cH}_{g,n}(A_k) := \bar{\cM}_{g,n}(A_k) \minus
\bar{\cM}_{g,n}(A_k^+)
\]
of $\bar{\cM}_{g,n}(A_k)$ correspond to the VGIT minus and plus chambers associated to the action of $\Aut(C,\pn)^{\star}$ on $\hat{\Def}(C,\pn)$. 

The first and key step is to write down the action of $\Aut(C,\pn)^{\star}$ on $\TT^1(C,\pn)$, where $(C,\pn)$ is a monomial $H_{m,1}/H_{m,2}$-curve. To do this, we prove that 
$$\TT^1(C,\pn)=T_{0}\cS_{m,1} \oplus T_{0}\cH_{m,1} 
\text{ (resp., $\TT^1(C,\pn) = T_{0}\cS_{m,2} \oplus T_{0}\cH_{m,2}$ ) }$$
and use the explicit descriptions of $\cS_{m,1}$ and $\cH_{m,1}$ (resp.,  $\cS_{m,1}$ and $\cH_{m,1}$) given in Section \ref{section-sh} to compute the action. Using the description of the closed points of $\SM_{g,n}(A_k)$ in Section 5, the description of the action $\Aut(C,\pn)^{\star}$ on for an arbitrary closed point boils down to the combinatorics of tails and links of monomial $H_{m,1}/H_{m,2}$-curve (Propositions \ref{first-order-action-even} and \ref{first-order-action-odd}). Finally, it is an easy formal argument to enhance this description to the level of the complete local ring $\hat{\Def}(C,\pn)$ (Propositions \ref{formal-action-even} and \ref{formal-action-odd}).

\subsection{First-order deformations}

Let $(E, q)$ (resp., $(E, q_1, q_2)$) be a 
monomial $H_{m,1}$-curve (resp., monomial $H_{m,2}$-curve).  Recall that $\xi \in E$ denotes the singular point. Denote by
\[
\Cr^1(E, q_1) \quad (\text{resp., } \Cr^1(E, q_1, q_2))
\]
the vector space of first-order deformations of
$(E, q_1)$ (resp., $(E, q_1, q_2)$)
which induce a trivial deformation of
$\hat{\cO}_{E,\xi}$, so that there is a short exact sequence
\begin{equation}\label{deformation-sequence}
\begin{aligned}
 &0 \arr \Cr^1(E,q_1)   \arr
        \TT^1(E,q_1)   \arr
        \TT^1(\hat{\cO}_{E,\xi}) \arr  0 \\
(\text{resp., } & 0 \arr \Cr^1(E,q_1, q_2)   \arr
        \TT^1(E,q_1, q_2)   \arr
        \TT^1(\hat{\cO}_{E,\xi}) \arr  0).
\end{aligned}
\end{equation}
$(E,q_1)$ (resp., $(E, q_1, q_2)$) defines a point of
$\overline{\cM}_{m,1}(A_{2m})$ (resp., $\overline{\cM}_{m,2}(A_{2m+1})$) contained in $
\overline{\cS}_{m,1} \cap \overline{\cH}_{m,1}$ (resp., $
\overline{\cS}_{m,2} \cap \overline{\cH}_{m,2}$) which we
denote by $0$. Here, we use the obvious identifications
\begin{align*}
\bar{\cS}_{m,1}(A_{2m})=\bar{\cS}_{m,1}
\quad&\text{and}\quad
\bar{\cS}_{m,2}(A_{2m+1})=\bar{\cS}_{m,2}\\
\bar{\cH}_{m,1}(A_{2m})=\bar{\cH}_{m,1}
\quad&\text{and}\quad
\bar{\cH}_{m,2}(A_{2m+1})=\bar{\cH}_{m,2},
\end{align*}
where $\bar{\cS}_{m,1}, \bar{\cS}_{m,2}, \bar{\cH}_{m,1}, \bar{\cH}_{m,2}$ are the stacks introduced in Section \ref{section-sh}.

\begin{lem}\label{splitting}
With notation as above, there are natural isomorphisms
$$\begin{aligned} 
\Cr^1(E,q_1) 	&	\cong\TT_{\overline{\cS}_{m,1},0} 	
& (\text{resp. }   \Cr^1(E,q_1, q_2)  	& \cong \TT_{\overline{\cS}_{m,2},0})\\
\TT^1(E,q_1) 	&	\cong \TT_{\overline{\cM}_{m,1}(A_{2m}),0} 	
& (\text{resp. }	 \TT^1(E,q_1, q_2)  	& \cong \TT_{\overline{\cM}_{m,2}(A_{2m+1}),0}) \\
\TT^1(\hat{\cO}_{E,\xi}) & \cong \TT_{\overline{\cH}_{m,1},0}
& (\text{resp. }   \TT^1(\hat{\cO}_{E,\xi}) \arr  0)  & \cong \TT_{\overline{\cH}_{m,2},0} )
\end{aligned}$$
such that map of tangent spaces
\[
\TT_{\overline{\cS}_{m,1},0} \arr
\TT_{\overline{\cM}_{m,1}(A_{2m}),0}
\quad
(\text{resp., } 
\TT_{\overline{\cS}_{m,2},0} \arr
\TT_{\overline{\cM}_{m,2}(A_{2m+1}),0})
\]
corresponds to the injection in 
\eqref{deformation-sequence} and the map of tangent spaces
\[
\TT_{\overline{\cH}_{m,1},0} \arr
\TT_{\overline{\cM}_{m,1}(A_{2m}),0}
\quad
(\text{resp., } 
\TT_{\overline{\cH}_{m,2},0} \arr
\TT_{\overline{\cM}_{m,2}(A_{2m+1}),0})
\]
induces a splitting of \eqref{deformation-sequence} which is
equivariant for the natural linear action of
$\Aut(E,q_1)$ (resp., $\Aut(E, q_1, q_2)$).
\end{lem}

\begin{proof}

Let $\nu \in \{1,2\}$.  Let $G:=\Aut(E,\{q_i\}_{i=1}^\nu))=\GG_m$. It suffices to
show that the composition
\[
\alpha:\ \TT_{\overline{\cH}_{m,\nu},0} \arr
\TT_{\overline{\cM}_{m,\nu}(A_{2m+\nu-1}),0} = 
\TT^1(E,\{q_i\}_{i=1}^\nu)) \arr
\TT^1(\hat{\cO}_{E,\xi}) 
\]
is a $G$-equivariant isomorphism. The isomorphism
$\bar{\cH}_{m,\nu} \isom [\AA^{2m+\nu-1}/\GG_m]$ of
Proposition~\ref{h-quotient-stack} identifies
$\TT_{\bar{\cH}_{m,\nu},0}$ with $\TT_{\AA^{2m+\nu-1},0}$, i.e.
\[
\TT_{\bar{\cH}_{m,\nu},0}
=\CC\langle a_0,\ldots,a_{2m+\nu-2}\rangle.
\]
On the other hand, noting that $\ch(\CC)=0$, an easy calculation
shows that
\[
\begin{array}{l}
\TT^1(\widehat{\cO}_{E,\xi}) =\\
\{
\Spec \CC[[x,y,\eps]]/
(y^2-x^{2m+\nu}-a_{2m+\nu-2}\eps x^{2m+\nu-2}
               -a_{2m+\nu-3}\eps x^{2m+\nu-3}
               -\ldots-a_{0}\eps,\eps^2) \\
\quad: a_{2m+\nu-2},\ldots,a_{0}\in\CC \}.
\end{array}
\]

We can describe explicitly the universal family of curves $\cC
\arr \AA^{2m+\nu-1}$ induced by the identification
$\bar{\cH}_{m,\nu} \isom [\AA^{2m+\nu-1}/\GG_m]$.  It comes with
an action of $\GG_m$ on $\cC$ compatible with the action of
$\GG_m$ on $\AA^{2m+\nu-1}$.
Choose coordinates $a_0,\ldots,a_{2m+\nu-2}$ on $\AA^{2m+\nu-1}$.
Let $S_{m+1}$ be the surface obtained by identifying $\Spec
\CC[x,y]$ and $\Spec \CC[u,v]$ along $D(x)$ and $D(u)$ via
$u=x^{-1}$ and $v=x^{-m-1}y$. ($S_{m+1}$ is the complement in
$\PP(\cO_{\PP^1}\oplus\cO_{\PP^1}(-m-1))$ of the section
corresponding to $\cO_{\PP^1}\ctdin
\cO_{\PP^1}\oplus\cO_{\PP^1}(-m-1)$.) 

If $\nu=1$, $\cC$ is the
subscheme of $S_{m+1} \times \AA^2_m$ defined by
\[
y^2-x^{2m+1}-a_{2m-1}x^{2m-1}-a_{2m-2}x^{2m-2}-\ldots-a_{0}
\]
and
\[
v^2-u-a_{2m-1}u^3-a_{2m-2}u^4-\ldots-a_{0}u^{2m+2}.
\]
The section of
$\cC \arr \AA^{2m}$ is
\[
u \mapsto 0, v \mapsto 0.
\]
The action of $\GG_m=\Spec \CC[t,t^{-1}]$ on $S_{m+1} \times
\AA^k$ given by
\[
x \mapsto t^{-2} x, y \mapsto t^{-2m-1} y,
u \mapsto t^2 u, v \mapsto t^ v.
\]
induces the action of $\GG_m$ on $\cC$.

If $\nu=2$, $\cC$ is the subscheme of $S_{m+1} \times \AA^{2m+1}$
defined by
\[
y^2-x^{2m+2}-a_{2m}x^{2m}-a_{2m-1}x^{2m-1}-\cdots-a_{0}
\]
and
\[
v^2-1-a_{2m}u^2-a_{2m-1}u^3-\cdots-a_{0}u^{2m+2}.
\]
The two sections of $\cC \arr \AA^k$ are
\[
u \mapsto 0, v \mapsto \pm 1.
\]
The action of $\GG_m=\Spec \CC[t,t^{-1}]$ on $S_{m+1} \times
\AA^k$ given by
\[
x \mapsto t^{-1} x, y \mapsto t^{-m-1} y,
u \mapsto t u, v \mapsto v,
\]
induces an action of $\GG_m$ on $\cC$.

It follows from the explicit description of the universal curve
over $\AA^{2m+\nu-1}$ that $\alpha$ is the identity map in the
coordinates described above, in particular an
isomorphism. Moreover, the actions of $G$ in these coordinates are
the same (cf. Lemma~\ref{first-order-action-even} and
Lemma~\ref{first-order-action-odd} below), so $\alpha$ is a
$G$-equivariant isomorphism as required.
\end{proof}

We are now ready to describe the action of $\Aut(C,\pn)^{\star}$ on the space of first-order deformations. We start with the case that $k=2m$ is even. According to
Definition~\ref{canonical-decomposition-even}, we have three
possible cases for the canonical decomposition of a maximally degenerate
$A_{2m}$-stable curve $(C, \pn)$:
\begin{description}
\item[Case I] $(C,\pn) = (K \unn E_1 \unn \ldots \unn E_r,\pn)$,
where for $1 \leq i \leq r$, $E_i$ is a monomial $H_{m,1}$-curve meeting $K$ in a single node $q_i$.\\
\item[Case I$'$] $C = E_1 \cup E_2$ where $(E_1,q_1)$ and $(E_2,q_2)$ are
monomial $H_{m,1}$-curves attached nodally via $q_1 \sim q_2$.\\
\item[Case I$''$] $(C, p_1)=(E,q_1)$ is a monomial
$H_{m,1}$-curve.
\end{description}

\begin{prop}\label{first-order-action-even}
Let $(C, \pn)$ be a maximally degenerate
$A_{k}$-stable curve with $k = 2m$.  
\begin{description}
\item[Case I] 
There exist decompositions
\begin{align*}
\Aut(C,\pn)^{\circ} &= 
\Aut(K,\pn,\qr)^{\circ} 
 \times \Aut(C,\pn)^{\star}\\
 &=\Aut(K,\pn,\qr)^{\circ} 
 \times \prod_{i=1}^r \Aut(E_i,q_i)
\end{align*}
and
\[
\TT^1(C, \pn) = \TT^1(K,\pn,\qr)
         \oplus \bigoplus_{i=1}^r 
         \left(\Cr^1(E_i,q_i)
         \oplus  \TT^1(\hat{\cO}_{E_i,\xi_i})\right)
         \oplus \bigoplus_{i=1}^r \TT^1(\hat{\cO}_{C,q_i}).
\]
For $1 \leq i \leq r$, let $t_i$ be the coordinate on
$\Aut(E_i,q_i) \cong \GG_m$ given by the isomorphism following
Definition \ref{definition-monomial}. Let $g(K)$ be the arithmetic genus of the core $K$. There
are coordinates 
\[
\begin{array}{lllll}
\text{``{\bf k}ore''}		&\bk=(k_1,\ldots,k_{g(K)+n+r}) & \text{on}  &	\TT^1(K,\pn,\qr) & \\
\text{``{\bf c}rimping''}	&\bc_i=(c_{i,1},\ldots,c_{i,m-1}) & \text{on} & \Cr^1(E_i,q_i) & \text{for $1 \leq i \leq r$} \\
\text{``{\bf s}ingularity''}	&\bs_i=(s_{i,0},\ldots,s_{i,2m-1}) & \text{on} & \TT^1(\hat{\cO}_{E_i,\xi_i}) & \text{for $1 \leq i \leq r$} \\
\text{``{\bf n}ode''}	&
n_i & \text{on} & \TT^1(\hat{\cO}_{C,q_i}) &
 \text{for $1 \leq i \leq r$}
\end{array}
\]
such that the
action of $\Aut(C,\pn)^{\star}$ on $\TT^1(C,\pn)$ is
given by
\[
\begin{array}{lll}
k_l & \mapsto & k_l \\ 
c_{i,l} & \mapsto & t_i^{2l-1} c_{i,l} \\ 
s_{i,l} & \mapsto & t_i^{2l-4m-2} s_{i,l} \\ 
n_i  & \mapsto & t_i n_i
\end{array}
\]
\item[Case I$'$]
There are decompositions
\[
\Aut(C)^{\circ} = \Aut(C)^{\star} = \prod_{i=1}^2 \Aut(E_i,q)
\]
and
\[
\TT^1(C) = 
\bigoplus_{i=1}^2 
\left(\Cr^1(E_i,q) \oplus  \TT^1(\hat{\cO}_{E_i,\xi_i})\right)
\oplus \TT^1(\hat{\cO}_{C,q}).
\]
For $1 \leq i \leq 2$, let $t_i$ be the coordinate on
$\Aut(E_i,q) \cong \GG_m$ given by the isomorphism following
Definition \ref{definition-monomial}. There are coordinates
\[
\begin{array}{llll}
\bc_i=(c_{i,1},\ldots,c_{i,m-1}) & \text{on} & 
 \Cr^1(E_i,q) & \text{for $1 \leq i \leq 2$} \\
\bs_i=(s_{i,0},\ldots,s_{i,2m-1}) & \text{on} &
 \TT^1(\hat{\cO}_{E_i,\xi_i}) & \text{for $1 \leq i \leq 2$} \\
n & \text{on} & \TT^1(\hat{\cO}_{C,q}) &
\end{array}
\]
such that the action of $\prod_{i=1}^2
\Aut(E_i,q)$ on $\TT^1(C)$ is given by
\[
\begin{array}{lll}
c_{i,l} & \mapsto & t_i^{2l-1} c_{i,l} \\
s_{i,l}  & \mapsto  & t_i^{2l-4m-2} s_{i,l}  \\
n & \mapsto & t_1 t_2 n.
\end{array}
\]
\item[Case I$''$]
In this case $(C,p)=(E,q)$ so
\[
\Aut(C,p)^{\circ} = \Aut(C,p)^{\star} = \Aut(E,q)
\]
and there is a decomposition
\[
\TT^1(C,p) = 
\Cr^1(E,q) \oplus  \TT^1(\hat{\cO}_{E,\xi})
\oplus \TT^1(\hat{\cO}_{C,q}).
\]
Let $t$ be the coordinate on $\Aut(E,q) \cong \GG_m$ given by
the isomorphism following Definition
\ref{definition-monomial}. There are coordinates 
\[
\begin{array}{lll}
\bc=(c_{1},\ldots,c_{m-1}) & \text{on} & 
 \Cr^1(E,q) \\
\bs=(s_{0},\ldots,s_{2m-1}) & \text{on} &
 \TT^1(\hat{\cO}_{E,\xi}) 
 \end{array}
\]
such that
the action of $\Aut(E,q)$ on $\TT^1(C,p)$ is given
by
\[
\begin{array}{lll}
c_{l} & \mapsto & t^{2l-1}  c_{l}  \\
s_{l} & \mapsto & t^{2l-4m-2} s_{l} \\
\end{array}
\]
\end{description}
\end{prop}

\begin{proof}

We give the proof for case $\mathbf{I}$ under the assumption that
$r=1$. The proof for arbitrary $r$ is not more difficult and the
proofs for cases $\mathbf{I'}$ and $\mathbf{I''}$ are
similar. Write $(E,q)$ for $(E_1,q_1)$ and set $G:=\Aut(E,q)$.

The action of $G$ on $\TT^1(K,\pn,q)$ is trivial, since $C$
is nodal at $q$.

The coordinates specified in the definition of the monomial
$H_{m,1}$-curve induce an isomorphism $\hat{\cO}_{E,\xi}
\iso \CC[[x,y]]/(y^2-x^{2m+1})$. Fix the following isomorphism
$\TT^1(\hat{\cO}_{E,\xi}) \iso \CC^{\oplus 2m}$:
\[
\Spec \CC[[x,y,\eps]]/
(y^2-x^{2m+1}-s_{2m-1}\eps x^{2m-1}-s_{2m-2}\eps x^{2m-2}
 -\ldots-s_{0}\eps,\eps^2) 
\mapsto (s_0,\ldots,s_{2m-1}).
\]
Then one computes from the definition that the action of $G$ is:
\[
s_l \mapsto t^{2l-4m-2} s_l \quad\text{for}\quad
 0 \leq l \leq 2m-1
\]

Choose an isomorphism
$\hat{\cO}_{C,q} \iso \CC[[s,u]]$ such that there
is a commutative diagram
\[
\xymatrix{ \widehat{\cO}_{C,q} \ar[r] \ar[d] &
            \CC[[s,u]]/(s u) \ar[d]           \\
           \widehat{\cO}_{E,q} \ar[r]        &
           \CC[[s]]                          }
\]
where the lower horizontal map is the isomorphism induced by the
coordinates specified in the definition of the monomial
$H_{m,1}$-curve. Fix the following isomorphism
$\TT^1(\hat{\cO}_{C,q}) \iso \CC$:
\[
\Spec \CC[[s,u,\eps]]/
(s u-c\eps,\eps^2) 
\mapsto c.
\]
Then one computes from the definition that the action of $G$ is:
\[
c \mapsto t c.
\]

From Example~1.111 of \cite{wyck}, one sees that there is the
following isomorphism $\Cr^1(E,q) \iso \CC^{\oplus (m-1)}$:
\[
\Spec \CC[(s+c_1\eps s^2+c_2\eps
s^4+\ldots+c_{m-1}\eps s^{2m-2})^2,
s^{2m},\ldots,s^{4m-1},\eps]/(\eps)^2 \mapsto
(c_1,\ldots,c_{m-1}).
\]
Then one computes from the definition that the action of $G$ is:
\[
c_l \mapsto t^{2l-1} c_l \quad\text{for}\quad
 1 \leq l \leq m-1.
\]
\end{proof}


Next we give the analogous proposition for the case that $k=2m+1$
is odd. Again, there are three possible cases for the canonical
decomposition of a maximally degenerate $A_{2m+1}$-stable curve $(C, \pn)$ (
Definition~\ref{canonical-decomposition-odd}).
\begin{description}

\item[Case II] $(C,\pn)=(K \cup E_1 \cup \ldots
\cup E_r \cup E_{r+1} \cup \ldots \cup E_{r+s},\pn)$ where for $1 \leq i \leq r$, $E_i=\bigcup_{j=1}^{l_i}(E_{ij},q_{i,2j-1}, q_{i,2j})$ is a link of $l_i$ monomial $H_{m,2}$-curves meeting $K$ at two nodes $\{q_{i,1}, q_{i,2l_i}\}$ and with $E_{i,j}$ glued nodally to $E_{i,j+1}$ at $q_{i,2j} \sim q_{i,2j+1}$, and for $r+1 \leq i \leq r+s$,  $E_i=\bigcup_{j=1}^{l_i}(E_{ij},q_{i,2j-1}, q_{i,2j})$ is a link of $l_i$ monomial $H_{m,2}$-curves meeting $K$ at one point $q_{i,1}$ and terminating in the marked point $p_{n-s-r+i}$. \\

\item[Case II$'$] $(C,p_1,p_2)=(E_1 \cup \ldots \cup E_r,q_1,q_{2r})$,
where $(E_i,q_{2i-1},q_{2i})$ are monomial $H_{m,2}$-curves glued nodally at $q_{2i} \sim q_{2i+1}$ for each $1 \leq i \leq r-1$. \\
\item[Case II$''$] $C=(E_1 \cup \ldots \cup E_r)$ where $(E_i,q_{2i-1},q_{2i})$ are monomial $H_{m,2}$-curves glued nodally at $q_{2i} \sim q_{2i+1}$ for each $1 \leq i \leq r-1$ and at $q_{2r} \sim q_1$.
\end{description}

\begin{prop}\label{first-order-action-odd}
Let $(C, \pn)$ be a maximally degenerate $A_{k}$-stable curve with $k = 2m+1$ and $m=1$.  
\begin{description}
\item[Case II]There are decompositions
\begin{align*}
\Aut(C,\pn)^{\circ} &= \Aut(K,\{p_i\}_{i=1}^{n-s},\{q_{i,1}\}_{i=1}^{r+s}, \{q_{i,2l_i}\}_{i=1}^{r}) \times \Aut(C,\pn)^{\star}\\
&=  \Aut(K,\{p_i\}_{i=1}^{n-s},\{q_{i,1}\}_{i=1}^{r+s}, \{q_{i,2l_i}\}_{i=1}^{r})
 \times \prod_{i=1}^{r+s} \prod_{j=1}^{l_i} \Aut(E_{ij}, q_{i,2j-1}, q_{i,2j})
\end{align*}
and
$$
\begin{aligned}
\TT^1(C,\pn) = & \TT^1(K,\{p_i\}_{i=1}^{n-s},\{q_{i,1}\}_{i=1}^{r+s}, \{q_{i,2l_i}\}_{i=1}^{r})
         \oplus \bigoplus_{i=1}^{r+s} \bigoplus_{j=1}^{l_{i}}
        (  \TT^1(\hat{\cO}_{E_{ij},\xi_{ij}})) \\
         &\oplus \bigoplus_{i=0}^{r} (\TT^1(\hat{\cO}_{C,q_{i,1}}) \oplus \TT^1(\hat{\cO}_{C,q_{i,2l_i}})) \oplus  \bigoplus_{i=r+1}^{r+s} (\TT^1(\hat{\cO}_{C,q_{i,1}}) )
\end{aligned}
$$
Now let $t_{i,j}$ be the coordinate on
$\Aut(E_{ij},q_{i,2j-1}, q_{i,2j}) \cong \GG_m$ given by the isomorphism
following Definition \ref{definition-monomial}, and let $g(K)$ be the arithmetic genus of $K$.  There are
coordinates
\[
\begin{tabular}{lllll}
``{\bf k}ore''		& $\bk=\left(k_l\right)_{l=1}^{g(K)+n+2r}$ & \text{on} &  \multicolumn{2}{l}{$\TT^1(K,\{p_i\}_{i=1}^{n-s},\{q_{i,1}\}_{i=1}^{r+s}, \{q_{i,2l_i}\}_{i=1}^{r})$}   \\
``{\bf s}ingularity''	&$\bs_{i,j}=\left( s_{i,j,l}\right)_{l=1}^{2m}$
& on & $\TT^1(\hat{\cO}_{E_i,\xi_i})$ \qquad \quad \, &
 $1 \leq i \leq r+s$, $1 \leq j \leq l_{i}$ \\
``{\bf n}ode''	&$n_{i,j}$ & \text{on} & $\TT^1(\hat{\cO}_{C,q_{i,2j}})$ &
$1 \leq i \leq r+s$, $1 \leq j \leq l_{i}-1$\\
&&&& $r+1 \leq i \leq r+s$, $j=l_i$
\end{tabular}
\]
such that the action of $\Aut(C,\pn)^{\star}$ on $\TT^1(C,\pn)$ is given by
\[
\begin{array}{lll}
k_l & \mapsto & k_l \\
s_{i,j,l} &  \mapsto & t_{i,j}^{l-2m-2} s_{i,j,l}  \\
n_{i,j} & \mapsto & t_{i,j} t_{i,j+1} n_{i,j} \,\,(j \neq 0, l_i)\\
n_{i,0} & \mapsto & t_{i,1} n_{i,0}\\
n_{i,l_i} & \mapsto & t_{i,l_i} n_{i,l_i}
\end{array}
\]

\item[Case II$'$] There are decompositions
\[
 \Aut(C, p_1, p_2)^{\circ} =  \Aut(C, p_1, p_2)^{\star} =
\prod_{i=1}^r \Aut(E_i, q_{2i-1}, q_{2i})
\]
and
\[
\TT^1(C,p_1,p_2) 
= \bigoplus_{i=1}^r
   ( \TT^1(\hat{\cO}_{E_i,\xi_i}))
  \oplus \bigoplus_{i=1}^{r-1} \TT^1(\hat{\cO}_{C,q_{2i}})
\]
For $1 \leq i \leq r$, let $t_i$ be the coordinate on
$\Aut(E_i,q_{2i-1}, q_{2i}) \cong \GG_m$ given by the isomorphism
following Definition \ref{definition-monomial}. There are
coordinates
\[
\begin{array}{llll}
\bs_i=(s_{i,0},\ldots,s_{i,2m}) & \text{on} &
  \TT^1(\hat{\cO}_{E_i,\xi_i}) & \text{for $1 \leq i \leq r$} \\
n_i & \text{on} & \TT^1(\hat{\cO}_{q_{2i}}) &
  \text{for $1 \leq i \leq r-1$}
\end{array}
\]
such that the action of $\Aut(C,\pn)^{\star}$ on $\TT^1(C,p_1,p_2)$ is given by
\[
\begin{array}{lll}
s_{i,l} & \mapsto & t_i^{l-2m-2} s_{i,l}  \\
n_i & \mapsto & t_i t_{i+1} n_i.
\end{array}
\]
\item[Case II$''$] There are decompositions
\[
\Aut(C)^{\circ} = \prod_{i=1}^{r} \Aut(E_i, q_{2i-1}, q_{2i})
\]
and
\[
\TT^1(C) = 
\bigoplus_{i=1}^{r} 
 (\TT^1(\hat{\cO}_{E_i,\xi_i}))
\oplus \bigoplus_{i=1}^{r} \TT^1(\hat{\cO}_{C,q_{2i}})
\]
For $1 \leq i \leq r$, let $t_i$ be the coordinate on
$\Aut(E_i,q_{2i-1}, q_{2i}) \cong \GG_m$ given by the
isomorphism in Definition \ref{definition-monomial}. There are
coordinates
\[
\begin{array}{llll}
\bs_i=(s_{i,0},\ldots,s_{i,2m}) & \text{on} &
  \TT^1(\hat{\cO}_{E_i,\xi_i}) & \text{for $1 \leq i \leq r$} \\
n_i & \text{on} & \TT^1(\hat{\cO}_{q_{2i}}) &
  \text{for $1 \leq i \leq r$}
\end{array}
\]
such that
\[
\begin{array}{lll}
s_{i,l}&  \mapsto & t_i^{l-2m-2} s_{i,l} \\
n_i & \mapsto & t_i t_{i+1} n_i 
\end{array}
\]
where $t_{r+1}:=t_1$.
\end{description}
\end{prop}

\begin{proof}

We give the proof for case $\mathbf{II}$ under the assumption
that $r=s=1$. The proof for arbitrary $r,s$ is not more difficult and
the proofs for cases $\mathbf{II'}$ and $\mathbf{II''}$ are
similar. Write $E$ for $E_1$ and set $G:=\Aut(E,q_1,q_2)$.

The coordinates specified in the definition of the monomial $H_{m,2}$-curve 
induce an isomorphism $\hat{\cO}_{E,\xi}
\iso \CC[[x,y]]/(y^2-x^{2m+2})$. Fix the following isomorphism
$\TT^1(\hat{\cO}_{E,\xi}) \iso \CC^{\oplus 2m+1}$:
\[
\Spec \CC[[x,y,\eps]]/
(y^2-x^{2m+2}-s_{2m}\eps x^{2m}-s_{2m-1}\eps x^{2m-1}
 -\ldots-s_{0}\eps,\eps^2) 
\mapsto (s_0,\ldots,s_{2m}).
\]
Then one computes from the definition that the action of $G$ is:
\[
s_l \mapsto t^{l-2m-2} s_l \quad\text{for}\quad 0 \leq l \leq 2m
\]

For $i=1,2$, choose an isomorphism
$\hat{\cO}_{C,q_i} \isom \CC[[s_i,u]]$ such that there
is a commutative diagram
\[
\xymatrix{ \widehat{\cO}_{C,q_i} \ar[r] \ar[d] &
            \CC[[s_i,u]]/(s_i u) \ar[d]           \\
           \widehat{\cO}_{E,q_i} \ar[r]        &
           \CC[[s_i]]                          }
\]
where the lower horizontal map is the isomorphism induced by the
coordinates specified in the definition of the monomial
$H_{m,2}$-curve. Fix the following isomorphism
$\TT^1(\hat{\cO}_{C,q_i}) \iso \CC$:
\[
\Spec \CC[[s_i,u,\eps]]/
(s_i u-n_i\eps,\eps^2) 
\mapsto n_i.
\]
Then one computes from the definition that the action of $G$ is:
\[
n_i \mapsto t n_i \quad\text{for}\quad 1 \leq i \leq 2.
\]
\end{proof}

\subsection{Formal deformations}\label{subsection-action}

Proposition~\ref{formal-action-even} and
Proposition~\ref{formal-action-odd} below will be used in the
proof of Proposition~\ref{theorem-local-GIT} in the next section.

\begin{prop}\label{formal-action-even}

Let $(C,\pn)$ be a maximally-degenerate $A_{2m}$-stable curve in
case $\mathbf{I}$ (see Proposition \ref{first-order-action-even}).  Set
\[
A:=\CC[[\bk, \bc_1,\ldots,\bc_r,
\bs_1,\ldots,\bs_r, n_1,\ldots,n_r]].
\]
There is a miniversal
formal deformation $\psi:\Spf A \arr \bar{\cM}_{g,n}(A_k)$ of
$(C,\pn)$ such that there is a $2$-cartesian diagram
\[
\xymatrix{ \Spf A/\inn_{i=1}^r (\bs_i) 
              \ar[r] \ar[d]^{\theta}     &
           \Spf A \ar[d]^\psi                            &
           \Spf A/\inn_{i=1}^r (\bc_i,n_i)
              \ar[l]\ar[d]^{\phi}    \\
           \overline{\cS}_{g,n}(A_{2m}) \ar[r]^{\imath} &
           \overline{\cM}_{g,n}(A_{2m})                 &
           \overline{\cH}_{g,n}(A_{2m}) \ar[l]_{\jmath} }
\]
and the action of the group scheme $\Aut(C,\pn)^{\star} $ on
$\hat{\Def}(C,\pn)$ induced by $\psi$ coincides with the action
induced by the action of the group $\Aut(C,\pn)^{\star} $ on
$\TT^1(C,\pn)$ described in
Proposition~\ref{first-order-action-even}.
\end{prop}

\begin{proof}

Let $G:=\Aut(C,\pn)^{\star} =\prod_{i=1}^r \Aut(E_i,q_i)$. We may choose a
miniversal formal deformation $\psi':\Spf A' \arr
\bar{\cM}_{g,n}(A_{2m})$ of $(C,\pn)$ where $A'$ has a linear $G$-action 
(induced by the action of $\fm_{A'}/\fm_{A'}^2$). 
Let $\cC' \arr \Spf A'$ be
the corresponding family of curves. Let $\Spf S_i'$ be the formal
closed subscheme of $\Spf A'$ where the $A_{2m}$-singularity
$\xi_i$ is preserved in $\cC'$ ($1 \leq i \leq r$).
Let $\Spf H_i'$ be the formal closed subscheme of $\Spf A'$ where
the node $q_i$ is preserved in $\cC'$ and the induced deformation
of $(E_i,q_i)$ is a family of $H_{m,1}$-curves ($1 \leq i \leq
r$). Clearly there is a $2$-cartesian diagram
\[
\xymatrix{ \bigcup_{i=1}^r \Spf S_i' \ar[r] 
            \ar[d]^{\unn_{i=1}^r\theta_i'} & 
           \Spf A' \ar[d]^{\psi'} &
           \bigcup_{i=1}^r \Spf H_i' \ar[l] 
            \ar[d]^{\unn_{i=1}^r\phi_i'} \\
           \bar{\cS}_{g,n}(A_{2m}) \ar[r] & 
           \bar{\cM}_{g,n}(A_{2m}) &
           \bar{\cH}_{g,n}(A_{2m}). \ar[l] }
\]
We claim that we can choose an isomorphism $\Spf A \arr \Spf A'$
such that for $1 \leq i \leq r$, the base change of $\Spf S_i'$ is
$\Spf A/(\bs_i)$ and the base change of $\Spf H_i'$ is $\Spf
A/(\bc_i,n_i)$ and in addition the action of $G$ on
$\Spf A$ is as claimed.

Let $S_i'=A'/I_i'$ and $H_i'=A'/J_i'$. It is easy to see that 
the inclusions $\Spf S_i'
\arr \Spf A'$ and $\Spf H_i' \arr \Spf A'$ are equivariant with
respect to the actions of $G$ induced by $\theta_i'$, $\psi'$ and
$\phi_i'$.


There is a commutative diagram of equivariant exact sequences
\[
\xymatrix{ 0 \ar[r] &
           I_i' \ar[d] \ar[r] &
           m_{A'} \ar[d] \ar[r] &
           m_{S_i'} \ar[d] \ar[r] &
           0 \\
           0 \ar[r]  &
           I_i'/m_{A'} I_i' \ar[r] \ar[d]^\isom &
           m_{A'}/m_{A'}^2 \ar[r] \ar[d]^\isom &
           m_{S_i'}/m_{S_i'}^2  \ar[r] \ar[d]^= &
           0  \\
           0 \ar[r] &
           \TT^1(\hat{\cO}_{E_i,\xi_i})^\vee \ar[r] &
           \TT^1(C,\pn)^\vee \ar[r] &
           \TT_{S_i}^\vee \ar[r] &
           0.  }
\]
The second row is exact on the left
because $A'$ and $S_i'$ are both formally smooth.
All the maps are equivariant. 

By Proposition~\ref{first-order-action-even}, we can choose, for
$1 \leq i \leq r$, a basis
$\bs_i'=(s_{i,0}',\ldots,s_{i,2m-1}')$ for
$\TT^1(\hat{\cO}_{E_i,\xi_i})^\vee$ such that the action of $G$ is
\[
s_{i,l}' \arr t_i^{2l-4m-2} s_{i,l}'.
\]
Choose an equivariant section $\sigma_i:
\TT^1(\hat{\cO}_{E_i,\xi_i})^\vee \arr I_i'$ and, abusing
notation, write $s_{i,l}'$ also for the image of $s_{i,l}'$ under
$\sigma_i$. By Nakayama's Lemma, $s_{i,0}',\ldots,s_{i,2m-1}'$
generate $I_i'$.

Recall from Lemma~\ref{splitting} that for $1 \leq i \leq r$ there
is an equivariant inclusion
\[
\Cr^1(E_i,q_i)^\vee \arr \TT^1(C,\pn)^\vee
\]
There is a commutative diagram of equivariant exact sequences
\[
\xymatrix{ 0 \ar[r] &
           J_i' \ar[d] \ar[r] &
           m_{A'} \ar[d] \ar[r] &
           m_{H_i'} \ar[d] \ar[r] &
           0 \\
           0 \ar[r]  &
           J_i'/m_{A'} J_i' \ar[r] \ar[d]^\isom &
           m_{A'}/m_{A'}^2 \ar[r] \ar[d]^\isom &
           m_{H_i'}/m_{H_i'}^2  \ar[r] \ar[d]^= &
           0  \\
           0 \ar[r] &
           \Cr^1(E_i,q_i)^\vee 
              \oplus \TT^1(\hat{\cO}_{C,q_i})^\vee \ar[r] &
           \TT^1(C,\pn)^\vee \ar[r] &
           \TT_{H_i}^\vee \ar[r] &
           0.  }
\]
Therefore, using Proposition~\ref{first-order-action-even} and
using an equivariant section as above, we can choose, for $1 \leq
i \leq r$, $\bc_i'=(c_{i,1}',\ldots,c_{i,m-1}')$ and
$n_i'$ such that $J_i'=(\bc_i',n_i')$ and the action of $G$ is
\[
c_{i,l}' \arr t_i^{2l-1} c_{i,l}' \qquad
n_i' \arr t_i n_i'.
\]

The subset $S:=\{\bc_1', \ldots, \bc_r', n_1', \ldots,
n_r', \bs_1', \ldots, \bs_r' \}$ of $\TT^1(C,\pn)^\vee$ is
linearly independent.
Choose $k_1', \ldots, k_{g(K)+n}'$ such that $S \unn \{k_1', \ldots, k_{g(K)+n}'\}$ is a basis and the
action of $G$ is
\[
k_l' \mapsto k_l'.
\]
Choose an equivariant section $\sigma:
\TT^1(C,\pn)^\vee \arr m_{A'}$
and, abusing notation, write $k_l'$ also for the image of
$k_l'$ under $\sigma$.
Define a map $A \arr A'$ by
\[
k_i \arr k_i' \qquad
c_{i,l} \arr c_{i,l}' \qquad
n_i \arr n_i' \qquad
s_{i,l} \arr s_{i,l}'.
\]
This is a local homomorphism inducing an isomorphism $A/m_A^2 \arr
A'/m_{A'}^2$ and therefore an isomorphism.
\end{proof}

\begin{remark} Similar descriptions can be given for cases I$'$ and I$''$.
\end{remark}

\begin{prop}\label{formal-action-odd}

Let $(C,\pn)$ be a maximally degenerate
$A_{2m+1}$-stable 
 curve in case $\mathbf{II}$ (see Proposition
\ref{first-order-action-odd}). Set
\[
A:=\CC[[\bk, \{
\bs_{i,j}, n_{i,j} \co 1 \leq i \leq  r+s, 1 \leq j \leq l_i \}]].
\]
There is a miniversal
formal deformation $\psi:\Spf A \arr \bar{\cM}_{g,n}(A_k)$ of
$(C,\pn)$ such that there is a $2$-cartesian diagram
\[
\xymatrix{ \Spf A/\inn_{i,j} (\bs_{i,j}) 
              \ar[r] \ar[d]^{\theta}     &
           \Spf A \ar[d]^\psi                            &
           \Spf A/\inn_{i,\mu,\nu \in S} J_{\mu,\nu} 
              \ar[l]\ar[d]^{\phi}    \\
           \overline{\cS}_{g,n}(A_{2m+1}) \ar[r]^{\imath} &
           \overline{\cM}_{g,n}(A_{2m+1})                 &
           \overline{\cH}_{g,n}(A_{2m+1}) \ar[l]_{\jmath}, }
\]
where $S:=\{i, \mu,\nu: 1 \leq i \leq r+s, 1 \leq \mu \leq \lceil\frac{l_i}{2}\rceil,\,0
\leq \nu \leq l_i-2\mu+1\}$ and
\[
J_{i,\mu,\nu}=
(n_{i,\nu}, 
 \bs_{i,\nu+2}, 
 \ldots, \bs_{i,\nu+2\mu-2}, 
n_{i,\nu+2\mu-1}),
\]
and the action of the group scheme $\Aut(C,\pn)^{\star} $ on
$\hat{\Def}(C,\pn)$ induced by $\psi$ coincides with the action
induced by the action of the group $\Aut(C,\pn)^{\star} $ on
$\TT^1(C,\pn)$ described in
Proposition~\ref{first-order-action-even}.
\end{prop}

\begin{proof}

The proof is similar to the proof of
Proposition~\ref{formal-action-even}. Let
$G:=\Aut(C,\pn)^{\star} =\prod_{i=1}^{r+s} \prod_{j=1}^{l_i} \Aut(E_{ij},q_{i,2j-1}, q_{i,2j})$. Choose a
miniversal formal deformation $\psi':\Spf A' \arr
\bar{\cM}_{g,n}(A_{2m+1})$ of $(C,\pn)$. Let $\cC' \arr \Spf A'$
be the corresponding family of curves. Let $\Spf S_{i,j}'$ be the
formal closed subscheme of $\Spf A'$ where the
$A_{2m+1}$-singularity $\xi_{i,j}$ is preserved in $\cC'$.
Let $\Spf N_{i,j}'$ be the formal closed subscheme of $\Spf A'$ where
the node $\eta_{i,j}$ is preserved in $\cC'$. Clearly there is a $2$-cartesian diagram
\[
\xymatrix{ \bigcup_{i,j} \Spf S_{i,j}' 
             \ar[r] \ar[d]^{\unn_{i,j}\theta_{i,j}'} & 
           \Spf A' \ar[d]^{\psi'} &
           \bigcup_{i,\mu,\nu \in S} \Spf H_{i,\mu,\nu}'  \ar[l]
           \ar[d]^{\unn_{i,\mu,\nu \in S} \phi_{i,\mu,\nu}'} \\
           \bar{\cS}_{g,n}(A_{2m+1}) \ar[r] & 
           \bar{\cM}_{g,n}(A_{2m+1}) &
           \bar{\cH}_{g,n}(A_{2m+1}), \ar[l] }
\]
where 
\[
\Spf H_{i,\mu,\nu}' = 
\Spf N_{i,\nu}' \inn \Spf N_{i,\nu+2\mu -1}' \inn
\bigcap_{j=1}^{m-1} \Spf S_{i,\nu+2i}'.
\]
Now argue as in the proof of Proposition~\ref{formal-action-even},
this time using Proposition~\ref{first-order-action-odd}, that we
can choose an isomorphism $\Spf A \arr \Spf A'$ such that the base change of $\Spf S_{i,j}'$ is $\Spf A/(\bs_{i,j})$
and the base change of $\Spf N_{i,j}'$ is $\Spf A/(n_{i,j})$ and in
addition the action of $G$ on $\Spf A$ is as claimed.
\end{proof}

\begin{remark}  Similar descriptions can be given for cases II$'$ and II$''$.
\end{remark}

\section{Local variation of GIT}
\label{section-local-vgit}

In this section, we calculate the variation of GIT chambers 
for the action of the automorphism group on the deformation space of a 
maximally degenerate curve in $\bar{\cM}_{g,n}(A_k)$.  The main result is 
Proposition \ref{theorem-etale-variation} which states that the inclusions
 $$\bar{\cM}_{g,n}(A_k^{-}) 
 \subseteq \bar{\cM}_{g,n}(A_k) \supseteq  
 \bar{\cM}_{g,n}(A_k^{+})$$
 correspond \'etale locally around closed points in
 $\bar{\cM}_{g,n}(A_k)$ to the variation of GIT chambers on the
 deformation space.  We begin by reviewing variation of GIT on an affine in Section \ref{section-variation-GIT}. We then use this formalism to prove a formal local version of the result in a sequence of successively more complicated cases (Sections \ref{subsection-monomial-case} - \ref{subsection-general-case}), culminating in Proposition \ref{theorem-local-GIT}. Finally,  in Section \ref{section-etale-GIT}, we show that Proposition \ref{theorem-etale-variation} follows from Proposition \ref{theorem-local-GIT} and the algebraization result of Proposition \ref{prop-algebraization}.

\subsection{Variation of GIT}
\label{section-variation-GIT}
Let $G$ be a linearly reductive algebraic group over $\CC$ acting on an 
affine scheme $X=\Spec A$ finite type over $\CC$.  
Let $\sigma: A \to \Gamma(G) \tensor A$ denote the dual action.  
Let $\chi: G \to \GG_m = \Spec \CC[t]_t$ be a character.  For an integer $n$, define
$$A_n = \{f \in A \mid \sigma(f) = (\chi^*t)^n f\}.$$
Note that $A_0 = A^G$.
Define $V^{-}$ and $V^{+}$ to be the reduced $G$-invariant closed subschemes 
of $X$ defined by the ideals 
$$\begin{aligned}
I^{-} &= (f \in A \mid f \in A_n \text{ for } n < 0), \\ 
I^{+} &= (f \in A \mid f \in A_n \text{ for } n > 0).
\end{aligned}$$
We define $X^{-} = X \setminus V^{-}$ and $X^{+} = X \setminus V^{+}$ to be 
the $G$-invariant open subschemes which, of course, depend on the character $\chi$.  Then it is easy to see (see \cite{dolgachev-hu} or \cite{thaddeus})  that that there is a commutative diagram 
$$\xymatrix{
X^{-} \ar@{^(->}[r] \ar[d]		&X  \ar[d]		& X^{+} \ar@{_(->}[l]\ar[d]\\
\Proj \bigoplus_{n \ge 0} A_{(-n)} = X^-//G	\ar[r]	& X//G= \Spec A^G			&X^+//G =  \Proj \bigoplus_{n \ge 0} A_n \ar[l] \\
}$$
where the vertical arrows are good GIT quotients.  Note that the induced morphisms $X^- \hookarr X$ and $X^+ \hookarr X$ are open immersions while $X^{-}//G \to X//G$ and $X^{+} //G \to X//G$ are projective.  
We will also use the following stack-theoretic language: set $\cX = [X/G]$, $\cX^{-} = [X^{-} / G]$ and $\cX = [X^{+} / G]$.  We have a commutative diagram
$$\xymatrix{
\cX^{-} \ar@{^(->}[r]\ar[d]		&\cX  \ar[d]		& \cX^{+} \ar@{_(->}[l] \ar[d]\\
X^{-}//G \ar[r]	& X//G	& X^{+}//G \ar[l]
}$$
where the vertical arrows are good moduli spaces.

\begin{remark}  The character $\chi$ induces a $G$-linearization $\cL_{\chi}$ of the structure sheaf $\oh_X$.  The semi-stable locus $X^{\ss}_{\cL_\chi}$ (resp., $X^{\ss}_{\cL^{\dual}_{\chi}}$) is identified with $X^{+}$ (resp., $X^{-}$).
\end{remark}

\begin{prop}[Affine Hilbert-Mumford criterion]
\label{prop-hilbert-mumford}
Suppose $G$ is a linearly reductive algebraic group over $\CC$ acting on an affine scheme $X = \Spec A$ finite type over $\CC$.  Let $\chi: G \to \GG_m$ be a character.  Let $x \in X(\CC)$.  Then $x \in V^{-}$ (resp., $x \in V^{+}$) if and only if there exists a one-parameter subgroup $\lambda: \GG_m \to G$ with $\chi \circ \lambda > 0$ (resp., $\chi \circ \lambda < 0$) such that $\lim_{t \to 0} \lambda(t) \cdot x$ exists.
\end{prop}

\begin{proof}
Suppose there exists a one-parameter subgroup $\lambda: \GG_m \to G$ with $\chi \circ \lambda > 0$ such that $\lim_{t \to 0} \lambda(t) \cdot x$ exists.  Let $f \in A$ satisfy $\sigma(f) = \chi^*(t)^n f$ for $n < 0$.  Then under $\GG_m \to X, t \mapsto \lambda(t) \cdot x$, the function $f$ pulls back to $t^{n (\chi \circ \lambda)} f(x)$.  Since the limit exists and $n(\chi \circ \lambda) < 0$, $f(x)=0$.  Therefore $x \in V^{-}$.

Conversely, let $x \in V^{-}$.  Consider the induced action of $G$ on $\Spec A[y]$ via $y \mapsto (\chi^*t)y$.  (This is precisely the $G$-line bundle over $X$ corresponding to $\chi$.)  Then $\overline{O(x,1)} \cap \{y=0\} \neq 0$.  Otherwise, there would exist a function $f \in A[y]^G$ with $f(x,1) \neq 0$ and $f(x,0)= 0$; by writing $f = \sum_n f_n y^n$, we see that for some $n > 0$, $f_n \mapsto (\chi^* t)^{-n} f_n$ and $f_n(x) \neq 0$, which contradicts $x \in V^{-}$.  Therefore, in the closure of the $G^{\circ}$-orbit of $(x,1)$ there is a point $(x_0,0)$ with closed orbit with $x_0 \neq 0$.  By the Hilbert-Mumford criterion (\cite[Theorem 2.1]{git}), there exists a one-parameter subgroup $\lambda: \GG_m \to G$ such that $\lim_{t \to 0} \lambda(t) \cdot (x,1) = (x_0, 0)$.   This gives the desired one-parameter subgroup.  The $V^{+}$ case is similar. 
\end{proof}

\begin{lem} 
\label{lemma-vgit-product}
 Let $G_i$ be linearly reductive algebraic groups acting on affine schemes $X_i$ and $\chi_i: G_i \to \GG_m$be characters for $i=1, \ldots, n$.  Consider the diagonal action of $G = \prod_i G_i$ on $X = \prod_i X_i$ and the character $\prod_i \chi_i: G \to \GG_m$.  Then 
$$\begin{aligned}
V^{-} = & \bigcup_{i=1}^{n} X_1 \times \cdots \times V_i^{-} \times \cdots \times X_n \\
V^{+} = & \bigcup_{i=1}^{n} X_1 \times \cdots \times V_i^{+} \times \cdots \times X_n
\end{aligned}$$
\end{lem}

\begin{proof}  This follows from Proposition \ref{prop-hilbert-mumford}.
\end{proof}

\begin{lem}
\label{lemma-vgit-closed}
Let $G$ be a linearly reductive algebraic group over $\CC$ acting on an affine $X = \Spec A$ finite type over $\CC$.  Let $\chi: G \to \GG_m$ be a character.  Let $Z \subseteq X$ be a $G$-invariant closed subscheme.  Then with respect to the character $\chi$, we have $Z^{-} = X^{-} \cap Z$ and $Z^{+} = X^{+} \cap Z$.
\end{lem}

\begin{proof}  Clearly $Z^{-} \subseteq X^{-} \cap Z$.  Let $I \subseteq A$ be the invariant ideal defining $Z$.  If we consider the induced action of $G$ on $\Spec A[y]$ where $y \mapsto \chi^{*}t y$, then since $G$ is linearly reductive $A[y]^G \to (A[y]/I)^G \cong (A/I [y])^G$ is surjective.  Let  $z \in Z^{-}$ and $f \in A/I$ with $f \mapsto \chi^{*} t^d f$ for $d < 0$ with $f(z) \neq 0$.  It follows that there exists a lift $\tilde f \in A$ with $\tilde f \mapsto \chi^{*} t^d \tilde f$ for $d < 0$ with $\tilde f(z) \neq 0$ so $z \in X^{-}$.  The $Z^{+}$ case is similar.
\end{proof}



Now we will analyze the $V^-$ and $V^+$ chambers for the natural action of $\Aut(C,\pn)^\circ$ on the first order deformation space ${\Def}(C,\pn)$ in a sequence of successively more general cases.

\subsection{The case of a monomial $H_{m,1}$-curve/$H_{m,2}$-curve}
\label{subsection-monomial-case}

\noindent
Let $k = 2m$, and let $(E,q)$ be the monomial $H_{m,1}$-curve (Definition \ref{definition-monomial}).
By Case I$''$ of Proposition \ref{first-order-action-even}, we can write 
$${\Def}(E, q) = \Spec  \CC[s_0, \ldots, s_{2m-1}, c_1, \ldots, c_{m-1}]$$ 
where $\bs = (s_0, \ldots, s_{2m-1})$ are the ``singularity'' coordinates and $\bc = (c_1, \ldots, c_{m-1})$ are the ``crimping'' coordinates.  Furthermore, Proposition \ref{first-order-action-even} implies the action of $\Aut(E,q) = \GG_m$ is given by:
 $$
 s_l \mapsto t^{2l-4m-2} s_l \qquad c_l \mapsto t^{2l-1} c_l
 $$
 
 Next, let $k=2m+1=3$, 
 and let $(E,q_1,q_2)$ be the monomial $H_{m,2}$-curve  (Definition \ref{definition-monomial}). By Case II$''$ of Proposition \ref{first-order-action-odd} with $r=1$, we can write:
$${\Def}(E, q) = \Spec  \CC[s_0, \ldots, s_{2m},
]$$ 
where $\bs = (s_1, \ldots, s_{2m})$ are the ``singularity'' coordinates. 
Furthermore, Proposition \ref{first-order-action-even} implies the action of $\Aut(E, q_1, q_2) = \GG_m$ is given by:
 $$
 s_l \mapsto t^{l-2m-2} s_l 
 $$

In either case, since the singularity coordinates all have negative weight and the crimping coordinates all have positive weight, 
we obtain: 
 \begin{lem} \label{lemma-VGIT-monomial}
 For the monomial $H_{m,1}$-curve $(E, q)$ or monomial $H_{m,2}$-curve $(E,q_1,q_2)$ and with notation as above, we have
 $$ \begin{aligned}
 	V^{-} = &  V(\bs)  \\
 	V^{+} = &  V(\bc)
\end{aligned}
$$
\end{lem}

\begin{proof} Immediate from the definitions of $V^-$ and $V^+$.
\end{proof}


\subsection{The case of an $H_{m,2}$-link}
\label{subsection-chains}

\noindent
In this section, we handle the special case of an $H_{m,2}$-link. If $k=2m+1=3$ is odd, $n=0$, and $C = K \cup E_1 \cup \cdots \cup E_r$ consists of the union of a core and a single $H_{m,2}$-link of monomial $H_{m,2}$-curves, case II of Proposition \ref{first-order-action-odd} (applied with $r=1$,$s=0$, $l_1=r$ implies that we can write
$${\Def}(C) = \Spec \CC[\bs_1, \ldots, \bs_r, 
n_0, \ldots, n_r]$$
where the $\bs_i$ are the ``singularity'' coordinates 
and $n_i$ are the ``node'' coordinates.  Furthermore, Proposition  \ref{first-order-action-odd} states that the action by $\Aut(C)^\circ = \GG_m^r$ is given by:
 $$\begin{array}{llll}
 								&					 n_0 \mapsto t_1 n_0\\
 s_{i,l} \mapsto t_i^{l-2m-2} s_{i,j} \qquad  
 & n_l \mapsto t_i t_{i+1} n_l, \qquad & k_l \mapsto k_l \\
  								&					 n_r \mapsto t_r n_r
 \end{array}
 $$
 \begin{lem} 
 \label{lemma-invariant-calculation}
  With the above notation, 
 $$
 	V^{-} =  \bigcup_{j=1}^r V( \bs_j) \qquad \qquad
 	V^{+} =  \bigcup_{\mu \ge 1} \bigcup_{\nu = 0}^{r-2\mu+1} V_{\mu, \nu}
$$ 
where $V_{\mu, \nu} = V(n_{\nu} , 
\bs_{\nu+2}, 
\ldots , \bs_{\nu+2\mu-2} , 
n_{\nu+2\mu-1})$.
\end{lem}

\begin{remark}  For instance, $V_{1,\nu} = V(n_{\nu}, 
n_{\nu+1})$ and $V_{2, \nu} = V(n_{\nu}, 
\bs_{\nu+2} , 
n_{\nu+3})$.
\end{remark}

\begin{proof}
We will use the Hilbert-Mumford criterion of Proposition \ref{prop-hilbert-mumford}.  For the $V^{-}$ case, suppose for $x \in \Def(C)$ that for some $j$, $\bs_j(x) = 0$.  Set $\lambda = (\delta^j_i): \GG_m \to \GG_m^r \cong \prod_{i=1}^r \Aut(E_i, q_{2i-1}, q_{2i})$.  Then $\lim_{t \to 0} \lambda(t) \cdot x$ exists so $x \in V^{-}$.  Conversely, let $\lambda = (\lambda_i)$ be a one-parameter subgroup with $\sum_i \lambda_i > 0$ such that $\lim_{t \to 0} \lambda(t) \cdot x$ exists.  Then for some $j$, $\lambda_j > 0$ which implies that $\bs_j(x)  = 0$.

For the $V^{+}$ case, the inclusion $\supseteq$ is easy: suppose that $x \in V_{\mu, \nu}$ for $\mu \ge 1$ and $\nu = 0, \ldots, r-2\mu+1$.  Set
$$
\lambda = \big( \underbrace{0, \ldots, 0}_{\nu} , \underbrace{-1, 1, -1, \ldots, 1, -1}_{2\mu-1} , \underbrace{0, \ldots, 0}_{r-2\mu-\nu+1} \big)
$$
Then $\sum_i \lambda_i = -1$ and $\lim_{t \to 0} \lambda(t) \cdot x$ exists 
so $x \in V^{+}$.  For the $\subseteq$ inclusion, we will use induction on $r$.  
If $r =1$, then clearly $V^+ = V(n_0, 
n_1)$.  For $r > 1$, 
suppose $x \in V^{+}$ and $\lambda = (\lambda_i): \GG_m \to \GG_m^r$ is a 
one-parameter subgroup with $\sum_{i=1}^r \lambda_i < 0$ such that 
$\lim_{t \to 0} \lambda(t) \cdot x$ exists.  If $\lambda_r \ge 0$, then 
$\sum_{i=1}^{r-1} \lambda_i < 0$ so by the induction hypothesis 
$x \in V_{\mu, \nu}$ for some $\mu \ge 1$ and $\nu = 0, \ldots, r-2\mu$. 
 If $\lambda_r < 0$, then we immediately conclude that $n_r(x) = 
 0$.  If $\lambda_{r-1} + \lambda_r < 0$, then $n_{r-1}(x) = 0$ so $x \in V_{1, {r-1}}$.  
 If $\lambda_{r-1} + \lambda_r \ge 0$, then $\lambda_{r-1} \ge 0$ so $\bs_{r-1}(x) = 0$.  Furthermore, $\sum_{i=1}^{r-2} \lambda_i < 0$ 
 so by applying the induction hypothesis and restricting to the locus $V(n_{r-2}, \bs_{r-1}, 
 n_{r-1}, \bs_{r}, 
 n_{r})$,  
we can conclude either:  (1) $x \in V_{\mu, \nu}$ for $\mu \ge 1$ and $\nu = 0, \ldots, r-2\mu-1$, or (2) $x \in V(n_{r-\mu-4} , 
\bs_{r-\mu-2} , 
\ldots , \bs_{r-3}) $
for some $\mu \ge 1$.  In case (2), since  $\bs_{r-1}(x) = 
n_r(x) = 0 $, we have $x \in V_{\mu+1,r-\mu-4}$.
 \end{proof}
 
 \begin{remark} The chamber $V^{-}$ is the closed locus in the deformation space where an $A_k$-singularity is preserved.  The chamber $V^{+}$ is the closed locus of curves containing an $H_{m,2}$-chain.
 \end{remark}

\subsection{The general case}
\label{subsection-general-case}
\noindent
Let $(C, \pn)$ be a maximally degenerate $A_k$-stable curve.  Consider the action of $\Aut(C, \pn)$ on $\Def(C, \pn)$ described in Section \ref{subsection-action}.  Let 
$$\chi^{\star}: \Aut(C, \pn)^{\star} \to \GG_m$$
 be the character which is the product of the natural characters on the monomial $H_{m,1}$ and $H_{m,2}$-subcurves and trivial on the core (see Definition \ref{canonical-decomposition-even}).
Let $V^{-}$ and $V^{+}$ be the reduced closed subschemes of $\Def(C, \pn)$ defined by the character $\chi^{\circ}$.
Let $\Spf \hat{A} \to \bar{\cM}_{g,n}(A_k)$ be a miniversal deformation space of $(C, \pn)$.  We can identify $\hat{A}$ with the completion of the origin in the first order deformation space $\Def(C,$ $ \pn)$.   Define the closed formal subschemes $\fZ^{-}$ and $\fZ^{+}$ of $\Spf \hat{A}$ as the cartesian products
\begin{equation}
\label{diagram-formal}
\xymatrix{
\fZ^{-} \ar@{^(->}[r] \ar[d]		&\Spf \hat{A}  \ar[d]		& \fZ^{+} \ar@{_(->}[l] \ar[d]\\
V^{-} \ar@{^(->}[r]	&\Def	& V^{+} \ar@{_(->}[l]
}
\end{equation}

\medskip \noindent
Recall that $\bar \cS_{g,n}(A_k)  = \bar \cM_{g,n}(A_k) \setminus \bar \cM_{g,n}(A_k^-)$ is the locus of curves with an $A_k$-singularity and $\bar \cH_{g,n}(A_k)  = \bar \cM_{g,n}(A_k) \setminus \bar \cM_{g,n}(A_k^+)$ is the locus of curves containing an $H_{m,1}$-tail (resp., $H_{m,2}$-chain)  if $k=2m$ (resp., $k =2m+1$).

\begin{prop}  
\label{theorem-local-GIT}
Let $(C, \pn)$ be a maximally degenerate $A_k$-stable curve.   With the notation above, there is a cartesian diagram
$$\xymatrix{
\fZ^{-} \ar@{^(->}[r] \ar[d]		&\Spf \hat{A}  \ar[d]		& \fZ^{+} \ar@{_(->}[l] \ar[d]\\
 \bar{\cS}_{g,n}(A_k) \ar@{^(->}[r] 						& \bar{\cM}_{g,n}(A_k) &  \bar{\cH}_{g,n}(A_k) \ar@{_(->}[l] 	
}$$
\end{prop}

\begin{proof}  We split the proof into the cases according to the canonical decomposition of Definition \ref{canonical-decomposition-even}. 

\begin{description}
\item[Case I] 
$C = K \cup E_1 \cup \cdots \cup E_r$ where $(E_i,q_i)$ is an $H_{m,1}$-tail.  
By Lemma \ref{lemma-vgit-product}, we may assume $r=1$.  By Proposition \ref{formal-action-even}, the miniversal deformation space is $\hat{A} \cong \CC[[\bk, \bc,
\bs, n]]$ with the action of $\Aut^\circ(C) = \GG_m$ given by 
 $$
 s_{l} \mapsto t^{2l-4m-2)} s_{l} \qquad c_{l} \mapsto t^{2l-1} c_{l} \qquad n \mapsto t n,  \qquad k_l \mapsto k_l
 $$
where the $\bs = (s_0, \ldots, s_{2m-1})$ are the ``singularity'' coordinates, the $\bc = (c_1, \ldots,$ $c_{m-1})$ are the ``crimping'' coordinates, the $n$ variable is the ``node'' coordinate, and the $\bk=(k_i)$ are ``kore'' coordinates.

We see that $V^{-} = V(\{s_l\}_{l=0}^{2m-1})$ which defines the closed locus in the deformation space where the $A_k$-singularity is preserved.  On the other hand, $V^{+} = V(n, \{c_l\}_{l=1}^{m-1})$ defines the closed locus in the deformation space where an $H_{m,1}$-tail is attached at a node.

\medskip \noindent

\item[ Case I$'$] 
$C=E_1 \cup E_2$ where $(E_1,q_1)$ and $(E_2,q_2)$ are $H_{m,1}$-tails.  
We have $\hat{A} \cong \CC[[\bs_1, \bs_2, \bc_1, $ $\bc_2, n]]$ with 
$\bs_i = (s_{i,0}, \ldots, s_{i,2m-1})$ and $\bc_i = (c_{i,1}, \ldots, \bc_{i,m-1})$ for 
$i=1,2$.  The action of $\Aut(C)^{\circ} \cong \GG_m^2$ is given by
 $$
 s_{i,l} \mapsto t_i^{2l-4m-2} s_{i,l} \qquad c_{i,l} \mapsto t_i^{2l-1} c_{i,l} \qquad n \mapsto t_1 t_2 n
 $$
 Let $x \in \Def(C)$.  If $\lambda=(\lambda_1, \lambda_2): \GG_m \to \GG_m^2$ is a one-parameter 
 subgroup with $\lambda_i > 0$ such that $\lim_{t \to 0} \lambda(t) \cdot x$ exists then 
 $s_{i,0}(x) = \cdots = s_{i,2m-1}(x) = 0$.  Conversely, if $s_{i,0}(x) = \cdots = s_{i,2m-1}(x) = 0$ 
 for $j=1,2$, then $\lambda=(1,0)$ if $i=1$ or $\lambda=(0,1)$ if $i=2$ is a one-parameter 
 subgroup such that $\lim_{t \to 0} \lambda(t) \cdot x$ exists.  By 
 Proposition \ref{prop-hilbert-mumford}, 
 $V^{-} = V(\{s_{1,l} \}_{l=0}^{2m-1}) \cup V(\{s_{2,l}\}_{l=0}^{2m-1})$ 
 which corresponds in the deformation space to where one of the 
 two $A_k$-singularities is preserved.  

Let $x \in \Def(C)$.  If $\lambda=(\lambda_1, \lambda_2)  \to \GG_m^2$ is a one-parameter subgroup 
with $\lambda_1 + \lambda_2 < 0$ and $\lambda_j < 0$ such that $\lim_{t \to 0} \lambda(t) \cdot x$ exists, then 
$n(x) = c_{i,1}(x) = \cdots = c_{i,m-1}(x) = 0$.   Conversely, if 
$n = c_{i,1} = \cdots = c_{i,m-1} = 0$, then $\lambda=(-1,0)$ if $i=1$ or 
$\lambda=(0,-1)$ if $i=2$ is a one-parameter subgroup such that 
$\lim_{t \to 0} \lambda(t) \cdot x$ exists.  By Proposition \ref{prop-hilbert-mumford}, 
$V^{+} = V(n, \{c_{1,l} \}_{l=1}^{m-1}) \cup V(n, \{c_{2,l}\}_{l=1}^{m-1})$ which 
corresponds in the deformation space to where the node is preserved and one 
of the components is an $H_{m,1}$-tail.  

 \medskip \noindent

\item[Case I$''$]  This is Lemma \ref{lemma-VGIT-monomial}.

 \medskip \noindent

\item[Case II]
$C = K \cup E_1 \cup \cdots \cup E_r \cup E_{r+1} \cup \cdots \cup E_{r+s}$ where 
$E_1, \ldots, E_r$ are $H_{m,2}$-links intersecting $K$ at two nodes, and $E_{r+1}, \ldots, E_{r+s}$ 
are $H_{m,2}$-link intersecting $K$ at one node and terminating in a marked point.

By Lemma \ref{lemma-vgit-product}, it is enough to consider the case when either $r = 1, s=0$ or $r=0, s=1$.   The case of $r=1$ and $s=0$ is the example worked out in Lemma \ref{lemma-invariant-calculation}; the addition of marked points does not affect the calculation of Lemma \ref{lemma-invariant-calculation}.  If $r=1, s=0$, 
the action of $\Aut(C, \{p_i\}_{i=1}^n)^{\star}$ on $\Def(C, \{p_i\}_{i=1}^n)$ is precisely the action given in Section \ref{subsection-chains} restricted to the closed subscheme $V(n_{r+1}) = 0$.  
This case therefore follows from Lemmas \ref{lemma-vgit-closed} and \ref{lemma-invariant-calculation}.

\medskip \noindent
\item[Case II$'$] $(C,p_1,p_2)=(E_1 \cup \ldots \cup E_r,q_1,q_{2r})$,
where for $1 \leq j \leq r-1$, $E_j$ meets
$E_{j+1}$ in a node $q_{2j}=q_{2j+1}$ and for $1 \leq j \leq r$, $(E_j,q_{2j-1},q_{2j})$
is a monomial $H_{m,2}$-curve.
The action of $\Aut(C, \{p_i\}_{i=1}^n)^{\star}$ on $\Def(C, \{p_i\}_{i=1}^n)$ is 
the action given in Section \ref{subsection-chains} restricted to the closed 
subscheme $V(n_0,n_{r+1}) = 0$ so this case follows from 
Lemmas \ref{lemma-vgit-closed} and \ref{lemma-invariant-calculation}.

 \medskip \noindent

\item[Case II$''$]  This follows from an argument similar to the proof of Lemma \ref{lemma-invariant-calculation}.

\end{description}
  \end{proof}


\subsection{\'Etale local presentations by GIT chambers} \label{section-etale-GIT}

\begin{prop} \label{prop-algebraization} {\rm (\cite[Theorem 3]{alper_quotient})}
  Let $\cX$ be an algebraic stack of finite type over $\CC$.  Suppose $\cX$ is a quotient stack $[X/G]$, where $G$ is a connected algebraic group acting on a smooth and separated scheme $X$. If $x \in X(\CC)$ has linearly reductive stabilizer, there exists a locally closed $G_x$-invariant affine $W \hookarr X$ with $w \in W$ such that
$$[W/G_x] \to [X/G]$$
is affine and \'etale.
\end{prop}

\begin{cor} \label{corollary-algebraization}
  Let $\cX$ be an algebraic stack finite type over $\CC$.  Suppose $\cX$ is a quotient stack $[X/G]$, where $G$ is a connected algebraic group acting on a smooth and separated scheme $X$.  If $x \in X(\CC)$ has linearly reductive stabilizer, there is an affine scheme $W=\Spec A$ with an action by the stabilizer $G_x$, a closed $G_x$-invariant point $w \in W$ and a commutative diagram
\begin{equation}
\label{diagram-refined}
\xymatrix{
\hat{\Def}(x) \ar[r]^-j	& \cW = [W / G_x ] \ar[ld]_f \ar[rd]^g  \ar[dd] \\
\cX	&					& [\Def(x) / G_x] \ar[dd] \\
				& W//G_x \ar[rd]^{\bar g} \\
				&					& \Def(x) // G_x
} \end{equation}
such that
\begin{enumerate}
\item There is an isomorphism $\hat{\Def}(x) \to \Spf \hat{\oh}_{W,w}$ inducing $j$ and $f \circ j: \hat{\Def}(x) \to \cX$ is a miniversal deformation space of $x$,
\item $f$ is \'etale, affine, and stabilizer preserving at $x$ with $f(w) = x$, 
\item $[W/G_x] \to W//G_x$ and $[\Def(x)/G_x] \to \Def(x) //G_x$ are good moduli spaces,
\item $g$ is affine, \'etale, stabilizer preserving and saturated (in particular, $g$ maps closed points to closed points), and
\item the right parallelogram is cartesian.
\end{enumerate}
\end{cor}

\begin{proof}
Proposition \ref{prop-algebraization} implies the existence of an affine scheme $W = \Spec A$ with a $G_x$-action, a closed $G_x$-invariant point $w \in W$ and a morphism
$$f: [\Spec A / G_x] \to \cX$$
which is \'etale and affine.  Furthermore, $f(w) = x$ and  $f$ is stabilizer preserving at $w$.  The maximal ideal $\fm \subseteq A$ corresponding to $w$ is $G_x$-invariant which induces a $G_x$-representation $\fm/\fm^2$.  Since $G_x$ is linearly reductive, there exists a splitting $\fm/\fm^2 \hookarr \fm$ of the surjection $\fm \to \fm/\fm^2$.  The inclusion $\fm/\fm^2 \hookarr \fm \subseteq A$ induces a morphism on algebras $\Sym^* \fm / \fm^2 \to A$ which is $G_x$-equivariant which in turns gives a morphism 
$$g: [\Spec A / G_x] \to [\Def(x)  / G_x]$$
such that $g(x)$ is the origin, $g$ is \'etale at $x$ and stabilizer preserving at $x$.  Therefore we have a commutative diagram as in (\ref{diagram-refined}) where $(1)$ and $(2)$ satisfied.
 By \cite[Theorem 6.10]{alper_quotient}, we may shrink $\Spec A$ by choosing an affine saturated open of $\Spec A$ containing $w$ such that the parallelogram is cartesian and $\bar{g}$ is \'etale which establishes $(3)-(5)$.
 \end{proof}

\begin{prop} 
\label{theorem-etale-variation}
 Let $x \in \bar{\cM}_{g,n}(A_k)$ be a closed point for $k \le 4$.
There exists a morphism $f: \cW = [\Spec A /G_x] \to \bar{\cM}_{g,n}(A_k)$ where $G_x$ acts on an affine scheme $\Spec A$ fixing a point $w$ with $f(w) = x$, a morphism $g: \cW \to [\Def(x) / G_x]$ and a commutative diagram
\begin{equation} \label{diagram-etale-variation}
\xymatrix{
 \cW^{-} \ar[d] \ar@{^(->}[r]						& \cW= [\Spec A / G_x] \ar[d]^f & \cW^{+} \ar@{_(->}[l] \ar[d]\\
 \bar{\cM}_{g,n}(A_k^-) \ar@{^(->}[r] 						& \bar{\cM}_{g,n}(A_k) &  \bar{\cM}_{g,n}(A_k^+) \ar@{_(->}[l] 	
}
\end{equation}
such that 
\begin{enumerate}
\item $f$ is \'etale, affine and stabilizer preserving at $w$, 
\item the induced map $\hat{\Def}(x) \to  \bar{\cM}_{g,n}(A_k)$ is a miniversal deformation space,
\item the squares are cartesian,
\item there exist good moduli spaces $\cW \to Y = \Spec A^{G_x}$, $\cW^{-} \to Y^{-}$ and $\cW^{+} \to Y^{+}$ and a commutative diagram
$$\xymatrix{
\cW^{-} \ar@{^(->}[r] \ar[d]		&\cW = [\Spec A/G_x]  \ar[d]		& \cW^{+} \ar@{_(->}[l] \ar[d]\\
Y^{-} \ar[r]	& Y	& Y^{+} \ar[l]
}$$
with $Y^{-} \to Y$ and $Y^{+} \to Y$ projective; in particular, $\cW^-$ and $\cW^+$ are weakly proper over $Y$, 
\item the morphism $g$ is affine, \'etale, stabilizer preserving and saturated, and
\item $g^{-1}([V^+ / G_x]) = \cW^+$ and $g^{-1}([V^- / G_x]) = \cW^-$ where $V^-$ and $V^+$ are the open GIT chambers of the deformation space $\Def(x)$ given by the character $\chi: G_x^{\circ} \to \GG_m$ defined in Section \ref{subsection-general-case}. 
\end{enumerate}
\end{prop}

\begin{proof} 
Since the stack $\bar{\cM}_{g,n}(A_k)$ is smooth and parameterizes canonically polarized curves, we may apply Corollary \ref{corollary-algebraization} to find morphisms $f: [\Spec A/G_x] \to \bar{\cM}_{g,n}(A_k)$ and $g: [\Spec A/G_x] \to [\Def(x) / G_x]$ giving a diagram as in (\ref{diagram-refined}) such that \ref{corollary-algebraization}$(1)-$\ref{corollary-algebraization}$(5)$ are satisfied.  In particular, $(1),(2)$ and $(5)$ in this theorem are satisfied. 

\medskip \noindent
 Let $V^{-}, V^{+} \hookarr \Def(x)$ and $\fZ^-, \fZ^+ \hookarr \Spf \hat{A}$ be as in Diagram (\ref{diagram-formal}).   By Proposition \ref{theorem-local-GIT},   $f^{-1}(\bar{\cS}_{g,n}(A_k))$ and $g^{-1}([V^-/G_x])$ (resp.,  $f^{-1}(\bar{\cH}_{g,n}(A_k))$ and $g^{-1}([V^+/G_x])$) are closed substacks that agree in a formal neighborhood of $w$.  Therefore, they agree in a Zariski-open neighborhood.  We may restrict to a saturated $G_x$-invariant open affine neighborhood of $w$ in $\Spec A$ giving a diagram as in (\ref{diagram-refined}) still satisfying \ref{corollary-algebraization}$(1)-$\ref{corollary-algebraization}$(5)$ and such that
$$\begin{aligned}
f^{-1}(\bar{\cS}_{g,n}(A_k))&= g^{-1}([V^-/G_x]), \\
f^{-1}(\bar{\cH}_{g,n}(A_k)) &= g^{-1}([V^+/G_x]).
\end{aligned}$$
This establishes that there exists the desired diagram (\ref{diagram-etale-variation}) satisfying properties $(1)-(3)$.  Furthermore, by variation of GIT on $\Def(x)$ with respect to $\chi$ (see Section \ref{section-variation-GIT}), we have a commutative diagram
$$\xymatrix{
[V^{-}/G_x] \ar@{^(->}[r] \ar[d]		&[\Def(x) / G_x]  \ar[d]		& [V^{+} / G_x] \ar@{_(->}[l] \ar[d]\\
V^{-}//G_x \ar[r]	& \Def(x)//G_x	& V^{+}//G_x \ar[l]
}$$
where the vertical arrows are good moduli spaces and the morphisms $V^{-}//G_x \to \Def(x)//G_x$ and $V^{+}//G_x \to \Def(x)//G_x$ are projective.  Base changing this diagram by $\Spec A^{G_x} \to \Def(x) //G_x$ gives properties $(4)$ and $(6)$.
\end{proof}


\section{Weak properness of $\SM_{g,n}(A_{k}^-),\SM_{g,n}(A_{k}),\SM_{g,n}(A_{k}^+)$} \label{section-proof}
In this section, we prove our main theorem.

\begin{thm}\label{main-theorem}
For $k \in \{2,3,4\}$, the stacks $\SM_{g,n}(A_{k}^-)$, $\SM_{g,n}(A_{k})$ and
$\SM_{g,n}(A_{k}^+)$ are weakly proper.
\end{thm}

\noindent The proof is by induction. Since
$$\SM_{g,n}(A_2^-)=\SM_{g,n}$$ is weakly proper, it suffices to show that 
$$\SM_{g,n}(A_k^-) \text{ weakly proper }\implies \SM_{g,n}(A_k) \text{ weakly proper }$$
and
$$\SM_{g,n}(A_k) \text{ weakly proper }\implies \SM_{g,n}(A_k^+) \text{ weakly proper}.$$ 
These implications are proved in the following two sections as
Proposition~\ref{theorem-induction-step1} and
Proposition~\ref{theorem-induction-step2} respectively.
\subsection{$\SM_{g,n}(A_k^-)$ weakly proper $\implies$ $\SM_{g,n}(A_k)$ weakly proper}
Given a family $(\C^* \rightarrow \Delta^*, \{\sigma^*\}_{i=1}^{n})$ of $A_k$-stable curves, we must show that there exists a unique closed $A_k$-stable limit. The idea of the argument is as follows: Suppose first that the family $\C^* \rightarrow \Delta$ has no $A_k$-singularities in the geometric fibers, i.e. is actually $A_k^-$-stable. Using the weak properness of $\SM_{g,n}(A_k^-)$, one can fill in this family by a uniquely-determined closed $A_k^-$-stable limit (Lemmas \ref{L:IsotrivialLemma2} and \ref{L:lifting}). The essential point is that this $A_k^-$-stable limit actually determines the isomorphism class of any closed $A_k$-stable limit; roughly speaking, the unique $A_k$-stable limit is obtained by replacing all nodally-attached $H_k$-curves/chains by monomial $H_k$-curves/chains (Lemma \ref{L:CoreFamilies}). The case when the general fiber actually has $A_k$-singularities is then reduced to this case by a fairly standard normalization argument (Proposition \ref{prop-S-weakly-proper}).

\begin{lemma}\label{L:IsotrivialLemma2}  Suppose $(C, \pn) \rightsquigarrow (C_0, \pn)$ is an isotrivial specialization of $A_k$-stable curves satisfying:
\begin{enumerate}
\item  $(C, \pn)$ is an $A_k^-$-stable curve,  
\item $(C_0, \pn)$ is a closed point in $\bar{\cM}_{g,n}(A_k)$.
\end{enumerate}
Then there exists an isotrivial specialization $(C, \pn) \rightsquigarrow (C^-, \pn)$ in $\SM_{g,n}(A_k^-)$ satisfying:
\begin{enumerate}
\item $(C^-, \pn)$ is a closed point in  $\bar{\cM}_{g,n}(A_k^-)$.
\item There exists an isotrivial specialization $(C^-, \pn) \rightsquigarrow (C_0, \pn)$.
\end{enumerate}
\end{lemma}

\begin{proof}
We will prove the case when $k=2m$ is even (the case $k$ odd is essentially identical). Let $(\C \rightarrow \Delta, \sigman)$ be the family of $A_k$-stable curves witnessing the isotrivial specialization $(C, \pn) \rightsquigarrow (C_0, \pn)$. Let $\tilde{\C} \rightarrow \C$ be the normalization of $\C$ along the locus of attaching nodes of the canonical decomposition of the generic fiber as in Lemma \ref{L:LimitCanDecomp}, so that
$$
\tilde{\C}=\K \cup \E_1 \cup \ldots \cup \E_r
$$
where $\K \rightarrow \Delta$ picks out the core of the generic fiber, while $\E_1, \ldots, \E_r$ pick out the $H_{m,1}$-tails of the generic fiber. Note that $K_0$ is a closed $A_k$-stable curve with no $A_k$-singularities, hence a closed $A_k^-$-stable curve.

We construct a new isotrivial specialization $(C, \pn)
\rightsquigarrow (C^-,\pn)$ simply by gluing the family $\K$ to $r$ trivial
families $\E_1', \ldots, \E_r'$ whose fibers are all isomorphic to
the geometric general fiber of $\E_1, \ldots, \E_r$
respectively. The special fiber of this new family is a closed $A_k^-$-stable curve by Lemma \ref{L:AkMinusClosed} and it is obvious from the construction that it specializes to $(C_0, \pn)$.
\end{proof}

\begin{lemma}\label{L:lifting}
Let $j:\Delta \rightarrow \SM_{g,n}(A_k)$  be any map such that $j(\eta) \in \SM_{g,n}(A_k^-)$ and $j(0)$ is closed. Then there exists a lift $j^-:\Delta \rightarrow \SM_{g,n}(A_k^-)$ such that 
\begin{itemize}
\item $j|_{\eta}=j^-|_{\eta}$,
\item $j^{-}(0)$ is closed,
\item  $j(0) \in \overline{j^{-}(0)}$.
\end{itemize}
\end{lemma}
\begin{proof}
To construct the lift $j^-$, apply Proposition \ref{theorem-etale-variation} with $x=j(0)$. We obtain an affine, \'etale morphism $f: \cW \to \bar{\cM}_{g,n}(A_k)$ inducing the following Cartesian diagram:
$$
\xymatrix{
	&  \cW^{-} \ar[d] \ar@{^(->}[r]	\ar[d]^{f^-} 					& \cW \ar[d]^f \\
\Delta^* \ar[r] \ar[d] &  \bar{\cM}_{g,n}(A_k^-) \ar@{^(->}[r] 						& \bar{\cM}_{g,n}(A_k) \\
\Delta \ar@/_2pc/[rru]^{j}universally
}
$$

\noindent
Since $f$ is \'etale, we may lift $j$ to a morphism $\Delta \rightarrow \cW$. Since the square is Cartesian, we may then lift the map $\Delta^* \rightarrow \SM_{g,n}(A_k^*)$ to a map $\Delta^* \rightarrow \cW^-$. Now we have a commutative diagram: 
$$
\xymatrix{
\Delta^* \ar[r] \ar[d]  &	\cW^{-} \ar[d] \ar@{^(->}[r]		& \cW \ar[d] \\
\Delta \ar[rru]			&	W^{-} \ar[r]				& W
}
$$
Now, since the morphism $W^- \rightarrow W$ is projective, the composition $\Delta \rightarrow W$ may be lifted to a morphism $\Delta \rightarrow W^-$. Then, since $\cW^- \rightarrow W^-$ is universally closed, we may lift to a morphism $\Delta \rightarrow \cW^-$. Finally, composing with $f^-$, we obtain a map $j':\Delta \rightarrow \SM_{g,n}(A_k^-)$. The point $j'(0)$ necessarily admits an isotrivial specialization to $j(0)$, the only problem is that $j'(0)$ may not be closed. However, by applying lemma 
\ref{L:IsotrivialLemma2} we see that $j'(0)$ admits an isotrivial specialization to a closed point in $\SM_{g,n}(A_k^-)$ which still specializes isotrivially to $j(0)$. Thus, using the valuative criterion for algebraic stacks, there exists a map $j^-:\Delta \rightarrow \SM_{g,n}(A_k^-)$ with the desired properties.
\end{proof}

\begin{lemma}\label{L:CoreFamilies}
Suppose $\C \rightarrow \Delta, \sigman$ is an isotrivial specialization in $\SM_{g,n}(A_k)$ satisfying
\begin{enumerate}
\item the generic fiber is a closed point of $\SM_{g,n}(A_k^-)$,
\item the special fiber is a closed point of $\SM_{g,n}(A_k)$.
\end{enumerate}
Then the core of $C_{\bar{\eta}}$ is isomorphic to the core of $C_0$.
\end{lemma}
\begin{proof}
We will prove the lemma in the case $k=2m$ is even (the case when $k$ odd is essentially identical.) Let $C_{\bar{\eta}}=K_{\overline{\eta}} \cup E_1 \cup \ldots \cup E_r$ be the canonical decomposition of the geometric generic fiber. Let $\tilde{\C} \rightarrow \C$ denote the normalization of $\C$ along the locus of attaching nodes of the canonical decomposition as in Lemma \ref{L:LimitCanDecomp}, so we have
$$
\tilde{\C}=\K \cup \E_1 \cup \ldots \cup \E_r,
$$
where $\K$ and $\E_i$ are isotrivial specializations of $A_k$-stable curves with generic fiber $K_{\overline{\eta}}$ and $E_i$ respectively.
Since $K$ is a maximally degenerate $A_k^-$-stable curve with no nodally attached $H_{m,1}$-tails ($k=2m$) or $H_{m,2}$-chains ($k=2m+1$), $K$ is a maximally degenerate $A_k$-stable curve by Lemma \ref{L:AlmostClosed}. It follows that the isotrivial specialization $\K$ is trivial, i.e. $K_0 \simeq K_{\overline{\eta}}$. To complete the proof of the lemma, we only need to show that $K_0$ is the core of $C_0$.

On the one hand, since $\bar{\H}_{m,1} \subset \SM_{m,1}(A_k)$ is closed by Proposition \ref{P:Openness}, it is clear that the limits $(\E_i)_0$ are all $H_{m,1}$-curves. Thus, we only need to see that $K_0$ contains no nodally attached $H_{m,1}$-curves. By our characterization of closed points of $\SM_{g}(A_k)$ (Proposition \ref{P:ClosedPoints}), any nodally attached $H_{m,1}$-curve in $K_0$ must be monomial and hence contains an $A_k$-singularity. But since $K_0 \simeq K_{\overline{\eta}}$ is $A_k^-$-stable, this is clearly impossible.
\end{proof}

\begin{prop}
\label{theorem-induction-step1}
If $\bar{\cM}_{g,n}(A_k^-)$ is weakly proper, then $\bar{\cM}_{g,n}(A_k)$ is weakly proper.
\end{prop}

\begin{proof}
{\it Existence of $A_{k}$-stable limits:}
Let $\C^* \rightarrow \Delta^*$ be a family of $A_k$-stable
curves.  If $\C^* \to \Delta^*$ is a family of $A_k^-$-stable
curves, then since $\SM_{g,n}(A_k^-)$ is universally closed by
hypothesis, there exists a limit (after a base change). Otherwise
$\C^* \to \Delta^*$ is a family of curves in
$\bar{\cS}_{g,n}(A_k)$, which is universally closed by
Proposition~\ref{prop-S-weakly-proper}, so again there exists a
limit after a base change.
 
 \medskip
 \noindent
{\it Uniqueness of closed $A_k$-stable limits: }
Suppose we have a diagram 
$$
\xymatrix{
 \Delta^* \ar[r] \ar[d]									&  \bar{\cM}_{g,n}(A_k) \ar[d] \\
 \Delta \ar[r] \ar@<.5ex>[ur]^{h_1} \ar@<-.5ex>[ur]_{h_2}		& \Spec \CC
}
$$
with two lifts $h_1, h_2: \Delta \to  \bar{\cM}_{g,n}(A_k)$ such that $h_1(0), h_2(0) \in |\bar{\cM}_{g,n}(A_k)  \times_\CC \Delta|$ are closed. If $h_1(\eta)=h_2(\eta)$ lies in $\bar{\cS}_{g,n}(A_k)$, then $h_1(0) = h_2(0)$ by Proposition~\ref{prop-S-weakly-proper}. Otherwise $h_1(\eta)=h_2(\eta)$ lies in $\SM_{g,n}(A_k^-)$.

In this case, let $\C_1 \rightarrow \Delta$ and $\C_2 \rightarrow \Delta$ be the
two families induced by $h_1$ and $h_2$. We must show that $C_1
\simeq C_2$. Using Lemma \ref{L:lifting}, we can lift $h_1$, $h_2$ to maps $h_1^-, h_2^-: \Delta \rightarrow\SM_{g,n}(A_k^-)$ such that the special fibers $C_1^-, C_2^-$ of the associated families $\C_1^- \rightarrow \Delta$ and $\C_2^- \rightarrow \Delta$ satisfy:
\begin{enumerate}
\item $[C_1] \in \overline{[C_1^-]}$, $[C_2] \in \overline{[C_2^-]}$;
\item $[C_1^-]$, $[C_2^-]$ are closed points of $\SM_{g,n}(A_k^-)$.
\end{enumerate}

Now by the weak properness of $\SM_{g,n}(A_k^-)$, we know that $C_1^- \simeq C_2^-$. In particular, their cores are isomorphic. By Lemma \ref{L:CoreFamilies}, the isotrivial specialization $C_i^- \rightsquigarrow C_i$ induces an isomorphism of cores for $i=1,2$. Thus, the core of $C_1$ is isomorphic to the core of $C_2$. By Corollary \ref{C:CoreIso}, it follows that $C_1$ is isomorphic to $C_2$ as desired.
\end{proof}

Finally, it remains to prove the existence and uniqueness of limits in the case when the general fiber of $\C^* \rightarrow \Delta^*$ has $A_k$-singularities, or, equivalently, the weak properness of $\bar{\cS}_{g,n}(A_k):=\SM_{g,n}(A_k) \backslash \SM_{g,n}(A_k^-) $. As in the case of nodal curves, the result is reduced to the case where the general fiber has no $A_k$-singularities by considering the pointed normalization along the locus of $A_k$-singularities. However, there is one additional complication: since the crimping data of an $A_k$-singularity can vary in families, we must understand how to take the limits of families with varying crimping data.

\begin{prop} \label{prop-S-weakly-proper}
$\bar{\cS}_{g,n}(A_k)$ is weakly proper for $k \in \{2,3,4\}$.
\end{prop}

\begin{proof} {\it Existence of limits:} Assume first that $k=2m$ is even. Given a family ($\cC^*
\arr \Delta^*,\sigmanstar)$ in $\bar{\cS}_{g,n}(A_k)$, we may decompose $\cC^*$ as:
$$
\K^* \cup \E_1^* \cup \ldots \cup \E_r^*, 
$$
where $\K^*$ picks out the core of the geometric fiber, and $\E_1^*, \ldots, \E_r^*$ comprise the nodally attached $H_{m,1}$-tails of the geometric fiber. Since $\overline{\H}_{m,1}$ is weakly proper (Corollary \ref{cor-H-weakly-proper}), we may complete the families $\E_i^*$ to families $\E_i$ such that the special fiber is a monomial $H_{m,1}$-tail. Thus, it suffices to to complete the family $\K^*$.
 
After a finite base change, we may assume there exist sections $\{\tau_i^*\}_{i=1}^m$ picking out the $A_k$-singularities of the geometric fiber, and let $\tilde{\K}^* \rightarrow \K^*$ be the pointed normalization of $\K^*$ along these sections. Note that $(\tilde{\K}^* \rightarrow \Delta, \{\tilde{\tau}^*_i\}_{i=1}^m)$ is now a family of $A_k$-stable curves with no $A_k$-singularities, i.e. the family is in fact $A_k^-$-stable. Thus, by Proposition \ref{theorem-induction-step1} and induction on the genus, we may complete $(\tilde{\K}^*,\{\tilde{\tau}^*_i\}_{i=1}^m)$ to a family $(\tilde{\K},\{\tilde{\tau}_i\}_{i=1}^m)$. Now we must ``recrimp" the sections $\{\tilde{\tau}_i\}_{i=1}^m$ back to $A_k$-singularities. We will show that this is always possible, possibly after making some blow-ups and blow-downs in the special fiber which will have the effect of introducing monomial $H_{m,1}$-tails at a subset of the points $\{\tilde{\tau}_i(0)\}_{i=1}^{m}$. Since the special fiber of $(\tilde{\K},\{\tilde{\tau}_i\}_{i=1}^m)$ is a pointed $A_k$-stable curve, the curve obtained by attaching monomial $H_{m,1}$-tails at any subset of the marked points will still be $A_k$-stable, so this completes the proof of existence of limits.

To prove the crimping statement, fix an isomorphism of the
$(2m-1)$\textsuperscript{th}-order neighborhood of $\tilde{\tau}_i$ in
$\tilde{\K}$ with $\Spec R[s]/(s)^{2m}$. Using the induced isomorphism of
the $(2m-1)$\textsuperscript{th}-order neighborhood of $\sigma^*$
in $\tilde{\K}^*$ with $\Spec K[s]/(s)^{2m}$, there corresponds to $\tilde{\K}^*
\arr \cC^* \arr \Delta^*$ a map from $\Delta^*$ to the
Grassmannian $\GG(m-1,2m-1)$. Fill this in to a map $\Delta \arr
\GG(m-1,2m-1)$ and consider the corresponding map $\tilde{\K} \arr \K$
over $\Delta$. The limit of the $A_k$-singularity of $\K^*$ in
$K_0$ may not be an $A_k$-singularity but we will show that this
can be rectified by blowing up $\tilde{\K}$ at $\sigma(0)$ a
suitable number of times.

The $K$-subalgebra of $K[s]/(s)^{2m}$ corresponding to the map
$\Delta^* \arr \GG(m-1,2m-1)$ can be written in the form
\[
K[(s + c_1 s^2 + \ldots + c_{m-1} s^{2m-2})^2,
s^{2m},\ldots,s^{4m-1}],
\]
where $c_1,\ldots,c_{m-1}\in K$. (Cf. the proof of
Proposition~\ref{first-order-action-even} and see
\cite[Example~1.111]{wyck} for details.) Let $t$ be a uniformizing
parameter of $R$. Let $\tilde{\K}'$ be obtained by blowing $\tilde{\K}$ up $b$
times at $\sigma(0)$. This has the effect of making the coordinate
change $s \mapsto t^b s$, so the $K$-subalgebra of $K[s]/(s)^{2m}$ corresponding to the map $(\tilde{\K}')^*
\arr \K^*$ over $\Delta^*$ is 
\[
K[(s + t^b c_1 s^2 + \ldots +
t^{(2m-3)b} c_{m-1} s^{2m-2})^2,s^{2m},\ldots,s^{4m-1}].
\]
For $1 \leq i \leq m-1$, write $c_i = u_i t_i^{b_i}$, where $u_i$
is a unit. Choose $b \geq 0$ minimal such that $(2i-1)b+b_i \geq
0$ for each $i$, so that $t^{(2i-1)b}c_i \in R$ for each $i$. For this choice of $b$, the
$R$-subalgebra of $R[s]/(s)^{2m}$ corresponding to the natural filling-in $\Delta \arr
\GG(m-1,2m-1)$ is
\[
R[(s + t^b c_1 s^2 + \ldots + t^{(2m-3)b} c_{m-1}
s^{2m-2})^2,s^{2m},\ldots,s^{4m-1}].
\]
Forming the corresponding map $\tilde{\K}' \arr \K$ over $\Delta$, the
limit in $K_0$ of the $A_k$-singularity of $\cC^*$ is evidently an
$A_k$-singularity, as required. Blowing down the $b-1$
nodally-attached $\PP^1$'s of $K_0$ in $\K$, we obtain the desired family.

The entire proof of existence of limits in the case that $k=2m+1$ is odd is similar. Here, we just discuss the analogue of the crimping statement. The relevant $K$-subalgebra of $K[s_1]/(s_1)^{m+1} \oplus K[s_2]/(s_2)^{m+1}$ is:
\[
K[s_1+c_1 s_1^2+\ldots c_{m-1} s_1^m
\oplus s_2, s_1^{m+1} \oplus 0,\ldots,s_1^{2m+1}\oplus 0, 0 \oplus
s_2^{m+1},\ldots,0 \oplus s_2^{2m+1}]
\]
One can arrange
that the limit $R$-subalgebra of $R[s_1]/(s_1)^{m+1} \oplus
R[s_2]/(s_2)^{m+1}$ is of the required form by blowing $\tilde{\K}$ up
$b$ times at $\sigma_1(0)$, where $b$ is chosen
minimal such that $ib+b_i \geq 0$ for $1 \leq i \leq m-1$, where
$c_i = u_i t^{b_i}$ as before.

\medskip
\noindent {\it Uniqueness of closed limits:} Let $(\cC \arr
\Delta,\sigman)$ be any completion of a given family $(\cC^* \arr
\Delta^*,\{\sigma^*\}_{i=1}^{n})$ such that the central fiber is a closed point of $\S_{g,n}(A_k)$. We will show that the isomorphism class of the central fiber is uniquely determined. We will prove the statement in the case $k=2m$ is even (the case where $k$ is odd is essentially identical).


As in Lemma \ref{L:LimitCanDecomp}, we may consider the normalization $\tilde{\C} \rightarrow \C$ along the locus of attaching nodes of the core decomposition of the generic fiber, and $\tilde{\C}$ decomposes as
$$
\K \cup \E_1 \cup \ldots \cup \E_r, 
$$
where $\K$ picks out the core of the geometric generic fiber, and $\E_1, \ldots, \E_r$ comprise the nodally attached $H_{m,1}$-tails of the geometric generic fiber. Now, since $\overline{\H}_{m,1}$ is weakly proper, the special fibers $E_i \subset \E_i$ are certainly unique (they are the unique monomial $H_{m,1}$-tails). It remains to show that the special fiber $K \subset \K$ is uniquely determined.

If $\K^*$ has no $A_k$-singularities, then the generic fiber is $A_k^-$-stable and the limit is uniquely determined by Proposition \ref{theorem-induction-step1}. Thus, we may assume that $\K^*$ contains at least one $A_k$-singularity, and we let $\{\tau_i\}_{i=1}^{m}$ be sections of $\K$ picking out the $A_k$-singularities of the generic fiber. Note that the limits $\tau_i(0)$ are necessarily $A_k$-sin\-gularities. Furthermore, by the characterization of closed points in Proposition \ref{P:ClosedPoints}, each of the limits $\tau_i(0)$ sits on a monomial $H_{m,1}$-tail. Now let $(\tilde{\K}, \tilde{\tau}_i)$ be the pointed normalization along $\{\tau_i\}_{i=1}^{m}$. Evidently, the general fiber of $(\tilde{\K}, \tilde{\tau}_i)$ is $A_k^-$-stable and the limit is $A_k^-$-stable, save for the existence of $m$ semistable $\P^1$'s containing each of the $m$ points $\tau_i(0)$.  Let $\tilde{\K} \rightarrow \K'$ be the blowdown of these $m$ $\P^1$'s, so that the special fiber $K'$ is $A_k^-$-stable. By the weak properness of $\SM_{g,n}(A_k^-)$, this limit $K'$ is uniquely determined. Since the limit $K$ is obtained simply by adjoining $m$ monomial $H_{m,1}$-tails to $K'$, we conclude that the isomorphism class of $K$ is uniquely determined as desired.
\end{proof}

\medskip\medskip

\subsection{$\bar{\cM}_{g,n}(A_k)$ weakly proper $\implies \bar{\cM}_{g,n}(A_k^+)$ weakly proper}

The following theorem is a formal consequence of the \'etale
local description of $\bar{\cM}_{g,n}(A_k)$ at a closed point
given in Proposition \ref{theorem-etale-variation}.

\begin{prop}
\label{theorem-induction-step2}
If $\bar{\cM}_{g,n}(A_k)$ is weakly proper, then $\bar{\cM}_{g,n}(A_k^+)$ is weakly proper.
\end{prop}

\begin{proof}
{\it Existence of $A_{k}^+$-stable limits:} Consider a diagram
\begin{equation}
\label{diagram-universally-closed}
\xymatrix{
\Delta^* \ar[r] \ar[d]	& \bar{\cM}_{g,n}(A_k^+) \ar[d] \\
 \Delta \ar[r]		& \Spec \CC
}
\end{equation}
Since $\bar{\cM}_{g,n}(A_k)$ is universally closed by hypothesis, after a base change there is an extension $h: \Delta \to \bar{\cM}_{g,n}(A_k)$ giving a commutative diagram
$$\xymatrix{
\Delta^* \ar[r] \ar[d]	& \bar{\cM}_{g,n}(A_k^+) \ar[d] \ar@{^(->}[r]	& \bar{\cM}_{g,n}(A_k) \\
 \Delta \ar[r] \ar@{-->}[rru]^h		& \Spec \CC
}$$
By Proposition \ref{theorem-etale-variation}, there is an affine scheme $\Spec A$ with an action of $G_x$ and an \'etale morphism $\cW = [\Spec A/G_x] \to  \bar{\cM}_{g,n}(A_k)$ whose image contains $h(0)$ and such that the diagram
\begin{equation}
\label{diagram-cartesian}
\xymatrix{
 \cW^{+} \ar[d] \ar@{^(->}[r]						& \cW= [\Spec A / G_x] \ar[d] \\
 \bar{\cM}_{g,n}(A_k^+) \ar@{^(->}[r] 						& \bar{\cM}_{g,n}(A_k) 
}
\end{equation}
is cartesian; furthermore, there are good moduli spaces $\cW \to W$ and $\cW^+ \to W^+$ with $W^+ \to W$ proper.  Since $\cW \to \bar{\cM}_{g,n}(A_k)$ is \'etale, after a base change there is an extension
$$\xymatrix{
								& \cW \ar[d] \\
\Delta \ar[r] \ar@{-->}[ur]				& \bar{\cM}_{g,n}(A_k) 
}$$
Since Diagram \ref{diagram-cartesian} is cartesian, there is a unique morphism $\Delta^* \to \cW^{+}$ extending the given morphisms $\Delta^* \to \bar{\cM}_{g,n}(A_k^+)$ and $\Delta^* \to \Delta \to \cW$.
We have a commutative diagram
$$\xymatrix{
\Delta^* \ar[r] \ar[d]  &	\cW^{+} \ar[d] \ar@{^(->}[r]		& \cW \ar[d] \\
\Delta \ar[rru]			&	W^{+} \ar[r]				& W
}$$
 Since $\cW^{+} \to W^{+}$ is universally closed (\cite[Theorem 4.16(ii)]{alper_good_arxiv}, the composition $\cW^{+} \to W$ is universally closed.  Therefore, after a base change there is a morphism $\Delta \to \cW^{+}$ which after composing with $\cW^{+} \to \bar{\cM}_{g,n}(A_k^+)$ gives the desired extension of Diagram \ref{diagram-universally-closed}.

\medskip \noindent
{\it Uniqueness of closed $A_k^+$-stable limits: }
Suppose we have a diagram 
$$
\xymatrix{
 \Delta^* \ar[r] \ar[d]									&  \bar{\cM}_{g,n}(A_k^+) \ar[d] \\
 \Delta \ar[r] \ar@<.5ex>[ur]^{h_1} \ar@<-.5ex>[ur]_{h_2}		& \Spec \CC
}
$$
with two lifts $h_1, h_2: \Delta \to  \bar{\cM}_{g,n}(A_k^+)$ such
that $h_1(0), h_2(0)$ are closed points in
$|\bar{\cM}_{g,n}(A_k^+)  \times_\CC \Delta|$.  Since
$\bar{\cM}_{g,n}(A_k)$ is weakly separated by hypothesis, after
possibly making a base change there is a lift $h_0: \Delta \to \bar{\cM}_{g,n}(A_k)$ such that 
$h_0(0) \in |\bar{\cM}_{g,n}(A_k) \times_\CC \Delta|$ is closed and there are specializations 
$h_1(0) \rightsquigarrow h_0(0)$ and $h_2(0) \rightsquigarrow h_0(0)$ in $ |\bar{\cM}_{g,n}(A_k) \times_\CC \Delta|$.  By Proposition \ref{theorem-etale-variation}, there exists a morphism $f: \cW \to \bar{\cM}_{g,n}(A_k)$ with $h_0(0) = f(w_0)$ for $w_0 \in |\cW|$ which induces a cartesian diagram as in Diagram \ref{diagram-cartesian}.  We have a commutative diagram 
$$
\xymatrix{
	&  \cW^{+} \ar[d] \ar@{^(->}[r]						& \cW \ar[d]^f \\
\Delta^* \ar[r] \ar[d] &  \bar{\cM}_{g,n}(A_k^+) \ar@{^(->}[r] 						& \bar{\cM}_{g,n}(A_k) \\
\Delta \ar@<.5ex>[ur]^{h_1} \ar@<-.5ex>[ur]_{h_2} \ar@/_2pc/[rru]^{h_0} 
}
$$

\noindent
Since $f$ is \'etale and $f^{-1}(\bar{\cM}_{g,n}(A_k^+)) = \cW^+$,
there exist unique points $w_1, w_2 \in |\cW^+|$ and
specializations $w_i \rightsquigarrow w_0$ over $h_i(0)
\rightsquigarrow h_0(0)$.  Let $\xi = h_0(\eta) \in \bar{\cM}_{g,n}(A_k)$ be the 
image of the generic point.
There exist $\chi \in |\cW|$ over $\xi$ and specializations
$\chi \rightsquigarrow w_i$ for $i=0,1,2$.
The specializations $\chi \rightsquigarrow w_i$ for $i=1,2$ can be realized, after
a base change, by morphisms $\tilde h_i: \Delta \to \cW^+$ which
lift $h_i: \Delta \to \bar{\cM}_{g,n}(A_k^+)$ such that $\tilde{h}_1|_{\Delta^*} \cong \tilde{h}_2|_{\Delta^*}$.  Note that $w_1$
and $w_2$ are necessarily closed in $|\cW^+ \times_\CC \Delta|$ as
$h_1(0)$ and $h_2(0)$ are closed in $|\bar{\cM}_{g,n}(A_k^+)
\times_\CC \Delta|$.  Since $\cW^+$ is weakly separated, it
follows that $w_1 = w_2$ so $h_1(0) = h_2(0)$ as required.
\end{proof}

\bibliography{references}{}
\bibliographystyle{amsalpha}

\end{document}